\documentclass[11pt, letter]{amsart}

\usepackage{style/package,style/command,style/bibstyle}
\usepackage{format}

\begin{document}


\title[Thermodynamic formalism for subsystems of expanding Thurston maps II]
{
    Thermodynamic formalism for subsystems of \\expanding Thurston maps II
}

\author{Zhiqiang~Li \and Xianghui~Shi}

\thanks{Z.~Li and X.~Shi were partially supported by NSFC Nos.~12101017, 12090010, 12090015, and BJNSF No.~1214021.}
    
\address{Zhiqiang~Li, School of Mathematical Sciences \& Beijing International Center for Mathematical Research, Peking University, Beijing 100871, China}
\email{zli@math.pku.edu.cn}

\address{Xianghui~Shi, School of Mathematical Sciences \& Beijing International Center for Mathematical Research, Peking University, Beijing 100871, China}
\email{sxh1999@pku.edu.cn}




\begin{abstract}
Expanding Thurston maps were introduced by M.~Bonk and D.~Meyer with motivation from complex dynamics and Cannon's conjecture from geometric group theory via Sullivan's dictionary. 
In this paper, we study subsystems of expanding Thurston maps motivated via Sullivan's dictionary as analogs of some subgroups of Kleinian groups.
We prove the uniqueness and various ergodic properties of the equilibrium states for strongly primitive subsystems and real-valued H\"older continuous potentials, and establish the equidistribution of preimages of subsystems with respect to the equilibrium states.
Here, the sphere $S^{2}$ is equipped with a natural metric, called a visual metric, introduced by M.~Bonk and D.~Meyer.
As a result, for strongly primitive subsystems of expanding Thurston maps without periodic critical point, we obtain a level-2 large deviation principle for Birkhoff averages and iterated preimages.
\end{abstract} 

\maketitle
\tableofcontents


\section{Introduction}
\label{sec:Introduction}



A Thurston map is a (non-homeomorphic) branched covering map on a topological $2$-sphere $S^{2}$ such that each of its critical points has a finite orbit (postcritically-finite).
The most important examples are given by postcritically-finite rational maps on the Riemann sphere $\ccx$.
While Thurston maps are purely topological objects, a deep theorem due to W.~P.~Thurston characterizes Thurston maps that are, in a suitable sense, described in the language of topology and combinatorics, equivalent to postcritically-finite rational maps (see \cite{douady1993proof}). 
This suggests that for the relevant rational maps, an explicit analytic expression is not so important, but rather a geometric-combinatorial description. 
This viewpoint is natural and fruitful for considering more general dynamics that are not necessarily conformal.

In the early 1980s, D.~P.~Sullivan introduced a “dictionary" that is now known as \emph{Sullivan's dictionary}, which connects two branches of conformal dynamics, iterations of rational maps and actions of Kleinian groups.
Under Sullivan’s dictionary, the counterpart to Thurston’s theorem in geometric group theory is Cannon’s Conjecture \cite{cannon1994combinatorial}. 
An equivalent formulation of Cannon's Conjecture, viewed from a quasisymmetric uniformization perspective (\cite[Conjecture~5.2]{bonk2006quasiconformal}), predicts that if the boundary at infinity $\partial_{\infty} G$ of a Gromov hyperbolic group $G$ is homeomorphic to $S^2$, then $\partial_{\infty} G$ equipped with a visual metric is quasisymmetrically equivalent to $\ccx$. 

Inspired by Sullivan’s dictionary and their interest in Cannon’s Conjecture, M.~Bonk and D.~Meyer \cite{bonk2010expanding,bonk2017expanding}, as well as P.~Ha{\"i}ssinsky and K.~M.~Pilgrim \cite{haissinsky2009coarse}, studied a subclass of Thurston maps, called \emph{expanding Thurston maps}, by imposing some additional condition of expansion.
These maps are characterized by a contraction property for inverse images (see Subsection~\ref{sub:Thurston_maps} for the precise definition). 
In particular, a postcritically-finite rational map on $\ccx$ is expanding if and only if its Julia set is equal to $\ccx$.
For an expanding Thurston map on $S^{2}$, we can equip $S^2$ with a natural class of metrics $d$, called \emph{visual metrics}, that are quasisymmetrically equivalent to each other and are constructed in a similar way as the visual metrics on the boundary $\partial_{\infty} G$ of a Gromov hyperbolic group $G$ (see \cite[Chapter~8]{bonk2017expanding} for details, and see \cite{haissinsky2009coarse} for a related construction). 
In the language above, the following theorem was obtained in \cite{bonk2010expanding,bonk2017expanding,haissinsky2009coarse}, which can be seen as an analog of Cannon’s conjecture for expanding Thurston maps.

\begin{theorem*}[M.~Bonk \& D.~Meyer \cite{bonk2010expanding,bonk2017expanding}; P.~Ha{\"i}ssinky \& K.~M.~Pilgrim \cite{haissinsky2009coarse}]
    Let $f \colon S^2 \mapping S^2$ be an expanding Thurston map with no periodic critical points and $d$ be a visual metric for $f$. 
    Then $f$ is topologically conjugate to a rational map if and only $(S^2, d)$ is quasisymmetrically equivalent to $\ccx$. 
\end{theorem*}

The dynamical systems that we study in this paper are called \emph{subsystems} of expanding Thurston maps (see Subsections~\ref{sub:Subsystems of expanding Thurston maps} for a precise definition), inspired by a translation of the notion of subgroups from geometric group theory via Sullivan’s dictionary.
To reveal the connections between subsystems and subgroups, we first quickly review some backgrounds in Gromov hyperbolic groups and recall the notion of \emph{tile graphs} for expanding Thurston maps (see \cite[Chapters~4 and~10]{bonk2017expanding} for details).

Let $G$ be a Gromov hyperbolic group and $S$ a finite generating set of $G$. 
Then the \emph{Cayley graph} $\mathcal{G}(G, S)$ of $G$ is Gromov hyperbolic with respect to the word-metric.
The boundary at infinity of $G$ is defined as $\partial_{\infty} G \define \partial_{\infty} \mathcal{G}(G, S)$, which is well-defined since a change of the generating set induces a quasi-isometry of the Cayley graphs.

Let $f \colon S^2 \mapping S^2$ be an expanding Thurston map and $\mathcal{C} \subseteq S^2$ a Jordan curve with $\post{f} \subseteq S^2$.
An associated \emph{tile graph} $\mathcal{G}(f, \mathcal{C})$ is defined as follows.
Its vertices are given by the tiles in the cell decompositions $\mathcal{D}^{n}(f, \mathcal{C})$ on all levels $n \in \n_{0}$.
We consider $X^{-1} \define S^2$ as a tile of level $-1$ and add it as a vertex.
One joins two vertices by an edge if the corresponding tiles intersect and have levels differing by at most $1$ (see \cite[Chapter~10]{bonk2017expanding} for details).  
Then the graph $\mathcal{G}(f, \mathcal{C})$ is Gromov hyperbolic and its boundary at infinity $\partial_{\infty}\mathcal{G} \define \partial_{\infty}\mathcal{G}(f, \mathcal{C})$ is well-defined and can be naturally identified with $S^2$.
Moreover, under this identification, a metric on $\partial_{\infty}\mathcal{G} \cong S^2$ is visual in the sense of Gromov hyperbolic spaces if and only if it is visual in the sense of expanding Thurston maps.

From the point of view of tile graphs and Cayley graphs, roughly speaking, for expanding Thurston maps, $1$-tiles together with the maps restricted to those tiles play the role of generators for Gromov hyperbolic groups. 
For example, one can construct the original expanding Thurston map $f$ from all its $1$-tiles and the maps restricted to those tiles. 
If we start with all $n$-tiles for some $n \in \n$, then we get an iteration $f^{n}$ of $f$, which corresponds to a finite index subgroup of the original group in the group setting.
Inspired by this similarity, it is natural to investigate more general cases, for example, a map generated by some $1$-tiles, which leads to our study of subsystems.
Moreover, the notion of tile graphs can be easily generalized to subsystems, and the similar identifications hold with $S^2$ generalized to the tile maximal invariant sets associated with subsystems.

Under Sullivan's dictionary, an expanding Thurston map corresponds to a Gromov hyperbolic group whose boundary at infinity is $S^2$.
In this sense, a subsystem corresponds to a Gromov hyperbolic group whose boundary at infinity is a subset of $S^{2}$.
In particular, for Gromov hyperbolic groups whose boundary at infinity is a \sierpinski carpet, there is an analog of Cannon's conjecture---the Kapovich--Kleiner conjecture. 
It predicts that these groups arise from some standard situation in hyperbolic geometry.
Similar to Cannon's conjecture, one can reformulate the Kapovich--Kleiner conjecture in an equivalent way as a question related to quasisymmetric uniformization.
For subsystems, it is easy to find examples where the tile maximal invariant set is homeomorphic to the standard \sierpinski carpet (see Subsection~\ref{sub:Subsystems of expanding Thurston maps} for examples of subsystems).
In this case, an analog of the Kapovich--Kleiner conjecture for subsystems is established in \cite{bonk2024dynamical}. 

In this paper, we study the dynamics of subsystems of expanding Thurston maps from the point of view of ergodic theory.
Ergodic theory has been an essential tool in the study of dynamical systems. 
The investigation of the existence and uniqueness of invariant measures and their properties has been a central part of ergodic theory. 
However, a dynamical system may possess a large class of invariant measures, some of which may be more interesting than others. 
It is, therefore, crucial to examine the relevant invariant measures.

The \emph{thermodynamic formalism} serves as a viable mechanism for generating invariant measures endowed with desirable properties.
More precisely, for a continuous transformation on a compact metric space, we can consider the \emph{topological pressure} as a weighted version of the \emph{topological entropy}, with the weight induced by a real-valued continuous function, called \emph{potential}. The Variational Principle identifies the topological pressure with the supremum of its measure-theoretic counterpart, the \emph{measure-theoretic pressure}, over all invariant Borel probability measures \cite{bowen1975equilibrium, walters1982introduction}. Under additional regularity assumptions on the transformation and the potential, one gets the existence and uniqueness of an invariant Borel probability measure maximizing the measure-theoretic pressure, called the \emph{equilibrium state} for the given transformation and the potential. 
The study of the existence and uniqueness of the equilibrium states and their various other properties, such as ergodic properties, equidistribution, fractal dimensions, etc., has been the primary motivation for much research in the area.


The ergodic theory for expanding Thurston maps has been investigated in \cite{li2017ergodic} by the first-named author of the current paper.
In \cite{li2018equilibrium}, the first-named author of the current paper works out the thermodynamic formalism and investigates the existence, uniqueness, and other properties of equilibrium states for expanding Thurston maps.
In particular, for each expanding Thurston map without periodic critical points, by using a general framework devised by Y.~Kifer \cite{kifer1990large}, the first-named author of the current paper establishes level-$2$ large deviation principles for iterated preimages and periodic points in \cite{li2015weak}.

The current paper is the second in a series of two papers (together with \cite{shi2024thermodynamic}) investigating the ergodic theory of subsystems of expanding Thurston maps.
In the previous paper \cite{shi2024thermodynamic}, we investigated the thermodynamic formalism and demonstrated the existence of equilibrium states for subsystems and real-valued \holder continuous potentials.
In the present paper, we study the uniqueness and ergodic properties of equilibrium states for subsystems and real-valued \holder continuous potentials.
Based on the existence and uniqueness of equilibrium states, for subsystems of expanding Thurston maps without periodic critical point, we establish level-2 large deviation principles for iterated preimages.

\subsection{Main results}%
\label{sub:Main results}


In order to state our results more precisely, we quickly review some key concepts.
We refer the reader to Section~\ref{sec:Preliminaries} for more details.

Let $f \colon S^2 \mapping S^2$ be an expanding Thurston map with a Jordan curve $\mathcal{C}\subseteq S^2$ satisfying $\post{f} \subseteq \mathcal{C}$. 
We say that a map $F \colon \domF \mapping S^2$ is a \emph{subsystem of $f$ with respect to $\mathcal{C}$} if $\domF = \bigcup \mathfrak{X}$ for some non-empty subset $\mathfrak{X} \subseteq \Tile{1}$ and $F = f|_{\domF}$.
We denote by $\subsystem$ the set of all subsystems of $f$ with respect to $\mathcal{C}$.

Consider a subsystem $F \in \subsystem$. For each $n \in \n_0$, we define the \emph{set of $n$-tiles of $F$} to be 
\[
    \Domain{n} \define \{ X^n \in \Tile{n} \describe X^n \subseteq F^{-n}(F(\domF)) \},
\]
where we set $F^0 \define \id{S^{2}}$ when $n = 0$. We call each $X^n \in \Domain{n}$ an \emph{$n$-tile} of $F$. 
We define the \emph{tile maximal invariant set} associated with $F$ with respect to $\mathcal{C}$ to be
\[
    \limitset(F, \mathcal{C}) \define \bigcap_{n \in \n} \Bigl( \bigcup \Domain{n} \Bigr), 
\]
which is a compact subset of $S^{2}$. 

One of the key properties of $\limitset(F, \mathcal{C})$ is $F(\limitset(F, \mathcal{C})) \subseteq \limitset(F, \mathcal{C})$ (see Proposition~\ref{prop:subsystem:preliminary properties}~\ref{item:subsystem:properties:limitset forward invariant}). 
Therefore, we can restrict $F$ to $\limitset(F, \mathcal{C})$ and consider the map $F|_{\limitset(F, \mathcal{C})} \colon \limitset(F, \mathcal{C}) \to \limitset(F, \mathcal{C})$ and its iterations.


In the following theorem, we investigate the uniqueness and various ergodic properties of the equilibrium states for strongly primitive subsystems (see Definition~\ref{def:irreducibility of subsystem} in Subsection~\ref{sub:Subsystems of expanding Thurston maps}) and \holder continuous potentials, and establish the equidistribution of preimages of subsystems with respect to the equilibrium states.

\begin{theorem}        \label{thm:main:equilibrium state of subsystems}
    Let $f \colon X \mapping X$ be either an expanding Thurston map on a topological 2-sphere $X = S^2$ equipped with a visual metric or a postcritically-finite rational map with no periodic critical points on the Riemann sphere $X = \ccx$ equipped with the chordal metric.
    Let $\potential \colon X \mapping \real$ be \holder continuous.
    Let $\mathcal{C} \subseteq X$ be a Jordan curve containing $\post{f}$ with the property that $f(\mathcal{C}) \subseteq \mathcal{C}$. 
    Consider a strongly primitive subsystem $F \in \subsystem$.
    Denote $\limitset \define \limitset(F, \mathcal{C})$.
    
    Then there exists a unique equilibrium state $\mu_{F, \potential}$ for $F|_{\limitset}$ and $\phi|_{\limitset}$.    
    Moreover, $\mu_{F, \potential}$ is non-atomic and the measure-preserving transformation $F|_{\limitset}$ of the probability space $(\limitset, \mu_{F, \potential})$ is forward quasi-invariant, exact, and in particular, mixing and ergodic.

    In addition, the preimages points of $F$ are equidistributed with respect to $\mu_{F, \potential}$, i.e., for each sequence $\sequen{x_{n}}$ of points in $X$ and each sequence $\sequen{\colour_{n}}$ of colors in $\colours$ satisfying $x_{n} \in X^0_{\colour_{n}}$ for each $n \in \n$, we have
    \[
        \frac{1}{Z_{n}(\potential)} \sum_{ y \in F^{-n}(x_{n}) } \ccndegF{\colour_{n}}{}{n}{y} \myexp[\big]{ S_{n}^{F} \potential(y) } \frac{1}{n} \sum_{i = 0}^{n - 1} \delta_{F^{i}(y)}
        \weakconverge \equstate  \qquad \text{as } n \to +\infty, 
    \]
    where $S_{n}^{F}\potential(y) \define \sum_{i = 0}^{n - 1} \potential(F^{i}(y))$ and $Z_{n}(\potential) \define \sum_{ y \in F^{-n}(x_{n}) } \ccndegF{\colour_{n}}{}{n}{y} \myexp[\big]{S_{n}^{F}\potential(y)}$.
\end{theorem}

Here $\juxtapose{X^0_{\black}}{X^0_{\white}} \in \mathbf{X}^0(f, \mathcal{C})$ are the black $0$-tile and the white $0$-tile (see Subsection~\ref{sub:Thurston_maps}), respectively, $\ccndegF{\black}{}{n}{x}$ and $\ccndegF{\white}{}{n}{x}$ are the black degree and white degree of $F^n$ at $x$ (see Definition~\ref{def:subsystem local degree} in Subsection~\ref{sub:Subsystems of expanding Thurston maps}), respectively, and the symbol $w^{*}$ indicates convergence in the weak$^{*}$ topology.

Theorem~\ref{thm:main:equilibrium state of subsystems} follows immediately from Remark~\ref{rem:chordal metric visual metric qs equiv}, Theorems~\ref{thm:subsystem:uniqueness of equilibrium state}, \ref{thm:subsystem exact with respect to equilibrium state}, Corollaries~\ref{coro:subsystem equilibrium state non-atomic}, \ref{coro:subsystem mixing}, and Theorem~\ref{thm:split version subsystem equidistribution for preimages}.

\begin{rmk}
The existence of equilibrium states in Theorem~\ref{thm:main:equilibrium state of subsystems} has been established in \cite[Theorem~1.1]{shi2024thermodynamic}. 
\end{rmk}

Based on Theorem~\ref{thm:main:equilibrium state of subsystems}, for strongly primitive subsystems of expanding Thurston maps without periodic critical point, we obtain a level-2 large deviation principle (see Subsection~\ref{sub:Level-2 large deviation principles} for definitions) for Birkhoff averages and iterated preimages.

\begin{theorem}   \label{thm:level-2 larger deviation principles for subsystem}
    Let $f \colon X \mapping X$ be either an expanding Thurston map with no periodic critical points on a topological 2-sphere $X = S^2$ equipped with a visual metric or a postcritically-finite rational map with no periodic critical points on the Riemann sphere $X = \ccx$ equipped with the chordal metric.
    Let $\potential \colon X \mapping \real$ be \holder continuous.
    Let $\mathcal{C} \subseteq X$ be a Jordan curve containing $\post{f}$ with the property that $f(\mathcal{C}) \subseteq \mathcal{C}$. 
    Consider a strongly primitive subsystem $F \in \subsystem$.
    Denote $\limitset \define \limitset(F, \mathcal{C})$.
    Let $\probmea{\limitset}$ denote the space of Borel probability measures on $\limitset$ equipped with the weak$^{*}$-topology.
    Let $\equstate$ be the unique equilibrium state for $F|_{\limitset}$ and $\potential|_{\limitset}$.

    For each $n \in \n$, let $V_{n} \colon \limitset \mapping \probmea{\limitset}$ be the continuous function defined by 
    \begin{equation}    \label{eq:def:delta measure for orbit:subsystem}
        \deltameasure{x} \define \frac{1}{n} \sum_{i = 0}^{n - 1} \delta_{F^{i}(x)},
    \end{equation}
    and denote $S_{n}^{F}\potential(x) \define \sum_{i = 0}^{n - 1} \potential(F^{i}(x))$ for each $x \in \limitset$.
    For each $n \in \n$, we consider the following Borel probability measures on $\probmea{\limitset}$.

    {\bf Birkhoff averages.} $\birkhoffmeasure \define (V_{n})_{*}(\equstate)$ (i.e., $\birkhoffmeasure$ is the push-forward of $\equstate$ by $V_{n} \colon \limitset \mapping \probmea{\limitset}$).


    {\bf Iterated preimages.} Given a sequence $\{ x_{j} \}_{j \in \n}$ of points in $\limitset \setminus \mathcal{C}$, put
    \begin{equation}    \label{eq:def:Iterated preimages distribution} 
        \Omega_{n}(x_{n}) \define \sum_{y \in (F|_{\limitset})^{-n}(x_{n})} \frac{ \myexp[\big]{ S_{n}^{F}\potential(y) } }{\sum_{y' \in (F|_{\limitset})^{-n}(x_{n})} \myexp[\big]{ S_{n}^{F}\potential(y') } }  \delta_{\deltameasure{y}}.
    \end{equation}

    Then each of the sequences $\{ \birkhoffmeasure \}_{n \in \n}$ and $\{ \Omega_{n}(x_{n}) \}_{n \in \n}$ converges to $\delta_{\equstate}$ in the weak$^{*}$ topology, and satisfies a large deviation principle with the rate function $\ratefun \colon \probmea{\limitset} \mapping [0, +\infty]$ given by
    \begin{equation}    \label{eq:def:rate function:subsystem}
         \ratefun(\mu) \define
        \begin{cases}
            \pressure - h_{\mu}(F|_{\limitset}) - \int\! \phi \,\mathrm{d}\mu & \mbox{if } \mu \in \mathcal{M}(\limitset, F|_{\limitset}); \\
            +\infty & \mbox{if } \mu \in \probmea{\limitset} \setminus \mathcal{M}(\limitset, F|_{\limitset}).
        \end{cases}
    \end{equation}
    Furthermore, for each convex open subset $\mathcal{G}$ of $\probmea{\limitset}$ containing some invariant measure, we have $\inf_{\mathcal{G}} \ratefun = \inf_{\overline{\mathcal{G}}} \ratefun$, and
    \begin{equation}    \label{eq:equalities for rate function:subsystem}
        \lim_{n \to +\infty} \frac{1}{n} \log \birkhoffmeasure(\mathcal{G})
        = \lim_{n \to +\infty} \frac{1}{n} \log \Omega_{n}(x_{n})(\mathcal{G})
        = -\inf_{\mathcal{G}} \ratefun,
    \end{equation}
    and \eqref{eq:equalities for rate function:subsystem} remains true with $\mathcal{G}$ replaced by its closure $\overline{\mathcal{G}}$.
\end{theorem}

We will prove Theorem~\ref{thm:level-2 larger deviation principles for subsystem} in Subsection~\ref{sub:Proof of large deviation principles}.
As an immediate consequence of Theorem~\ref{thm:level-2 larger deviation principles for subsystem}, we get the following corollary.
See Subsection~\ref{sub:Proof of large deviation principles} for the proof.

\begin{corollary} \label{coro:measure-theoretic pressure infimum on local basis:subsystem}
    Let $f \colon X \mapping X$ be either an expanding Thurston map with no periodic critical points on a topological 2-sphere $X = S^2$ equipped with a visual metric or a postcritically-finite rational map with no periodic critical points on the Riemann sphere $X = \ccx$ equipped with the chordal metric.
    Let $\potential \colon X \mapping \real$ be \holder continuous.
    Let $\mathcal{C} \subseteq X$ be a Jordan curve containing $\post{f}$ with the property that $f(\mathcal{C}) \subseteq \mathcal{C}$. 
    Consider a strongly primitive subsystem $F \in \subsystem$.
    Denote $\limitset \define \limitset(F, \mathcal{C})$.
    Let $\equstate$ be the unique equilibrium state for $F|_{\limitset}$ and $\potential|_{\limitset}$.

    Consider a sequence $\{ x_{n} \}_{n \in \n}$ of points in $\limitset \setminus \mathcal{C}$.
    Then for each $\mu \in \invmea[\limitset][F|_{\limitset}]$ and each convex local basis $G_{\mu}$ of $\probmea{\limitset}$ at $\mu$, we have
    \begin{equation}    \label{eq:coro:measure-theoretic pressure infimum on local basis:subsystem}
        \begin{split}
            h_{\mu}(F|_{\limitset}) + \int \! \potential \,\mathrm{d}\mu 
            &= \inf_{\mathcal{G} \in G_{\mu}} \biggl\{ \lim_{n \to +\infty} \frac{1}{n} \log \equstate (\set{x \in \limitset \describe \deltameasure[n]{x} \in \mathcal{G}}) \biggr\} + \pressure  \\
            &= \inf_{\mathcal{G} \in G_{\mu}} \biggl\{ \lim_{n \to +\infty} \frac{1}{n} \log \sum_{ y \in (F|_{\limitset})^{-n}(x_{n}), \deltameasure[n]{y} \in \mathcal{G} } \myexp[\big]{S_{n}^{F} \potential(y)} \biggr\}.
        \end{split}
    \end{equation}
    Here $V_{n}$ and $S_{n}^{F} \potential$ are as defined in Theorem~\ref{thm:level-2 larger deviation principles for subsystem}.
\end{corollary}

\subsection{Strategy and organization of the paper}%
\label{sub:Strategy and organization of the paper}

We now discuss the strategy of the proofs of our main results and describe the organization of the paper.

We divide the proof of Theorem~\ref{thm:main:equilibrium state of subsystems} into three parts: uniqueness, ergodic properties, and equidistribution results.
Note that in Theorem~\ref{thm:main:equilibrium state of subsystems} the existence of equilibrium states and the property that the measure-preserving transformation $F|_{\limitset}$ of the probability space $(\limitset, \mu_{F, \potential})$ is forward quasi-invariant have been established in \cite{shi2024thermodynamic} (see Theorem~\ref{thm:subsystem characterization of pressure and existence of equilibrium state}).

To prove the uniqueness of equilibrium states, we investigate the (G\^ateaux) differentiability of the topological pressure function and apply some techniques from functional analysis.
More precisely, a general fact from functional analysis (record in Theorem~\ref{thm:tengent}) states that for an arbitrary convex continuous function $Q \colon V \mapping \real$ on a separable Banach space $V$, there exists a unique continuous linear functional $L \colon V \mapping \real$ tangent to $Q$ at $x \in V$ if and only if the functional $t \mapsto Q(x + ty)$ is differentiable at $0$ for all $y$ in a subset $U$ of $V$ that is dense in the weak topology on $V$.
One then observes that for each continuous map $g \colon X \mapping X$ on a compact metric space $X$, the topological pressure function $P(g, \cdot) \colon C(X) \mapping \real$ is continuous and convex (see for example, \cite[Theorems~3.6.1 and 3.6.2]{przytycki2010conformal}), and if $\mu$ is an equilibrium state for $g$ and $\varphi \in C(X)$, then the continuous linear functional $u \mapsto \int \! u \,\mathrm{d}\mu$, for $u \in C(X)$, is tangent to $P(g, \cdot)$ at $\varphi$ (see for example, \cite[Theorem~3.6.6]{przytycki2010conformal}).
Thus, in order to verify the uniqueness of the equilibrium state associated with a subsystem $F$ and a real-valued \holder continuous potential $\phi$, it suffices to prove the function $t \mapsto P(F, \phi + t \gamma)$ is differentiable at $0$, for all $\gamma$ in a dense subspace of $C(S^2)$.
This is established in Theorem~\ref{thm:derivative of pressure subsystem}.

To prove Theorem~\ref{thm:derivative of pressure subsystem}, we introduce \emph{normalized split Ruelle operators} induced by split Ruelle operators, and establish some uniform bounds in Proposition~\ref{prop:modulus for function under split ruelle operator} and Lemma~\ref{lem:upper bound for uniform norm of normalized split ruelle operator}, which are then used to show uniform convergence results in Theorem~\ref{thm:converge of uniform norm of normalized split ruelle operator} and Lemma~\ref{lem:converge of derivative for pressure of subsystem}.
Unlike the case of expanding Thurston maps or uniformly expanding maps, for subsystems, one cannot define a normalized Ruelle operator just by normalizing potentials. 
More specifically, since the combinatorial structure of tiles of a subsystem is inadequate (for example, the number of white $1$-tiles may not equal the number of black $1$-tiles), the eigenfunctions of the split Ruelle operator may not be continuous on the sphere.
However, that the eigenfunctions are always continuous in the interior of $0$-tiles, i.e., discontinuities can only occur at the boundary of $0$-tiles.
To overcome such difficulties, we introduce the notion of the split sphere, which is defined as the disjoint union of two $0$-tiles (see Definition~\ref{def:split sphere}), instead of the original topological $2$-sphere.
Then we can define the normalized split Ruelle operators on the space of continuous functions on the split sphere.

We next prove that the measure-preserving transformation $F|_{\limitset}$ of the probability space $(\limitset, \equstate)$ is exact (Theorem~\ref{thm:subsystem exact with respect to equilibrium state}), where we use the Jacobian function and the Gibbs property of the equilibrium state $\equstate$ established in \cite[Proposition~6.29]{shi2024thermodynamic}.
It follows in particular that the equilibrium state $\equstate$ is non-atomic (Corollary~\ref{coro:subsystem equilibrium state non-atomic}) and the transformation $F|_{\limitset}$ is mixing and ergodic (Corollary~\ref{coro:subsystem mixing}).

Finally, we prove the equidistribution results for preimages (Theorem~\ref{thm:split version subsystem equidistribution for preimages}) by applying the uniform convergence results (Theorem~\ref{thm:converge of uniform norm of normalized split ruelle operator} and Lemma~\ref{lem:converge of derivative for pressure of subsystem}) established in the proof of the uniqueness.

To prove Theorem~\ref{thm:level-2 larger deviation principles for subsystem}, we use a variant of Y.~Kifer’s result \cite{kifer1990large} formulated by H.~Comman and J.~Rivera-Letelier \cite{comman2011large}, recorded in Theorem~\ref{thm:large deviation principles Kifer Comman Juan}. 
In order to apply Theorem~\ref{thm:large deviation principles Kifer Comman Juan}, we just need to verify three conditions:
(1) The existence and uniqueness of the equilibrium state.
(2) The upper semi-continuity of the measure-theoretic entropy.
(3) Some characterization of the topological pressure (see Propositions~\ref{prop:subsystem Birkhoff averages pressure characterization} and \ref{prop:subsystem preimage pressure:recall}).
The first condition has been established in Theorem~\ref{thm:main:equilibrium state of subsystems}.
The second condition is known for expanding Thurston maps without periodic critical points (see \cite[Theorem~1.1]{shi2024entropy}) and is satisfied in our setting by Lemma~\ref{lem:upper semi-continuity for submap}.
The last condition can be verified by using a characterization of topological pressure established in \cite{shi2024thermodynamic} (see \eqref{eq:equalities for characterizations of pressure} in Theorem~\ref{thm:existence of f invariant Gibbs measure}).

\smallskip

We now give a brief description of the structure of this paper.

In Section~\ref{sec:Notation}, we fix some notation that will be used throughout the paper.
In Section~\ref{sec:Preliminaries}, we first review some notions from ergodic theory and dynamical systems and go over some key concepts and results on Thurston maps.
Then we review some concepts and results on subsystems of expanding Thurston maps.  
In Section~\ref{sec:The Assumptions}, we state the assumptions on some of the objects in this paper, which we will repeatedly refer to later as the \emph{Assumptions in Section~\ref{sec:The Assumptions}}. 
In Section~\ref{sec:Uniqueness of the equilibrium states for subsystems}, we prove the uniqueness of the equilibrium states for subsystems. 
We introduce normalized split Ruelle operators and prove the uniform convergence for functions under iterations of the normalized split Ruelle operators.
In Section~\ref{sec:Ergodic Properties}, we prove some ergodic properties of the unique equilibrium state for subsystem.
In Section~\ref{sec:Equidistribution}, we establish equidistribution results for preimages for subsystems of expanding Thurston maps.
In Section~\ref{sec:Large deviation principles for subsystems}, we prove level-$2$ large deviation principles for iterated preimages for subsystems of expanding Thurston maps without periodic critical point.

\section{Notation}
\label{sec:Notation}
Let $\cx$ be the complex plane and $\ccx$ be the Riemann sphere. 
Let $S^2$ denote an oriented topological $2$-sphere.
We use $\n$ to denote the set of integers greater than or equal to $1$ and write $\n_0 \define \{0\} \cup \n$. 
The symbol log denotes the logarithm to the base $e$. 
For $x \in \real$, we define $\lfloor x \rfloor$ as the greatest integer $\leqslant x$, and $\lceil x \rceil$ the smallest integer $\geqslant x$.
We denote by $\sgn{x}$ the sign function for each $x \in \real$.
The cardinality of a set $A$ is denoted by $\card{A}$.


Let $g \colon X \mapping Y$ be a map between two sets $X$ and $Y$. We denote the restriction of $g$ to a subset $Z$ of $X$ by $g|_{Z}$.

Consider a map $f \colon X \mapping X$ on a set $X$. 
The inverse map of $f$ is denoted by $f^{-1}$. 
We write $f^n$ for the $n$-th iterate of $f$, and $f^{-n}\define (f^n)^{-1}$, for each $n \in \n$. 
We set $f^0 \define \id{X}$, the identity map on $X$. 
For a real-valued function $\varphi \colon X \mapping \real$, we write
\begin{equation}    \label{eq:def:Birkhoff average}
    S_n \varphi(x) = S^f_n \varphi(x) \define \sum_{j=0}^{n-1} \varphi \bigl( f^j(x) \bigr)
\end{equation}
for each $x\in X$ and each $n\in \n_0$. 
We omit the superscript $f$ when the map $f$ is clear from the context. Note that when $n = 0$, by definition we always have $S_0 \varphi = 0$.

Let $(X,d)$ be a metric space. For each subset $Y \subseteq X$, we denote the diameter of $Y$ by $\diam{d}{Y} \define \sup\{d(x, y) \describe \juxtapose{x}{y} \in Y\}$, the interior of $Y$ by $\interior{Y}$, and the characteristic function of $Y$ by $\indicator{Y}$, which maps each $x \in Y$ to $1 \in \real$ and vanishes otherwise. 
For each $r > 0$ and each $x \in X$, we denote the open (\resp closed) ball of radius $r$ centered at $x$ by $B_{d}(x,r)$ (\resp $\overline{B_{d}}(x,r)$). We often omit the metric $d$ in the subscript when it is clear from the context.

For a compact metrizable topological space $X$, we denote by $C(X)$ (\resp $B(X)$) the space of continuous (\resp bounded Borel) functions from $X$ to $\real$, by $\mathcal{M}(X)$ the set of finite signed Borel measures, and $\mathcal{P}(X)$ the set of Borel probability measures on $X$.
By the Riesz representation theorem (see for example, \cite[Theorems~7.17 and 7.8]{folland2013real}), we identify the dual of $C(X)$ with the space $\mathcal{M}(X)$.
For $\mu \in \mathcal{M}(X)$, we use $\norm{\mu}$ to denote the total variation norm of $\mu$, $\supp{\mu}$ the support of $\mu$ (the smallest closed set $A \subseteq X$ such that $|\mu|(X \setminus A) = 0$), and \[
    \functional{\mu}{u} \define \int \! u \,\mathrm{d}\mu
\]
for each $u \in C(X)$. 
For a point $x \in X$, we define $\delta_x$ as the Dirac measure supported on $\{x\}$.
For a continuous map $g \colon X \mapping X$, we set $\mathcal{M}(X, g)$ to be the set of $g$-invariant Borel probability measures on $X$.
If we do not specify otherwise, we equip $C(X)$ with the uniform norm $\normcontinuous{\cdot}{X} \define \uniformnorm{\cdot}$, and equip $\mathcal{M}(X)$, $\mathcal{P}(X)$, and $\mathcal{M}(X, g)$ with the weak$^*$ topology.

The space of real-valued \holder continuous functions with an exponent $\holderexp \in (0,1]$ on a metric space $(X, d)$ is denoted as $\holderspace$. For each $\phi \in \holderspace$, \[
    \holderseminorm{\phi}{X} \define \sup\left\{ \frac{|\phi(x) - \phi(y)|}{d(x, y)^{\holderexp}} \describe \juxtapose{x}{y} \in X, \, x \ne y \right\},
\]
and the \holder norm is defined as $\holdernorm{\phi}{X} \define \holderseminorm{\phi}{X} + \normcontinuous{\phi}{X}$.

\section{Preliminaries}
\label{sec:Preliminaries}

\subsection{Thermodynamic formalism}  
\label{sub:thermodynamic formalism}

We first review some basic concepts from ergodic theory and dynamical systems. 
We refer the reader to \cite[Chapter~3]{przytycki2010conformal}, \cite[Chapter~9]{walters1982introduction}, or \cite[Chapter~20]{katok1995introduction} for more detailed studies of these concepts.

\smallskip

Let $(X,d)$ be a compact metric space and $g \colon X \mapping X$ a continuous map. Given $n \in \n$, 
\[
    d^n_g(x, y) \define \operatorname{max} \bigl\{  d \bigl(g^k(x), g^k(y) \bigr) \describe k \in \{0, \, 1, \, \dots, \, n-1 \} \!\bigr\}, \quad \text{ for } \juxtapose{x}{y} \in X,
\]
defines a metric on $X$. A set $F \subseteq X$ is \emph{$(n, \epsilon)$-separated} (with respect to $g$), for some $n \in \n$ and $\epsilon > 0$, if for each pair of distinct points $\juxtapose{x}{y} \in F$, we have $d^n_g(x, y) \geqslant \epsilon$. Given $\epsilon > 0$ and $n \in \n$, let $F_n(\epsilon)$ be a maximal (in the sense of inclusion) $(n,\epsilon)$-separated set in $X$.

For each real-valued continuous function $\psi \in C(X)$, the following limits exist and are equal, and we denote these limits by $P(g, \psi)$ (see for example, \cite[Theorem~3.3.2]{przytycki2010conformal}):

\begin{equation}  \label{eq:def:topological pressure}
    \begin{split}
        P(g, \psi) \define &  \lim \limits_{\epsilon \to 0^{+}} \limsup\limits_{n \to +\infty} \frac{1}{n} \log \sum_{x \in F_n(\epsilon)} \myexp{S_n \psi(x)}
        = \lim \limits_{\epsilon \to 0^{+}} \liminf\limits_{n \to +\infty} \frac{1}{n} \log  \sum\limits_{x \in F_n(\epsilon)} \myexp{S_n \psi(x)},
    \end{split}
\end{equation}
where $S_n \psi (x) = \sum_{j = 0}^{n - 1} \psi(g^j(x))$ is defined in \eqref{eq:def:Birkhoff average}. We call $P(g, \psi)$ the \emph{topological pressure} of $g$ with respect to the \emph{potential} $\psi$. 
Note that $P(g, \psi)$ is independent of $d$ as long as the topology on $X$ defined by $d$ remains the same (see for example, \cite[Section~3.2]{przytycki2010conformal}).
The quantity $h_{\operatorname{top}}(g) \define P(g, 0)$ is called the \emph{topological entropy} of $g$. 

We denote by $\mathcal{M}(X, g)$ the set of all $g$-invariant Borel probability measures on $X$.


Let $\mu \in \mathcal{M}(X, g)$. 
Then we say that $g$ is \emph{ergodic} for $\mu$ (or $\mu$ is \emph{ergodic} for $g$) if for each set $A \in \mathcal{B}$ with $g^{-1}(A) = A$ we have $\mu(A) = 0$ or $\mu(A) = 1$. 
The map $g$ is called \emph{mixing} for $\mu$ if
\begin{equation}    \label{eq:def:mixing}
    \lim_{n \to +\infty} \mu(g^{-n}(A) \cap B) = \mu(A) \mu(B)
\end{equation}
for all $\juxtapose{A}{B} \in \mathcal{B}$. 
It is easy to see that if $g$ is mixing for $\mu$, then $g$ is also ergodic. 

For each real-valued continuous function $\psi \in C(X)$, the \emph{measure-theoretic pressure} $P_\mu(g, \psi)$ of $g$ for the measure $\mu \in \mathcal{M}(X, g)$ and the potential $\psi$ is
\begin{equation}  \label{eq:def:measure-theoretic pressure}
P_\mu(g, \psi) \define  h_\mu (g) + \int \! \psi \,\mathrm{d}\mu,
\end{equation}
where $h_{\mu}(g)$ is the measure-theoretic entropy of $g$ for $\mu$.

The topological pressure is related to the measure-theoretic pressure by the so-called \emph{Variational Principle}.
It states that (see for example, \cite[Theorem~3.4.1]{przytycki2010conformal})
\begin{equation}  \label{eq:Variational Principle for pressure}
    P(g, \psi) = \sup \{P_\mu(g, \psi) \describe \mu \in \mathcal{M}(X, g)\}
\end{equation}
for each $\psi \in C(X)$.
In particular, when $\psi$ is the constant function $0$,
\begin{equation}  \label{eq:Variational Principle for entropy}
    h_{\operatorname{top}}(g) = \sup\{h_{\mu}(g) \describe\mu \in \mathcal{M}(X, g)\}.
\end{equation}
A measure $\mu$ that attains the supremum in \eqref{eq:Variational Principle for pressure} is called an \emph{equilibrium state} for the map $g$ and the potential $\psi$. A measure $\mu$ that attains the supremum in \eqref{eq:Variational Principle for entropy} is called a \emph{measure of maximal entropy} of $g$.




Let $\widetilde{X}$ be another compact metric space. If $\mu$ is a measure on $X$ and the map $\pi \colon X \mapping \widetilde{X}$ is continuous, then the \emph{push-forward} $\pi_{*} \mu$ of $\mu$ by $\pi$ is the measure given by $\pi_{*}\mu(A) \define \mu \parentheses[\big]{ \pi^{-1}(A) }$ for all Borel sets $A \subseteq \widetilde{X}$. 

\subsection{Thurston maps}%
\label{sub:Thurston_maps}
In this subsection, we go over some key concepts and results on Thurston maps, and expanding Thurston maps in particular. 
For a more thorough treatment of the subject, we refer to \cite{bonk2017expanding}.

\smallskip

Let $S^2$ denote an oriented topological $2$-sphere. A continuous map $f \colon S^2 \mapping S^2$ is called a \emph{branched covering map} on $S^2$ if for each point $x\in S^2$, there exists a positive integer $d\in \n$, open neighborhoods $U$ of $x$ and $V$ of $y \define f(x)$, open neighborhoods $U'$ and $V'$ of $0$ in $\ccx$, and orientation-preserving homeomorphisms $\varphi \colon U \mapping U'$ and $\eta \colon V \mapping V'$ such that $\varphi(x) = 0, \eta(y) = 0$, and $\big(\eta \circ f\circ \varphi^{-1}\bigr)(z) = z^d$ for each $z \in U'$. 
The positive integer $d$ above is called the \emph{local degree} of $f$ at $x$ and is denoted by $\deg_f(x)$ or $\deg(f, x)$.

The \emph{degree} of $f$ is $\deg{f} = \sum_{x\in f^{-1}(y)} \deg_{f}(x)$ for $y\in S^2$ and is independent of $y$. If $f \colon S^2 \mapping S^2$ and $g \colon S^2 \mapping S^2$ are two branched covering maps on $S^2$, then so is $f\circ g$, and $\deg(f\circ g, x) = \deg(g, x) \deg(f, g(x))$
for each $x\in S^2$, and moreover, $\deg(f\circ g) = (\deg{f})(\deg{g})$.

A point  $x\in S^2$ is a \emph{critical point} of $f$ if $\deg_f(x) \geqslant 2$. The set of critical points of $f$ is denoted by $\crit{f}$. A point $y\in S^2$ is a \emph{postcritical point} of $f$ if $y = f^n(x)$ for some $x \in \crit{f}$ and $n\in \n$. The set of postcritical points of $f$ is denoted by $\post{f}$. Note that $\post{f} = \post{f^n}$ for all $n\in \n$.

\begin{definition}[Thurston maps]
    A Thurston map is a branched covering map $f \colon S^2 \mapping S^2$ on $S^2$ with $\deg f \geqslant 2$ and $\card{\post{f}}< +\infty$.
\end{definition}

We now recall the notation for cell decompositions of $S^2$ used in \cite{bonk2017expanding} and \cite{li2017ergodic}. A \emph{cell of dimension $n$} in $S^2$, $n \in \{1, \, 2\}$, is a subset $c \subseteq S^2$ that is homeomorphic to the closed unit ball $\overline{\mathbb{B}^n}$ in $\real^n$, where $\mathbb{B}^{n}$ is the open unit ball in $\real^{n}$. We define the \emph{boundary of $c$}, denoted by $\partial c$, to be the set of points corresponding to $\partial \mathbb{B}^n$ under such a homeomorphism between $c$ and $\overline{\mathbb{B}^n}$. The \emph{interior of $c$} is defined to be $\inte{c} = c \setminus \partial c$. For each point $x\in S^2$, the set $\{x\}$ is considered as a \emph{cell of dimension $0$} in $S^2$. For a cell $c$ of dimension $0$, we adopt the convention that $\partial c = \emptyset$ and $\inte{c} = c$. 

We record the following definition of cell decompositions from \cite[Definition~3.2]{bonk2017expanding}.

\begin{definition}[Cell decompositions]    \label{def:cell decomposition}
    Let $\mathbf{D}$ be a collection of cells in $S^2$. We say that $\mathbf{D}$ is a \emph{cell decomposition of $S^2$} if the following conditions are satisfied:
    \begin{enumerate}[label= (\roman*)]
        \smallskip
        
        \item the union of all cells in $\mathbf{D}$ is equal to $S^2$,
        
        \smallskip
        
        \item if $c \in \mathbf{D}$, then $\partial c$ is a union of cells in $\mathbf{D}$,
        
        \smallskip
        
        \item for $\juxtapose{c_1}{c_2} \in \mathbf{D}$ with $c_1 \ne c_2$, we have $\inte{c_1} \cap \inte{c_2} = \emptyset$,
        
        \smallskip
        
        \item every point in $S^2$ has a neighborhood that meets only finitely many cells in $\mathbf{D}$.
    \end{enumerate}
\end{definition}

\begin{definition}[Refinements]
    Let $\mathbf{D}'$ and $\mathbf{D}$ be two cell decompositions of $S^2$. We say that $\mathbf{D}'$ is a \emph{refinement} of $\mathbf{D}$ if the following conditions are satisfied:
    \begin{enumerate}[label = (\roman*)]
        \smallskip

        \item every cell $c \in \mathbf{D}$ is the union of all cells $c' \in \mathbf{D}'$ with $c' \subseteq c$.

        \smallskip

        \item for every cell $c' \in \mathbf{D}'$ there exists a cell $c \in \mathbf{D}$ with $c' \subseteq c$.
    \end{enumerate}
\end{definition}

\begin{definition}[Cellular maps and cellular Markov partitions]
    Let $\mathbf{D}'$ and $\mathbf{D}$ be two cell decompositions of $S^2$. We say that a continuous map $f \colon S^2 \mapping S^2$ is \emph{cellular} for $(\mathbf{D}', \mathbf{D})$ if for every cell $c \in \mathbf{D}'$, the restriction $f|_c$ of $f$ to $c$ is a homeomorphism of $c$ onto a cell in $\mathbf{D}$. We say that $(\mathbf{D}',\mathbf{D})$ is a \emph{cellular Markov partition} for $f$ if $f$ is cellular for $(\mathbf{D}', \mathbf{D})$ and $\mathbf{D}'$ is a refinement of $\mathbf{D}$.
\end{definition}

Let $f \colon S^2 \mapping S^2$ be a Thurston map, and $\mathcal{C}\subseteq S^2$ be a Jordan curve containing $\post{f}$. 
Then the pair $f$ and $\mathcal{C}$ induces natural cell decompositions $\mathbf{D}^n(f,\mathcal{C})$ of $S^2$, for each $n \in \n_0$, in the following way:

By the Jordan curve theorem, the set $S^2 \setminus \mathcal{C}$ has two connected components. We call the closure of one of them the \emph{white $0$-tile} for $(f,\mathcal{C})$, denoted by $X^0_{\white}$, and the closure of the other one the \emph{black $0$-tile} for $(f,\mathcal{C})$, denoted be $X^0_{\black}$. 
The set of $0$-\emph{tiles} is $\mathbf{X}^0(f, \mathcal{C}) \define \bigl\{X^0_{\black}, \, X^0_{\white} \bigr\}$. 
The set of $0$-\emph{vertices} is $\mathbf{V}^0(f, \mathcal{C}) \define \post{f}$. 
We set $\overline{\mathbf{V}}^0(f, \mathcal{C}) \define \bigl\{ \{x\} \describe x\in \mathbf{V}^0(f,\mathcal{C}) \bigr\}$. 
The set of $0$-\emph{edges} $\mathbf{E}^0(f,\mathcal{C})$ is the set of the closures of the connected components of $\mathcal{C} \setminus \post{f}$. 
Then we get a cell decomposition\[
    \mathbf{D}^0(f,\mathcal{C}) \define \mathbf{X}^0(f, \mathcal{C}) \cup \mathbf{E}^0(f,\mathcal{C}) \cup \overline{\mathbf{V}}^0(f,\mathcal{C})
\]
of $S^2$ consisting of \emph{cells of level }$0$, or $0$-\emph{cells}.

We can recursively define the unique cell decomposition $\mathbf{D}^n(f,\mathcal{C})$, $n\in \n$, consisting of $n$-\emph{cells} such that $f$ is cellular for $(\mathbf{D}^{n+1}(f,\mathcal{C}), \mathbf{D}^n(f,\mathcal{C}))$. We refer to \cite[Lemma~5.12]{bonk2017expanding} for more details. 
We denote by $\mathbf{X}^n(f,\mathcal{C})$ the set of $n$-cells of dimension 2, called $n$-\emph{tiles}; by $\mathbf{E}^n(f,\mathcal{C})$ the set of $n$-cells of dimension $1$, called $n$-\emph{edges}; by $\overline{\mathbf{V}}^n(f,\mathcal{C})$ the set of $n$-cells of dimension $0$; and by $\mathbf{V}^n(f,\mathcal{C})$ the set $\{x \describe \{x\} \in \overline{\mathbf{V}}^n(f,\mathcal{C})\}$, called the set of $n$-\emph{vertices}. 
The $k$-\emph{skeleton}, for $k\in \{0, \, 1, \, 2\}$, of $\mathbf{D}^n(f,\mathcal{C})$ is the union of all $n$-cells of dimension $k$ in this cell decomposition.

We record \cite[Lemma~5.17]{bonk2017expanding}.

\begin{lemma}[M.~Bonk \& D.~Meyer \cite{bonk2017expanding}]    \label{lem:cell mapping properties of Thurston map}
    Let $\juxtapose{k}{n} \in \n_0$, $f \colon S^2 \mapping S^2$ be a Thurston map, and $\mathcal{C} \subseteq S^2$ be a Jordan curve with $\post{f} \subseteq \mathcal{C}$.
    \begin{enumerate}[label=\rm{(\roman*)}]
        \smallskip

        \item     \label{item:lem:cell mapping properties of Thurston map:i}
        If $c \subseteq S^2$ is a topological cell such that $f^{k}|_{c}$ is a homeomorphism onto its image and $f^{k}(c)$ is an $n$-cell, then $c$ is an $(n + k)$-cell.

        \smallskip

        \item \label{item:lem:cell mapping properties of Thurston map:ii}
        If $X$ is an $n$-tile and $p \in S^2$ is a point with $f^{k}(p) \in \inte{X}$, then there exists a unique $(n + k)$-tile $X'$ with $p \in X'$ and $f^{k}(X') = X$.
    \end{enumerate}
\end{lemma}


For $n\in \n_0$, we define the \emph{set of black $n$-tiles} as\[
    \textbf{X}^n_{\black}(f,\mathcal{C}) \define \left\{ X \in \mathbf{X}^n (f,\mathcal{C}) \describe f^n(X) = X^0_{\black} \right\},
\]
and the \emph{set of white $n$-tiles} as\[
    \mathbf{X}^n_{\white}(f,\mathcal{C}) \define \left\{X\in \mathbf{X}^n(f,\mathcal{C}) \describe f^n(X) = X^0_{\white}\right\}.
\]

From now on, if the map $f$ and the Jordan curve $\mathcal{C}$ are clear from the context, we will sometimes omit $(f,\mathcal{C})$ in the notation above.

\begin{definition}[Expansion]     \label{def:expanding_Thurston_maps}
    A Thurston map $f \colon S^2 \mapping S^2$ is called \emph{expanding} if there exists a metric $d$ on $S^2$ that induces the standard topology on $S^2$ and a Jordan curve $\mathcal{C} \subseteq S^2$ containing $\post{f}$ such that
    \begin{equation}    \label{eq:definition of expansion}
        \lim_{n \to +\infty} \max\{ \diam{d}{X} \describe X \in \mathbf{X}^n(f,\mathcal{C}) \} = 0.
    \end{equation}
\end{definition}

\begin{remark}\label{rem:Expansion_is_independent}
    It is clear that if $f$ is an expanding Thurston map, so is $f^{n}$ for each $n \in \n$. 
    We observe that being expanding is a topological property of a Thurston map and independent of the choice of the metric $d$ that generates the standard topology on $S^2$. 
    By Lemma~6.2 in \cite{bonk2017expanding}, it is also independent of the choice of the Jordan curve $\mathcal{C}$ containing $\post{f}$. More precisely, if $f$ is an expanding Thurston map, then\[
        \lim_{n \to +\infty} \max \{ \diam{\widetilde{d}}{X} \describe X\in \mathbf{X}^n(f,\widetilde{\mathcal{C}}) \} = 0,        
    \]
    for each metric $\widetilde{d}$ that generates the standard topology on $S^2$ and each Jordan curve $\widetilde{\mathcal{C}} \subseteq S^2$ that contains $\post{f}$.
\end{remark}

For an expanding Thurston map $f$, we can fix a particular metric $d$ on $S^2$ called a \emph{visual metric for $f$}. 
For the existence and properties of such metrics, see \cite[Chapter~8]{bonk2017expanding}. 
For a visual metric $d$ for $f$, there exists a unique constant $\Lambda > 1$ called the \emph{expansion factor} of $d$ (see \cite[Chapter~8]{bonk2017expanding} for more details). 
One major advantage of a visual metric $d$ is that in $(S^2,d)$ we have good quantitative control over the sizes of the cells in the cell decompositions discussed above. 

\begin{remark}\label{rem:chordal metric visual metric qs equiv}
    If $f \colon \ccx \mapping \ccx$ is a rational expanding Thurston map, then a visual metric is quasisymmetrically equivalent to the chordal metric on the Riemann sphere $\ccx$ (see \cite[Theorem~18.1~(ii)]{bonk2017expanding}). 
    Here the chordal metric $\sigma$ on $\ccx$ is given by $\sigma (z, w) \define \frac{2\abs{z - w}}{\sqrt{1 + \abs{z}^2} \sqrt{1 + \abs{w}^2}}$ for all $\juxtapose{z}{w} \in \cx$, and $\sigma(\infty, z) = \sigma(z, \infty) \define \frac{2}{\sqrt{1 + \abs{z}^2}}$ for all $z \in \cx$. 
    We also note that quasisymmetric embeddings of bounded connected metric spaces are \holder continuous (see \cite[Section~11.1 and Corollary~11.5]{heinonen2001lectures}). 
    Accordingly, the classes of \holder continuous functions on $\ccx$ equipped with the chordal metric and on $S^2 = \ccx$ equipped with any visual metric for $f$ are the same (up to a change of the \holder exponent).
\end{remark}

A Jordan curve $\mathcal{C} \subseteq S^2$ is \emph{$f$-invariant} if $f(\mathcal{C}) \subseteq \mathcal{C}$. If $\mathcal{C}$ is $f$-invariant with $\post{f} \subseteq \mathcal{C}$, then the cell decompositions $\mathbf{D}^{n}(f, \mathcal{C})$ have nice compatibility properties. In particular, $\mathbf{D}^{n+k}(f, \mathcal{C})$ is a refinement of $\mathbf{D}^{n}(f, \mathcal{C})$, whenever $\juxtapose{n}{k} \in \n_0$. Intuitively, this means that each cell $\mathbf{D}^{n}(f, \mathcal{C})$ is ``subdivided'' by the cells in $\mathbf{D}^{n+k}(f, \mathcal{C})$. A cell $c\in \mathbf{D}^{n}(f, \mathcal{C})$ is actually subdivided by the cells in $\mathbf{D}^{n+k}(f, \mathcal{C})$ ``in the same way'' as the cell $f^n(c) \in \mathbf{D}^{0}(f, \mathcal{C})$ by the cells in $\mathbf{D}^{k}(f, \mathcal{C})$. 

For convenience we record Proposition~12.5~(ii) of \cite{bonk2017expanding} here, which is easy to check but useful. 

\begin{proposition}[M.~Bonk \& D.~Meyer \cite{bonk2017expanding}]    \label{prop:cell decomposition: invariant Jordan curve}
    Let $\juxtapose{k}{n} \in \n_0$, $f \colon S^2 \mapping S^2$ be a Thurston map, and $\mathcal{C} \subseteq S^2$ be an $f$-invariant Jordan curve with $\post{f} \subseteq \mathcal{C}$. Then every $(n+k)$-tile $X^{n+k}$ is contained in a unique $k$-tile $X^k$.
\end{proposition}

M.~Bonk and D.~Meyer \cite[Theorem~15.1]{bonk2017expanding} proved that there exists an $f^n$-invariant Jordan curve $\mathcal{C}$ containing $\post{f}$ for each sufficiently large $n$ depending on $f$. 

\begin{lemma}[M.~Bonk \& D.~Meyer \cite{bonk2017expanding}]    \label{lem:invariant_Jordan_curve}
    Let $f \colon S^2 \mapping S^2$ be an expanding Thurston map, and $\widetilde{\mathcal{C}} \subseteq S^2$ be a Jordan curve with $\post{f} \subseteq \widetilde{\mathcal{C}}$. Then there exists an integer $N(f, \widetilde{\mathcal{C}}) \in \n$ such that for each $n \geqslant N(f,\widetilde{\mathcal{C}})$ there exists an $f^n$-invariant Jordan curve $\mathcal{C}$ isotopic to $\widetilde{\mathcal{C}}$ rel. $\post{f}$.
\end{lemma}





We record the following lemma from \cite[Lemma~3.13]{li2018equilibrium}, which generalizes \cite[Lemma~15.25]{bonk2017expanding}.

\begin{lemma}[M.~Bonk \& D.~Meyer \cite{bonk2017expanding}; Z.~Li \cite{li2018equilibrium}]     \label{lem:basic_distortion}
    Let $f \colon S^2 \mapping S^2$ be an expanding Thurston map, and $\mathcal{C} \subseteq S^2$ be a Jordan curve that satisfies $\post{f} \subseteq \mathcal{C}$ and $f^{n_{\mathcal{C}}}(\mathcal{C}) \subseteq \mathcal{C}$ for some $n_{\mathcal{C}} \in \n$. Let $d$ be a visual metric on $S^2$ for $f$ with expansion factor $\Lambda > 1$. 
    Then there exists a constant $C_0 > 1$, depending only on $f$, $\mathcal{C}$, $n_{\mathcal{C}}$, and $d$, with the following property:

    If $\juxtapose{n}{k} \in \n_0$, $X^{n+k}\in \mathbf{X}^{n+k}(f,\mathcal{C})$, and $\juxtapose{x}{y} \in X^{n + k}$, then
    \begin{equation}     \label{eq:basic_distortion}
        C_0^{-1}d(x, y) \leqslant d(f^n(x), f^n(y)) / \Lambda^{n} \leqslant C_0 d(x, y).
    \end{equation}
\end{lemma}

The next distortion lemma follows immediately from \cite[Lemma~5.1]{li2018equilibrium}.

\begin{lemma}    \label{lem:distortion_lemma}
    Let $f \colon S^2 \mapping S^2$ be an expanding Thurston map, and $\mathcal{C} \subseteq S^2$ be a Jordan curve that satisfies $\post{f} \subseteq \mathcal{C}$ and $f^{n_{\mathcal{C}}}(\mathcal{C}) \subseteq \mathcal{C}$ for some $n_{\mathcal{C}} \in \n$. Let $d$ be a visual metric on $S^2$ for $f$ with expansion factor $\Lambda > 1$.
    Let $\potential \in \holderspacesphere$ be a real-valued \holder continuous function with an exponent $\holderexp \in (0, 1]$.
    Then there exists a constant $C_1 \geqslant 0$ depending only on $f$, $\mathcal{C}$, $d$, $\phi$, and $\holderexp$ such that for all $n \in \n_0$, $X^n \in \Tile{n}$, and $\juxtapose{x}{y} \in X^n$,
    \begin{equation}    \label{eq:distortion_lemma}
        \left| S_n\phi(x) - S_n\phi(y) \right| \leqslant C_1 d(f^n(x), f^n(y))^{\holderexp} \leqslant \Cdistortion.
    \end{equation}
    Quantitatively, we choose
    \begin{equation}     \label{eq:const:C_1}
        C_1 \define C_0  \holderseminorm{\potential}{S^2}  \big/ \bigl( 1 - \Lambda^{-\holderexp} \bigr) ,
    \end{equation}
    where $C_0 > 1$ is the constant depending only on $f$, $\mathcal{C}$, and $d$ from Lemma~\ref{lem:basic_distortion}.
\end{lemma}

\subsection{Subsystems of expanding Thurston maps}%
\label{sub:Subsystems of expanding Thurston maps}

In this subsection, we review some concepts and results on subsystems of expanding Thurston maps. 
We refer the reader to \cite[Section~5]{shi2024thermodynamic} for details. 

\smallskip

We first introduce the definition of subsystems along with relevant concepts and notations that will be used frequently throughout this paper.
Additionally, we will provide examples to illustrate these ideas.

\begin{definition}    \label{def:subsystems}
    Let $f \colon S^2 \mapping S^2$ be an expanding Thurston map with a Jordan curve $\mathcal{C}\subseteq S^2$ satisfying $\post{f} \subseteq \mathcal{C}$. 
    We say that a map $F \colon \domF \mapping S^2$ is a \emph{subsystem of $f$ with respect to $\mathcal{C}$} if $\domF = \bigcup \mathfrak{X}$ for some non-empty subset $\mathfrak{X} \subseteq \Tile{1}$ and $F = f|_{\domF}$.
    We denote by $\subsystem$ the set of all subsystems of $f$ with respect to $\mathcal{C}$.
    Define \[
        \operatorname{Sub}_{*}(f, \mathcal{C}) \define \{ F \in \subsystem \describe \domF \subseteq F(\domF) \}.
    \]
\end{definition}

Consider a subsystem $F \in \subsystem$. 
For each $n \in \n_0$, we define the \emph{set of $n$-tiles of $F$} to be
\begin{equation}    \label{eq:definition of tile of subsystem}
    \Domain{n} \define \{ X^n \in \Tile{n} \describe X^n \subseteq F^{-n}(F(\domF)) \},
\end{equation}
where we set $F^0 \define \id{S^{2}}$ when $n = 0$. We call each $X^n \in \Domain{n}$ an \emph{$n$-tile} of $F$. 
We define the \emph{tile maximal invariant set} associated with $F$ with respect to $\mathcal{C}$ to be
\begin{equation}    \label{eq:def:limitset}
    \limitset(F, \mathcal{C}) \define \bigcap_{n \in \n} \Bigl( \bigcup \Domain{n} \Bigr), 
\end{equation}
which is a compact subset of $S^{2}$. 
Indeed, $\limitset(F, \mathcal{C})$ is forward invariant with respect to $F$, namely, $F(\limitset(F, \mathcal{C})) \subseteq \limitset(F, \mathcal{C})$ (see Proposition~\ref{prop:subsystem:preliminary properties}~\ref{item:subsystem:properties:limitset forward invariant}). 
We denote by $\limitmap$ the map $F|_{\limitset(F, \mathcal{C})} \colon \limitset(F, \mathcal{C}) \mapping \limitset(F, \mathcal{C})$.

Let $\juxtapose{X^0_{\black}}{X^0_{\white}} \in \mathbf{X}^0(f, \mathcal{C})$ be the black $0$-tile and the white $0$-tile, respectively. 
We define the \emph{color set of $F$} as \[
    \colourset \define \bigl\{ \colour \in \colours \describe X^0_{\colour} \in \Domain{0} \bigr\}.
\]
For each $n \in \n_0$, we define the \emph{set of black $n$-tiles of $F$} as\[
    \bFTile{n} \define \bigl\{ X \in \Domain{n} \describe F^{n}(X) = X^0_{\black} \bigr\},
\] 
and the \emph{set of white $n$-tiles of $F$} as\[
    \wFTile{n} \define \bigl\{ X \in \Domain{n} \describe F^{n}(X) = X^0_{\white} \bigr\}. 
\]
Moreover, for each $n \in \n_0$ and each pair of $\juxtapose{\colour}{\colour'} \in \colours$ we define 
\[
    \ccFTile{n}{\colour}{\colour'} \define \bigl\{ X \in \cFTile{n} \describe X \subseteq X^0_{\colour'} \bigr\}.
\]
In other words, for example, a tile $X \in \ccFTile{n}{\black}{\white}$ is a \emph{black $n$-tile of $F$ contained in $\whitetile$}, i.e., an $n$-tile of $F$ that is contained in the white $0$-tile $X^0_{\white}$ as a set, and is mapped by $F^{n}$ onto the black $0$-tile $\blacktile$.

By abuse of notation, we often omit $(F, \mathcal{C})$ in the notations above when it is clear from the context.


We discuss three examples below and refer the reader to \cite[Subsection~5.1]{shi2024thermodynamic} for more examples.

\begin{example}    \label{exam:subsystems}
    Let $f \colon S^2 \mapping S^2$ be an expanding Thurston map with a Jordan curve $\mathcal{C}\subseteq S^2$ satisfying $\post{f} \subseteq \mathcal{C}$.
    Consider $F \in \subsystem$.
    \begin{enumerate}[label=(\roman*)]

        \item     \label{item:exam:subsystems:strongly irreducible but not primitive} 
            The map $F$ satisfies $\domF = X^1_{\black} \cup X^1_{\white}$ for some $X^1_{\black} \in \mathbf{X}^1_{\black}(f, \mathcal{C})$ and $X^1_{\white} \in \cTile{1}{\white}$ satisfying $X^1_{\black} \subseteq \inte[\big]{X^0_{\white}}$ and $X^1_{\white} \subseteq \inte[\big]{X^0_{\black}}$. 
            In this case, $F$ is surjective and $\limitset = \{p,\, q\}$ for some $p \in X^1_{\black}$ and $q \in X^1_{\white}$. 
            One sees that $F(\limitset) = \limitset$ since $F(p) = q$ and $F(q) = p$.

        \smallskip

        \item     \label{item:exam:subsystems:Sierpinski carpet}
            The map $F \colon \domF \mapping S^2$ is represented by Figure~\ref{fig:subsystem:example:carpet}.
            Here $S^{2}$ is identified with a pillow that is obtained by gluing two squares together along their boundaries.
            Moreover, each square is subdivided into $3\times 3$ subsquares, and $\dom{F}$ is obtained from $S^2$ by removing the interior of the middle subsquare $X^{1}_{\white} \in \cTile{1}{\white}$ and $X^{1}_{\black} \in \cTile{1}{\black}$ of the respective squares. 
            In this case, $\limitset$ is a \sierpinski carpet. 
            It consists of two copies of the standard square \sierpinski carpet glued together along the boundaries of the squares.
            \begin{figure}[H]
                \centering
                \begin{overpic}
                    [width=12cm, tics=20]{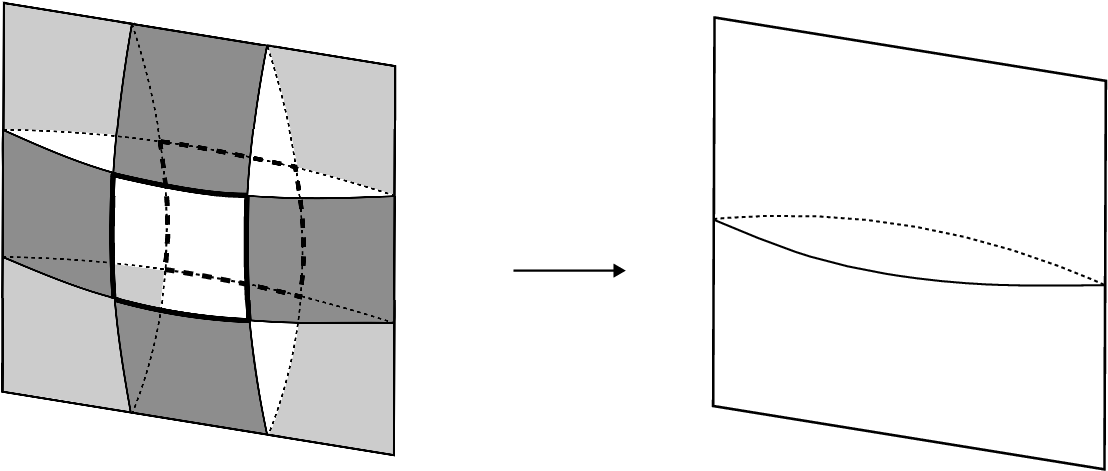}
                    \put(50,20){$F$}
                    \put(16,41){$\domF$}
                    \put(83,40){$S^2$}
                \end{overpic}
                \caption{A \sierpinski carpet subsystem.} 
                \label{fig:subsystem:example:carpet}
            \end{figure}

        \smallskip

        \item     \label{item:exam:subsystems:Sierpinski gasket}
            The map $F \colon \domF \mapping S^2$ is represented by Figure~\ref{fig:subsystem:example:gasket}.
            Here $S^{2}$ is identified with a pillow that is obtained by gluing two equilateral triangles together along their boundaries.
            Moreover, each triangle is subdivided into $4$ small equilateral triangles, and $\dom{F}$ is obtained from $S^2$ by removing the interior of the middle small triangle $X^{1}_{\black} \in \cTile{1}{\black}$ and $X^{1}_{\white} \in \cTile{1}{\white}$ of the respective triangle. 
            In this case, $\limitset$ is a \sierpinski gasket. 
            It consists of two copies of the standard \sierpinski gasket glued together along the boundaries of the triangles.
            \begin{figure}[H]
                \centering
                \begin{overpic}
                    [width=12cm, tics=20]{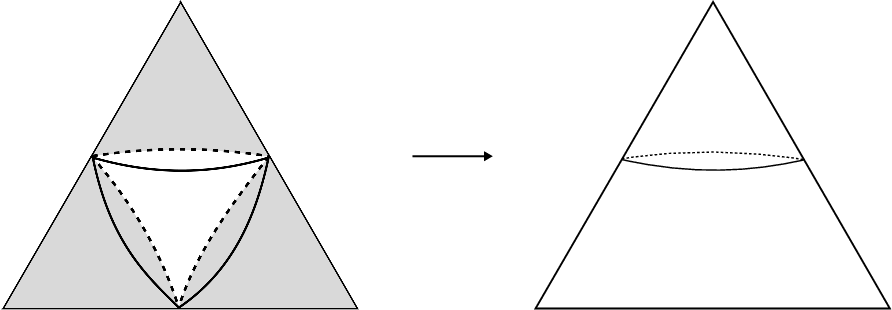}
                    \put(50,19){$F$}
                    \put(0,23){$\domF$}
                    \put(90,24){$S^2$}
                \end{overpic}
                \caption{A \sierpinski gasket subsystem.} 
                \label{fig:subsystem:example:gasket}
            \end{figure}
    \end{enumerate}
\end{example}

We summarize some preliminary results for subsystems in the following proposition.

\begin{proposition}[Z.~Li, X.~Shi, Y.~Zhang \cite{shi2024thermodynamic}]    \label{prop:subsystem:preliminary properties}
    Let $f \colon S^2 \mapping S^2$ be an expanding Thurston map with a Jordan curve $\mathcal{C}\subseteq S^2$ satisfying $\post{f} \subseteq \mathcal{C}$.
    Consider $F \in \subsystem$.
    Consider arbitrary $\juxtapose{n}{k} \in \n_0$.
    Then the following statements hold:
    \begin{enumerate}[label=\rm{(\roman*)}]
        \smallskip
        \item    \label{item:subsystem:properties:homeo} 
            If $X \in \Domain{n + k}$ is any $(n + k)$-tile of $F$, then $F^k(X)$ is an $n$-tile of $F$, and $F^{k}|_{X}$ is a homeomorphism of $X$ onto $F^{k}(X)$. As a consequence we have $\bigl\{ F^{k}(X) \describe X \in \Domain{n + k} \bigr\} \subseteq \Domain{n}$.
             
        \smallskip

        \item     \label{item:subsystem:properties:limitset forward invariant}
            The tile maximal invariant set $\limitset$ is forward invariant with respect to $F$, i.e., $F(\limitset) \subseteq \limitset$.

        \smallskip

        \item     \label{item:subsystem:properties:properties invariant Jordan curve:decreasing relation of domains}
            If $f(\mathcal{C}) \subseteq \mathcal{C}$, then $\domain{n + k} \subseteq \domain{n} \subseteq \domain{1} = \domF$ for all $\juxtapose{n}{k} \in \n$.

        \smallskip
        
        \item    \label{item:subsystem:properties:sursubsystem properties:property of domain and limitset} 
            If $f(\mathcal{C}) \subseteq \mathcal{C}$ and $F \in \sursubsystem$, then $F(\limitset) = \limitset \ne \emptyset$.    
    \end{enumerate}
\end{proposition}

Proposition~\ref{prop:subsystem:preliminary properties}~\ref{item:subsystem:properties:homeo} and \ref{item:subsystem:properties:limitset forward invariant} are from \cite[Proposition~5.4~(i) and (ii)]{shi2024thermodynamic}. 
Proposition~\ref{prop:subsystem:preliminary properties}~\ref{item:subsystem:properties:properties invariant Jordan curve:decreasing relation of domains} is from \cite[Proposition~5.5~(i)]{shi2024thermodynamic}. 
Proposition~\ref{prop:subsystem:preliminary properties}~\ref{item:subsystem:properties:sursubsystem properties:property of domain and limitset} is from \cite[Proposition~5.6~(ii)]{shi2024thermodynamic}.

\smallskip

We record the following notions of degrees and local degrees for subsystems from \cite[Subsection~5.3]{shi2024thermodynamic}.

\begin{definition}[Degrees]   \label{def:subsystem local degree}
    Let $f \colon S^2 \mapping S^2$ be an expanding Thurston map with a Jordan curve $\mathcal{C}\subseteq S^2$ satisfying $\post{f} \subseteq \mathcal{C}$.
    Consider $F \in \subsystem$.
    The \emph{degree} of $F$ is defined as \[
        \deg{(F)} \define \sup \bigl\{ \card[\big]{F^{-1}(\{y\})} \describe y \in S^2 \bigr\}.
    \]

    Fix arbitrary $x \in S^2$ and $n \in \n$. 
    We define the \emph{black degree} of $F^n$ at $x$ as
    \[ 
        \ccndegF{\black}{}{n}{x} \define \card{ \neighbortile{n}{\black}{}{x} },
    \]
    where $\neighbortile{n}{\black}{}{x} \define \{X \in \bFTile{n} \describe x \in X \}$ is the \emph{set of black $n$-tiles of $F$ at $x$}.
    Similarly, we define the \emph{white degree} of $F^n$ at $x$ as\[
        \ccndegF{\white}{}{n}{x} \define \card{ \neighbortile{n}{\white}{}{x} },
    \]
    where $\neighbortile{n}{\white}{}{x} \define \{X \in \wFTile{n} \describe x \in X \}$ is the \emph{set of white $n$-tiles of $F$ at $x$}. 
    Moreover, the \emph{local degree} of $F^{n}$ at $x$ is defined as
    \begin{equation}    \label{eq:subsystem local degree greater than color degree}
        \ccndegF{}{}{n}{x} \define \max\{\ccndegF{\black}{}{n}{x}, \ccndegF{\white}{}{n}{x}\},
    \end{equation}
    and the \emph{set of $n$-tiles of $F$ at $x$} is $\neighbortile{n}{}{}{x} \define \{X \in \Domain{n} \describe x \in X \}$.
    Furthermore, for each pair of $\juxtapose{\colour}{\colour'} \in \colours$ we define
    \begin{align*}
        \neighbortile{n}{\colour}{\colour'}{x} &\define \{X \in \ccFTile{n}{\colour}{\colour'} \describe x \in X \},  \\
        \ccndegF{\colour}{\colour'}{n}{x} &\define \card{ \neighbortile{n}{\colour}{\colour'}{x} }, 
    \end{align*}
    and the \emph{local degree matrix of $F^{n}$ at $x$} is
    \[
        \qquad \quad \Deg{n}{x} \define \begin{bmatrix}
            \ccndegF{\black}{\black}{n}{x} & \ccndegF{\white}{\black}{n}{x} \\
            \ccndegF{\black}{\white}{n}{x} & \ccndegF{\white}{\white}{n}{x}
        \end{bmatrix}.
    \]
\end{definition}

\smallskip

We record the following two definitions from \cite[Subsection~5.5]{shi2024thermodynamic}.

\begin{definition}[Irreducibility]    \label{def:irreducibility of subsystem}
    Let $f \colon S^2 \mapping S^2$ be an expanding Thurston map with a Jordan curve $\mathcal{C}\subseteq S^2$ satisfying $\post{f} \subseteq \mathcal{C}$.
    Consider $F \in \subsystem$.
    We say $F$ is an \emph{irreducible} (\resp a \emph{strongly irreducible}) subsystem (of $f$ with respect to $\mathcal{C}$) if for each pair of $\juxtapose{\colour}{\colour'} \in \colours$, there exists an integer $n_{\colour \colour'} \in \n$ and $X^{n_{\colour \colour'}} \in \cFTile{n_{\colour \colour'}}$ satisfying $X^{n_{\colour \colour'}} \subseteq X^0_{\colour'}$ (\resp $X^{n_{\colour \colour'}} \subseteq \inte[\big]{X^0_{\colour'}}$).
    We denote by $n_{F}$ the constant $\max_{\juxtapose{\colour}{\colour'} \in \colours} n_{\colour \colour'}$, which depends only $F$ and $\mathcal{C}$.
\end{definition}

Obviously, if $F$ is irreducible then $\colourset = \colours$ and $F(\domF) = S^{2}$.


\begin{definition}[Primitivity]    \label{def:primitivity of subsystem}
    Let $f \colon S^2 \mapping S^2$ be an expanding Thurston map with a Jordan curve $\mathcal{C}\subseteq S^2$ satisfying $\post{f} \subseteq \mathcal{C}$.
    Consider $F \in \subsystem$.
    We say that $F$ is a \emph{primitive} (\resp \emph{strongly primitive}) subsystem (of $f$ with respect to $\mathcal{C}$) if there exists an integer $n_{F} \in \n$ such that for each pair of $\juxtapose{\colour}{\colour'} \in \colours$ and each integer $n \geqslant n_{F}$, there exists $X^n \in \cFTile{n}$ satisfying $X^n \subseteq X^0_{\colour'}$ (\resp $X^n \subseteq \inte[\big]{X^0_{\colour'}}$). 
\end{definition}




We record \cite[Lemmas~5.21 and 5.22]{shi2024thermodynamic} below.
\begin{lemma}[Z.~Li, X.~Shi, Y.~Zhang \cite{shi2024thermodynamic}]    \label{lem:strongly irreducible:tile in interior tile for high enough level}
    Let $f \colon S^2 \mapping S^2$ be an expanding Thurston map with a Jordan curve $\mathcal{C}\subseteq S^2$ satisfying $\post{f} \subseteq \mathcal{C}$.
    Let $F \in \subsystem$ be irreducible (\resp strongly irreducible).
    Let $n_F \in \n$ be the constant from Definition~\ref{def:primitivity of subsystem}, which depends only on $F$ and $\mathcal{C}$. 
    Then for each $k \in \n_{0}$, each $\colour \in \colours$, and each $k$-tile $X^k \in \Domain{k}$, there exists an integer $n \in \n$ with $n \leqslant n_{F}$ and $X^{k + n}_{\colour} \in \cFTile{k + n}$ satisfying $X^{k + n}_{\colour} \subseteq X^k$ (\resp $X^{k + n}_{\colour} \subseteq \inte{X^k}$).
\end{lemma}

\begin{lemma}[Z.~Li, X.~Shi, Y.~Zhang \cite{shi2024thermodynamic}]    \label{lem:strongly primitive:tile in interior tile for high enough level}
    Let $f \colon S^2 \mapping S^2$ be an expanding Thurston map with a Jordan curve $\mathcal{C}\subseteq S^2$ satisfying $\post{f} \subseteq \mathcal{C}$.
    Let $F \in \subsystem$ be primitive (\resp strongly primitive).
    Let $n_F \in \n$ be the constant from Definition~\ref{def:primitivity of subsystem}, which depends only on $F$ and $\mathcal{C}$. 
    Then for each $n \in \n$ with $n \geqslant n_{F}$, each $m \in \n_0$, each $\colour \in \colours$, and each $m$-tile $X^m \in \Domain{m}$, there exists an $(n + m)$-tile $X^{n + m}_{\colour} \in \cFTile{n + m}$ such that $X^{n + m}_{\colour} \subseteq X^m$ (\resp $X^{n + m}_{\colour} \subseteq \inte{X^m}$).
\end{lemma}

The following distortion lemma serves as cornerstones in the development of thermodynamic formalism for subsystems of expanding Thurston maps (see \cite[Lemma~5.25]{shi2024thermodynamic}).

\begin{lemma}[Z.~Li, X.~Shi, Y.~Zhang \cite{shi2024thermodynamic}]    \label{lem:distortion lemma for subsystem}
    Let $f \colon S^2 \mapping S^2$ be an expanding Thurston map, and $\mathcal{C} \subseteq S^2$ be a Jordan curve that satisfies $\post{f} \subseteq \mathcal{C}$ and $f(\mathcal{C}) \subseteq \mathcal{C}$.
    Consider $F \in \subsystem$. 
    Let $d$ be a visual metric on $S^2$ for $f$ with expansion factor $\Lambda > 1$.
    Let $\potential \in \holderspacesphere$ be a real-valued \holder continuous function with an exponent $\holderexp \in (0, 1]$.
    Then the following statements hold:
    \begin{enumerate}[label=\rm{(\roman*)}]
        \smallskip
        
        \item     \label{item:lem:distortion lemma for subsystem:holder bound on same color tile}
        For each $n \in \n_{0}$, each $\colour \in \colourset$, and each pair of $\juxtapose{x}{y} \in X^0_{\colour}$, we have
        \begin{equation}    \label{eq:same color distortion bound for split operator}
            \frac{ \sum_{X^n \in \cFTile{n}} \myexp[\big]{ S^{F}_{n} \phi\bigl( (F^n|_{X^n})^{-1}(x) \bigr) } }{ \sum_{X^n \in \cFTile{n}} \myexp[\big]{ S^{F}_{n} \phi\bigl( (F^n|_{X^n})^{-1}(x) \bigr) } }  
            \leqslant \myexp[\big]{ C_1 d(x, y)^{\holderexp} } 
            \leqslant \myexp[\big]{ \Cdistortion },
        \end{equation}
        where $C_1 \geqslant 0$ is the constant defined in \eqref{eq:const:C_1} in Lemma~\ref{lem:distortion_lemma} and depends only on $f$, $\mathcal{C}$, $d$, $\phi$, and $\holderexp$.
            
        \smallskip

        \item     \label{item:lem:distortion lemma for subsystem:uniform bound}
        If $F$ is irreducible, then there exists a constant $\Csplratio \geqslant 1$ depending only on $F$, $\mathcal{C}$, $d$, $\phi$, and $\holderexp$ such that for each $n \in \n_{0}$, each pair of $\juxtapose{\colour}{\colour'} \in \colourset$, each $x \in X^0_{\colour}$, and each $y \in X^0_{\colour'}$, we have
        \begin{equation}    \label{eq:distinct color distortion bound for split operator}
            \frac{ \sum_{X^n_{\colour} \in \ccFTile{n}{\colour}{}} \myexp[\big]{ S^{F}_{n} \phi\bigl( (F^n|_{X^n_{\colour}})^{-1}(x) \bigr) } }{ \sum_{X^n_{\colour'} \in \ccFTile{n}{\colour'}{}} \myexp[\big]{ S^{F}_{n} \phi\bigl( (F^n|_{X^n_{\colour'}})^{-1}(y) \bigr) } } \leqslant \Csplratio.
        \end{equation}
        Quantitatively, we choose 
        \begin{equation}    \label{eq:const:Csplratio}
            \Csplratio \define \CsplratioExpression,
        \end{equation}
        where $n_{F} \in \n$ is the constant in Definition~\ref{def:irreducibility of subsystem} and depends only on $F$ and $\mathcal{C}$, and $C_1 \geqslant 0$ is the constant defined in \eqref{eq:const:C_1} in Lemma~\ref{lem:distortion_lemma} and depends only on $f$, $\mathcal{C}$, $d$, $\phi$, and $\holderexp$.
    \end{enumerate}
\end{lemma}

\subsection{Ergodic theory of subsystems}%
\label{sub:Ergodic theory of subsystems}

In this subsection, we review some concepts and results on ergodic theory of subsystems of expanding Thurston maps. 
We refer the reader to \cite[Section~6]{shi2024thermodynamic} for details and proofs. 

\smallskip

We first recall the topological pressure for subsystems.

\begin{definition}[Topological pressure]    \label{def:pressure for subsystem}
    Let $f \colon S^2 \mapping S^2$ be an expanding Thurston map with a Jordan curve $\mathcal{C}\subseteq S^2$ satisfying $\post{f} \subseteq \mathcal{C}$.
    Consider $F \in \subsystem$.
    For a real-valued function $\varphi \colon S^2 \mapping \real$, we denote \[
        Z_{n}(F, \varphi ) \define \sum_{X^n \in \Domain{n}} \myexp[\big]{ \sup\bigl\{S^F_n \varphi(x) \describe x \in X^n \bigr\} } 
    \]
    for each $n \in \n$. We define the \emph{topological pressure} of $F$ with respect to the \emph{potential} $\varphi$ by
    \begin{equation}    \label{eq:pressure of subsystem}
        \pressure[\varphi] \define \liminf_{n \mapping +\infty} \frac{1}{n} \log ( Z_{n}(F, \varphi) ).
    \end{equation}
    We denote 
    \begin{equation}    \label{eq:def:normed potential}
        \overline{\varphi} \define \varphi - \pressure[\varphi].
    \end{equation}
\end{definition}

We introduce the notion of the split sphere (see Definition~\ref{def:split sphere}), and set up some identifications and conventions (see Remarks~\ref{rem:disjoint union} and \ref{rem:probability measure in split setting}), which will be used frequently in this paper.

\begin{definition}    \label{def:split sphere}
    Let $f \colon S^2 \mapping S^2$ be an expanding Thurston map with a Jordan curve $\mathcal{C}\subseteq S^2$ satisfying $\post{f} \subseteq \mathcal{C}$. 
    We define the \emph{split sphere} $\splitsphere$ to be the disjoint union of $X^0_{\black}$ and $X^0_{\white}$, i.e., \[
        \splitsphere \define X^0_{\black} \sqcup X^0_{\white} = \bigl\{ (x, \colour) \describe \colour \in \colours, \, x \in X^0_{\colour} \bigr\}.
    \] 
    For each $\colour \in \colours$, let
    \begin{equation}    \label{eq:natural injection into splitsphere}
        i_{\colour} \colon X^0_{\colour} \mapping \splitsphere
    \end{equation}
    be the natural injection (defined by $i_{\colour}(x) \define (x, \colour)$). 
    Recall that the topology on $\splitsphere$ is defined as the finest topology on $\splitsphere$ for which both the natural injections $i_{\black}$ and $i_{\white}$ are continuous. 
    In particular, $\splitsphere$ is compact and metrizable.
\end{definition}

Let $X$ and $Y$ be normed vector spaces. Recall that a bounded linear map $T$ from $X$ to $Y$ is said to be an \emph{isomorphism} if $T$ is bijective and $T^{-1}$ is bounded (in other words, $\norm{T(x)} \geqslant C\norm{x}$ for some $C > 0$), and $T$ is called an \emph{isometry} if $\norm{T(x)} = \norm{x}$ for all $x \in X$.
\begin{proposition}[Dual of the product space is isometric to the product of the dual spaces]    \label{prop:Dual of the product is isometric to the product of the dual}
    Let $X$ and $Y$ be normed vector spaces and define $T \colon X^{*} \times Y^{*} \mapping (X \times Y)^{*}$ by $T(u, v)(x, y) = u(x) + v(y)$. Then $T$ is an isomorphism which is an isometry with respect to the norm $\norm{(x, y)} = \max\{\norm{x}, \norm{y}\}$ on $X \times Y$, the corresponding operator norm on $(X \times Y)^{*}$, and the norm $\norm{(u, v)} \define \norm{u} + \norm{v}$ on $X^{*} \times Y^{*}$.
\end{proposition}
See \cite[Proposition~6.12]{shi2024thermodynamic} for a proof of Proposition~\ref{prop:Dual of the product is isometric to the product of the dual}.

By Proposition~\ref{prop:Dual of the product is isometric to the product of the dual} and the Riesz representation theorem (see \cite[Theorems~7.17 and 7.8]{folland2013real}), we can identify $\bigl( \splfunspace \bigr)^{*}$ with the product of spaces of finite signed Borel measures $\splmeaspace$, where we use the norm $\norm{\splfun} \define \max\{\norm{u_{\black}}, \norm{u_{\white}}\}$ on $\splfunspace$, the corresponding operator norm on $\bigl(\splfunspace\bigl)^{*}$, and the norm $\norm{\splmea} \define \norm{\mu_{\black}} + \norm{\mu_{\white}}$ on $\splmeaspace$.

\smallskip

\noindent\textbf{Notational Remark.}
From now on, we write
\begin{align}
    \splmea(A_{\black}, A_{\white}) &\define \mu_{\black}(A_{\black}) + \mu_{\white}(A_{\white}), \nonumber\\
    \functional{\splmea}{\splfun} 
    &\define \functional{\mu_{\black}}{u_{\black}} + \functional{\mu_{\white}}{u_{\white}}
    = \int_{X^0_{\black}} \! u_{\black} \,\mathrm{d}\mu_{\black} + \int_{X^0_{\white}} \! u_{\white} \,\mathrm{d}\mu_{\white} \nonumber, 
\end{align}
whenever $\splmea \in \splmeaspace$, $\splfun \in \splboufunspace$, and $A_{\black}$ and $A_{\white}$ are Borel subset of $X^0_{\black}$ and $X^0_{\white}$, respectively.
In particular, for each Borel set $A \subseteq S^2$, we define
\begin{equation}    \label{eq:split measure of set}
    \splmea(A) \define \splmea \bigl(A \cap X^0_{\black}, A \cap X^0_{\white} \bigr)
        = \mu_{\black} \bigl(A \cap X^0_{\black} \bigr) + \mu_{\white} \bigl(A \cap X^0_{\white} \bigr).
\end{equation}

\begin{remark}    \label{rem:disjoint union}
    In the natural way, the product space $\splfunspace$ (\resp $\splboufunspace$) can be identified with $C(\splitsphere)$ (\resp $B(\splitsphere)$). 
    Similarly, the product space $\splmeaspace$ can be identified with $\mathcal{M}(\splitsphere)$. 
    Under such identifications, we write\[
        \int \! \splfun \,\mathrm{d}\splmea \define \functional{\splmea}{\splfun} \qquad \text{and } \qquad
        \splfun\splmea \define (u_{\black}\mu_{\black}, u_{\white}\mu_{\white})
    \]
    whenever $\splmea \in \splmeaspace$ and $\splfun \in \splboufunspace$.

    Moreover, we have the following natural identification of $\probmea{\splitsphere}$:\[
        \probmea{\splitsphere} = \bigl\{ \splmea \in \splmeaspace \describe \mu_{\black} \text{ and } \mu_{\white} \text{ are positive measures, } \mu_{\black}\bigl(X^0_{\black}\bigr) + \mu_{\white}\bigl(X^0_{\white}\bigr) = 1 \bigr\}.
    \]
    Here we follow the terminology in \cite[Section~3.1]{folland2013real} that a \emph{positive measure} is a signed measure that takes values in $[0, +\infty]$. 
\end{remark}

\begin{remark}\label{rem:probability measure in split setting}
    It is easy to see that \eqref{eq:split measure of set} defines a finite signed Borel measure $\mu \define \splmea$ on $S^2$. 
    Here we use the notation $\mu$ (\resp $\splmea$) when we view the measure as a measure on $S^2$ (\resp $\splitsphere$), and we will always use these conventions in this paper. 
    In this sense, for each $u \in B(S^2)$ we have
    \begin{equation}    \label{eq:split measure on S2}
        \functional{\mu}{u} = \int \! u \,\mathrm{d}\mu = \int \! \splfun \,\mathrm{d}\splmea = \int_{X^0_{\black}} \! u \,\mathrm{d}\mu_{\black} + \int_{X^0_{\white}} \! u \,\mathrm{d}\mu_{\white},
    \end{equation}
    where $u_{\black} \define u|_{X^0_{\black}}$ and $u_{\white} \define u|_{X^0_{\white}}$. 
    Moreover, if both $\mu_{\black}$ and $\mu_{\white}$ are positive measures and $\mu_{\black}\bigl(X^0_{\black}\bigr) + \mu_{\white}\bigl(X^0_{\white}\bigr) = 1$, then $\mu = \splmea$ defined by \eqref{eq:split measure of set} is a Borel probability measure on $S^2$. In view of the identifications in Remark~\ref{rem:disjoint union}, this means that if $\splmea \in \mathcal{P}(\splitsphere)$, then $\mu \in \mathcal{P}(S^2)$.
\end{remark}

We next introduce the split Ruelle operator and its adjoint operator, which are the main tools in \cite{shi2024thermodynamic} and this paper to develop the thermodynamic formalism for subsystems of expanding Thurston maps.
We summarize relevant definitions and facts about the split Ruelle operator and its adjoint operator, and refer the reader to \cite[Subsections~6.2 and 6.4]{shi2024thermodynamic} for a detailed discussion.

\begin{definition}[Partial split Ruelle operators]    \label{def:partial Ruelle operator}
    Let $f \colon S^2 \mapping S^2$ be an expanding Thurston map with a Jordan curve $\mathcal{C}\subseteq S^2$ satisfying $\post{f} \subseteq \mathcal{C}$.
    Consider $F \in \subsystem$ and $\varphi \in C(S^2)$.
    We define a map $ \poperator[\varphi]{n} \colon \boufunspace{\colour'} \mapping \boufunspace{\colour}$, for $\juxtapose{\colour}{\colour'} \in \colours$, and $n \in \n_{0}$, by
    \begin{equation}    \label{eq:definition of partial Ruelle operators}
        \begin{split}
            \poperator[\varphi]{n}(u)(y) 
            &\define \sum_{ x \in F^{-n}(y) }  \ccndegF{\colour}{\colour'}{n}{x} u(x) \myexp[\big]{ S_n^F \varphi(x) }   \\
            &= \sum_{ X^n \in \ccFTile{n}{\colour}{\colour'} }  u\bigl((F^n|_{X^{n}})^{-1}(y)\bigr)  \myexp[\big]{ S_n^F \varphi\bigl((F^n|_{X^{n}})^{-1}(y)\bigr) }
        \end{split}
    \end{equation}
    for each real-valued bounded Borel function $u \in \boufunspace{\colour'}$ and each point $y \in X^0_{\colour}$. 
\end{definition}

The following lemma proved in \cite[Lemma~6.8]{shi2024thermodynamic} shows that the partial split Ruelle operators are well-defined and well-behaved under iterations.

\begin{lemma}[Z.~Li, X.~Shi, Y.~Zhang \cite{shi2024thermodynamic}]    \label{lem:well-defined partial Ruelle operator}
    Let $f \colon S^2 \mapping S^2$ be an expanding Thurston map with a Jordan curve $\mathcal{C}\subseteq S^2$ satisfying $\post{f} \subseteq \mathcal{C}$ and $f(\mathcal{C}) \subseteq \mathcal{C}$.
    Consider $F \in \sursubsystem$ and $\varphi \in C(S^2)$.
    Then for all $\juxtapose{n}{k} \in \n_0$, $\juxtapose{\colour}{\colour'} \in \colours$, and $u \in \confunspace{\colour'}$, we have
    \begin{align}
        \label{eq:well-defined continuity of partial Ruelle operator}
        \poperator[\varphi]{n} (u) &\in \confunspace{\colour} \qquad\quad \text{ and }  \\
        \label{eq:iteration of partial Ruelle operator}
        \qquad \qquad \poperator[\varphi]{n + k}(u) &= \sum_{\colour'' \in \colours} \paroperator[\varphi]{n}{\colour}{\colour''} \bigl( \paroperator[\varphi]{k}{\colour''}{\colour'}(u) \bigr).  
    \end{align}

\end{lemma}

\begin{definition}[Split Ruelle operators]    \label{def:split ruelle operator}
    Let $f \colon S^2 \mapping S^2$ be an expanding Thurston map with a Jordan curve $\mathcal{C}\subseteq S^2$ satisfying $\post{f} \subseteq \mathcal{C}$.
    Consider $F \in \subsystem$ and $\varphi \in C(S^2)$.
    The \emph{split Ruelle operator} for the subsystem $F$ and the potential $\varphi$
    \[
        \splopt[\varphi] \colon C\bigl(X^0_{\black}\bigr) \times C\bigl(X^0_{\white}\bigr) \mapping C\bigl(X^0_{\black}\bigr) \times C\bigl(X^0_{\white}\bigr) 
    \]
    on the product space $\splfunspace$ is defined by
    \begin{equation}    \label{eq:def:split ruelle operator}
        \splopt[\varphi]\splfun \define \bigl( 
            \paroperator[\varphi]{1}{\black}{\black}(u_{\black}) + \paroperator[\varphi]{1}{\black}{\white}(u_{\white}), 
            \paroperator[\varphi]{1}{\white}{\black}(u_{\black}) + \paroperator[\varphi]{1}{\white}{\white}(u_{\white})
        \bigr)
    \end{equation}
    for each $u_{\black} \in C\bigl(X^0_{\black}\bigr)$ and each $u_{\white} \in C\bigl(X^0_{\white}\bigr)$.
\end{definition}

The following lemma proved in \cite[Lemma~6.10]{shi2024thermodynamic} shows that the split Ruelle operator is well-behaved under iterations.

\begin{lemma}[Z.~Li, X.~Shi, Y.~Zhang \cite{shi2024thermodynamic}]    \label{lem:iteration of split-partial ruelle operator}
    Let $f \colon S^2 \mapping S^2$ be an expanding Thurston map with a Jordan curve $\mathcal{C}\subseteq S^2$ satisfying $\post{f} \subseteq \mathcal{C}$ and $f(\mathcal{C}) \subseteq \mathcal{C}$.
    Consider $F \in \subsystem$ and $\varphi \in C(S^2)$.
    We assume in addition that $F(\domF) = S^2$. 
    Then for all $n \in \n_0$, $u_{\black} \in C\bigl(X^0_{\black}\bigr)$, and $u_{\white} \in C\bigl(X^0_{\white}\bigr)$,
    \begin{equation}    \label{eq:iteration of split-partial ruelle operator}
        \splopt[\varphi]^{n} \splfun = \bigl( \paroperator[\varphi]{n}{\black}{\black}(u_{\black}) + \paroperator[\varphi]{n}{\black}{\white}(u_{\white}), \
        \paroperator[\varphi]{n}{\white}{\black}(u_{\black}) + \paroperator[\varphi]{n}{\white}{\white}(u_{\white}) \bigr).
    \end{equation}
\end{lemma}

Let $f \colon S^2 \mapping S^2$ be an expanding Thurston map with a Jordan curve $\mathcal{C}\subseteq S^2$ satisfying $\post{f} \subseteq \mathcal{C}$.
Consider $F \in \subsystem$ and $\varphi \in C(S^2)$.
We assume in addition that $F(\domF) = S^2$. 
Note that the split Ruelle operator $\splopt[\varphi]$ (see Definition~\ref{def:split ruelle operator}) is a positive, continuous operator on $\splfunspace$. 
Thus, the adjoint operator \[
    \splopt[\varphi]^{*} \colon \bigl( \splfunspace \bigr)^{*} \mapping \bigl( \splfunspace \bigr)^{*}
\]
of $\splopt[\varphi]$ acts on the dual space $\bigl( \splfunspace \bigr)^{*}$ of the Banach space $\splfunspace$. 
Recall that we identify $\bigl( \splfunspace \bigr)^{*}$ with the product of spaces of finite signed Borel measures $\splmeaspace$, where we use the norm $\norm{\splfun} = \max\{\norm{u_{\black}}, \norm{u_{\white}}\}$ on $\splfunspace$, the corresponding operator norm on $\bigl( \splfunspace \bigr)^{*}$, and the norm $\norm{\splmea} = \norm{\mu_{\black}} + \norm{\mu_{\white}}$ on $\splmeaspace$.
Then by Remark~\ref{rem:disjoint union}, we can also view $\splopt[\varphi]$ (\resp $\dualsplopt[\varphi]$) as an operator on $C(\splitsphere)$ (\resp $\mathcal{M}(\splitsphere)$).

We record \cite[Proposition~6.15~(iv)]{shi2024thermodynamic} in the following.
Roughly speaking, the following proposition says that measures will be concentrated on the limit set under iteration of the adjoint operator $\dualsplopt[\varphi]$.

\begin{proposition}[Z.~Li, X.~Shi, Y.~Zhang \cite{shi2024thermodynamic}]    \label{prop:dual split operator}
    Let $f \colon S^2 \mapping S^2$ be an expanding Thurston map with a Jordan curve $\mathcal{C}\subseteq S^2$ satisfying $\post{f} \subseteq \mathcal{C}$ and $f(\mathcal{C}) \subseteq \mathcal{C}$.
    Consider $F \in \subsystem$ and $\varphi \in C(S^2)$.
    We assume in addition that $F(\domF) = S^2$. 
    Consider arbitrary $n \in \n$ and $\splmea \in \splmeaspace$. 
    Then we have \[
        \bigl( \dualsplopt[\varphi] \bigr)^n \splmea \parentheses[\Big]{ \bigcup \splDomain{n - 1} }  
        = \bigl( \dualsplopt[\varphi] \bigr)^n \splmea \parentheses[\Big]{ \bigcup \splDomain{n} } ,
    \] 
    where $\splDomain{k} \define  \bigcup_{\colour \in \colours} \bigl\{ i_{\colour}(X^k) \describe X^k \in \Domain{k}, \, X^k \subseteq X^0_{\colour} \bigr\}$ for each $k \in \n_{0}$ (here $i_{\colour}$ is defined in \eqref{eq:natural injection into splitsphere}).
\end{proposition}

We record the following three results (Lemma~\ref{lem:distortion lemma for normed split operator}, Theorems~\ref{thm:subsystem:eigenmeasure existence and basic properties}, and \ref{thm:existence of f invariant Gibbs measure}) on the split Ruelle operators and their adjoint operators in our context.

\begin{lemma}[Z.~Li, X.~Shi, Y.~Zhang \cite{shi2024thermodynamic}]    \label{lem:distortion lemma for normed split operator}
    Let $f \colon S^2 \mapping S^2$ be an expanding Thurston map with a Jordan curve $\mathcal{C}\subseteq S^2$ satisfying $\post{f} \subseteq \mathcal{C}$ and $f(\mathcal{C}) \subseteq \mathcal{C}$.
    Let $F \in \subsystem$ be irreducible.
    Let $d$ be a visual metric on $S^2$ for $f$ with expansion factor $\Lambda > 1$.
    Let $\potential \in \holderspacesphere$ be a real-valued \holder continuous function with an exponent $\holderexp \in (0, 1]$.
    Then there exists a constant $\Cnormspllocal \geqslant 0$ depending only on $F$, $\mathcal{C}$, $d$, $\phi$, and $\holderexp$ such that for each $n \in \n$, each $\colour \in \colours$, and each pair of $\juxtapose{x}{y} \in X^0_{\colour}$, the following inequalities holds:
    \begin{align}    
        \label{eq:first bound for normed split operator}
        \normsplopt^{n}\bigl(\indicator{\splitsphere}\bigr)(\widetilde{x}) \big/ \normsplopt^{n}\bigl(\indicator{\splitsphere}\bigr)(\widetilde{y})
        &\leqslant \myexp[\big]{ C_1 d(x, y)^{\holderexp} }  \leqslant \Csplratio,\\
        \label{eq:second bound for normed split operator}
        \Csplratio^{-1} & \leqslant \normsplopt^{n}\bigl(\indicator{\splitsphere}\bigr)(\widetilde{x}) \leqslant \Csplratio,   \\
        \label{eq:third bound for normed split operator}
        \abs[\big]{\normsplopt^{n}\bigl(\indicator{\splitsphere}\bigr)(\widetilde{x}) - \normsplopt^{n}\bigl(\indicator{\splitsphere}\bigr)(\widetilde{y})}
        &\leqslant \Csplratio \bigl( \myexp[\big]{ C_1 d(x, y)^{\holderexp} } - 1 \bigr) \leqslant \Cnormspllocal d(x, y)^{\holderexp},
    \end{align}
    where $\widetilde{x} \define i_{\colour}(x) = (x, \colour) \in \splitsphere$, $\widetilde{y} \define i_{\colour}(y) = (y, \colour) \in \splitsphere$ (recall Remark~\ref{rem:disjoint union}), $C_1 \geqslant 0$ is the constant in Lemma~\ref{lem:distortion_lemma} depending only on $f$, $\mathcal{C}$, $d$, $\phi$, and $\holderexp$, and $\Csplratio \geqslant 1$ is the constant in Lemma~\ref{lem:distortion lemma for subsystem}~\ref{item:lem:distortion lemma for subsystem:uniform bound} depending only on $F$, $\mathcal{C}$, $d$, $\phi$, and $\holderexp$.
\end{lemma}
Recall from \eqref{eq:def:normed potential} that $\overline{\potential} = \potential - \pressure$. 
Lemma~\ref{lem:distortion lemma for normed split operator} was proved in \cite[Lemma~6.22]{shi2024thermodynamic}.

\begin{definition}[Gibbs measures for subsystems]    \label{def:subsystem gibbs measure}
    Let $f \colon S^2 \mapping S^2$ be an expanding Thurston map with a Jordan curve $\mathcal{C}\subseteq S^2$ satisfying $\post{f} \subseteq \mathcal{C}$.
    Consider $F \in \subsystem$.
    Let $d$ be a visual metric on $S^2$ for $f$ and $\potential$ be a real-valued \holder continuous function on $S^2$ with respect to the metric $d$.
    A Borel probability measure $\mu \in \mathcal{P}(S^2)$ is called a \emph{Gibbs measure} with respect to $F$, $\mathcal{C}$, and $\phi$ if there exist constants $P_{\mu} \in \real$ and $C_{\mu} \geqslant 1$ such that for each $n\in \n_0$, each $n$-tile $X^n \in \Domain{n}$, and each $x \in X^n$, we have
    \begin{equation}    \label{eq:def:subsystem gibbs measure}
        \frac{1}{C_{\mu}} \leqslant \frac{\mu(X^n)}{ \myexp{ S_{n}^{F}\phi(x) - nP_{\mu} } } \leqslant C_{\mu}.
    \end{equation}
\end{definition}

One observes that for each Gibbs measure $\mu$ with respect to $F$, $\mathcal{C}$, and $\phi$, the constant $P_{\mu}$ is unique.

\begin{theorem}[Z.~Li, X.~Shi, Y.~Zhang \cite{shi2024thermodynamic}]    \label{thm:subsystem:eigenmeasure existence and basic properties}
    Let $f \colon S^2 \mapping S^2$ be an expanding Thurston map with a Jordan curve $\mathcal{C}\subseteq S^2$ satisfying $\post{f} \subseteq \mathcal{C}$ and $f(\mathcal{C}) \subseteq \mathcal{C}$.
    Consider $F \in \subsystem$.
    We assume in addition that $F(\domF) = S^2$. 
    Let $d$ be a visual metric on $S^2$ for $f$ and $\potential$ be a real-valued \holder continuous function on $S^2$ with respect to the metric $d$.
    Then there exists a Borel probability measure $\eigmea = (m_{\black}, m_{\white}) \in \mathcal{P}(\splitsphere)$ such that
    \begin{equation}    \label{eq:eigenmeasure for subsystem}
        \dualsplopt\spleigmea = \eigenvalue \spleigmea,
    \end{equation}
    where $\eigenvalue = \big\langle \dualsplopt\spleigmea, \indicator{\splitsphere} \big\rangle$. 
    Moreover, if $F$ is strongly irreducible, then any $\eigmea = \spleigmea \in \mathcal{P}(\splitsphere)$ that satisfies \eqref{eq:eigenmeasure for subsystem} for some $\eigenvalue > 0$ has the following properties:
    \begin{enumerate}[label=\rm{(\roman*)}]
        \smallskip

        \item     \label{item:thm:subsystem:eigenmeasure existence and basic properties:strongly irreducible edge zero measure}
        $\eigmea \bigl( \bigcup_{j = 0}^{+\infty} f^{-j}(\mathcal{C}) \bigr) = 0$.

        \smallskip

        \item     \label{item:thm:subsystem:eigenmeasure existence and basic properties:Jacobian}
        For each Borel set $A \subseteq \domF$ on which $F$ is injective, $\eigmea(F(A)) = \int_{A} \! \eigenvalue \myexp{-\phi} \,\mathrm{d}\mu$.

        \smallskip 
        
        \item     \label{item:thm:subsystem:eigenmeasure existence and basic properties:Jacobian:Gibbs property}
        The measure $\eigmea$ is a Gibbs measure with respect to $F$, $\mathcal{C}$, and $\phi$ with $P_{\mu} = \pressure = \log \eigenvalue$. 
        Here $\pressure$ is defined in \eqref{eq:pressure of subsystem}. 
    \end{enumerate}
\end{theorem}
We follow the conventions discussed in Remarks~\ref{rem:disjoint union} and \ref{rem:probability measure in split setting}.
In particular, we use the notation $\spleigmea$ (\resp $\eigmea$) to emphasize that we treat the eigenmeasure as a Borel probability measure on $\splitsphere$ (\resp $S^2$).
The existence of eigenmeasure in Theorem~\ref{thm:subsystem:eigenmeasure existence and basic properties} was established in \cite[Theorem~6.16]{shi2024thermodynamic}.
Theorem~\ref{thm:subsystem:eigenmeasure existence and basic properties}~\ref{item:thm:subsystem:eigenmeasure existence and basic properties:strongly irreducible edge zero measure} is part of \cite[Proposition~6.27]{shi2024thermodynamic}.
Theorem~\ref{thm:subsystem:eigenmeasure existence and basic properties}~\ref{item:thm:subsystem:eigenmeasure existence and basic properties:Jacobian} and \ref{item:thm:subsystem:eigenmeasure existence and basic properties:Jacobian:Gibbs property} follow immediately from \cite[Proposition~6.29]{shi2024thermodynamic}.

The next theorem was established in \cite[Theorem~6.25]{shi2024thermodynamic}.
\begin{theorem}[Z.~Li, X.~Shi, Y.~Zhang \cite{shi2024thermodynamic}]    \label{thm:existence of f invariant Gibbs measure}
    Let $f \colon S^2 \mapping S^2$ be an expanding Thurston map with a Jordan curve $\mathcal{C}\subseteq S^2$ satisfying $\post{f} \subseteq \mathcal{C}$ and $f(\mathcal{C}) \subseteq \mathcal{C}$.
    Let $F \in \subsystem$ be strongly irreducible.
    Let $d$ be a visual metric on $S^2$ for $f$ with expansion factor $\Lambda > 1$.
    Let $\potential \in \holderspacesphere$ be a real-valued \holder continuous function with an exponent $\holderexp \in (0, 1]$.
    Then the sequence $\bigl\{ \frac{1}{n} \sum_{j = 0}^{n - 1} \normsplopt^j\bigl(\indicator{\splitsphere}\bigr) \bigr\}_{n \in \n}$ converges uniformly to a function $\eigfun = \splfun \in \splholderspace$, which satisfies
    \begin{align}
        \label{eq:eigenfunction of normed split operator}
        \normsplopt(\eigfun) &= \eigfun \qquad \text{and} \\
        \label{eq:two sides bounds for eigenfunction}
        \Csplratio^{-1} \leqslant \eigfun(\widetilde{x}) &\leqslant \Csplratio, \qquad \text{for each } \widetilde{x} \in \splitsphere,
    \end{align}
    where $\Csplratio \geqslant 1$ is the constant from Lemma~\ref{lem:distortion lemma for subsystem}~\ref{item:lem:distortion lemma for subsystem:uniform bound} depending only on $f$, $\mathcal{C}$, $d$, $\phi$, and $\holderexp$. 
    Moreover, if we let $\eigmea = \spleigmea$ be an eigenmeasure from Theorem~\ref{thm:subsystem:eigenmeasure existence and basic properties}, then
    \begin{equation}    \label{eq:integration of eigenfunction with respect to eigenmeasure is one}
        \int_{\splitsphere} \! \eigfun \,\mathrm{d}\spleigmea = 1,
    \end{equation}
    and $\equstate = \splmea \define \eigfun \spleigmea$ is an $f$-invariant Gibbs measure with respect to $F$, $\mathcal{C}$, and $\phi$, with $\equstate(\limitset(F, \mathcal{C})) = 1$ and
    \begin{equation}    \label{eq:equalities for characterizations of pressure}
        P_{\equstate} = P_{\eigmea} = \pressure = \lim_{n \to +\infty} \frac{1}{n}\log \bigl( \splopt^{n}\bigl(\indicator{\splitsphere}\bigr)(\widetilde{y}) \bigr),
    \end{equation}
    for each $\widetilde{y} \in \splitsphere$. 
    In particular, $\equstate(U) \ne 0$ for each open set $U \subseteq S^{2}$ with $U \cap \limitset(F, \mathcal{C}) \ne \emptyset$.
\end{theorem}

We record the following Variational Principle and the existence of equilibrium states for subsystems established in \cite[Theorems~6.30 and 6.31]{shi2024thermodynamic}.
Recall that $F(\limitset) \subseteq \limitset$ by Proposition~\ref{prop:subsystem:preliminary properties}~\ref{item:subsystem:properties:limitset forward invariant}.

\begin{theorem}[Z.~Li, X.~Shi, Y.~Zhang \cite{shi2024thermodynamic}]    \label{thm:subsystem characterization of pressure and existence of equilibrium state}
    Let $f \colon S^2 \mapping S^2$ be an expanding Thurston map with a Jordan curve $\mathcal{C}\subseteq S^2$ satisfying $\post{f} \subseteq \mathcal{C}$ and $f(\mathcal{C}) \subseteq \mathcal{C}$.
    Let $F \in \subsystem$ be strongly irreducible.
    Let $d$ be a visual metric on $S^2$ for $f$ and $\potential$ be a real-valued \holder continuous function on $S^2$ with respect to the metric $d$.
    Then we have
    \begin{equation}    \label{eq:subsystem Variational Principle}
        \pressure = P(F|_{\limitset}, \potential|_{\limitset}) = \sup \Bigl\{ h_{\mu}(F|_{\limitset}) + \int_{\limitset} \! \phi \,\mathrm{d}\mu \describe \mu \in \mathcal{M}(\limitset, F|_{\limitset}) \Bigr\},
    \end{equation}
    and there exists an equilibrium state for $F|_{\limitset}$ and $\phi|_{\limitset}$,
    where $\pressure$ is defined by \eqref{eq:pressure of subsystem} and $P(F|_{\limitset}, \potential|_{\limitset})$ is defined by \eqref{eq:def:topological pressure}.
    
    Moreover, any measure $\equstate \in \mathcal{M}(S^{2}, f)$ defined in Theorem~\ref{thm:existence of f invariant Gibbs measure} is an equilibrium state for $F|_{\limitset}$ and $\phi|_{\limitset}$, and the map $F|_{\limitset}$ with respect to such $\equstate$ is forward quasi-invariant (i.e., for each Borel set $A \subseteq \limitset$, if $\equstate(A) = 0$, then $\equstate((F|_{\limitset})(A)) = 0$) and non-singular (i.e., for each Borel set $A \subseteq \limitset$, if $\equstate(A) = 0$, then $\equstate((F|_{\limitset})^{-1}(A)) = 0$).
\end{theorem} 
\section{The Assumptions}
\label{sec:The Assumptions}

We state below the hypotheses under which we will develop our theory in most parts of this paper. 
We will selectively use some of those assumptions in the later sections. 

\begin{assumptions}
\quad
    \begin{enumerate}[label=\textrm{(\arabic*)}]
        \smallskip

        \item \label{assumption:expanding Thurston map}
            $f \colon S^2 \mapping S^2$ is an expanding Thurston map.

        \smallskip

        \item \label{assumption:Jordan curve}
            $\mathcal{C} \subseteq S^2$ is a Jordan curve containing $\post{f}$ with the property that there exists an integer $n_{\mathcal{C}} \in \n$ such that $f^{n_{\mathcal{C}}}(\mathcal{C}) \subseteq \mathcal{C}$ and $f^m(\mathcal{C}) \not\subseteq \mathcal{C}$ for each $m \in \oneton[n_{\mathcal{C}} - 1]$.
        
        \smallskip

        \item \label{assumption:subsystem}
            $F \in \subsystem$ is a subsystem of $f$ with respect to $\mathcal{C}$.
        
        \smallskip

        \item \label{assumption:visual metric and expansion factor}
            $d$ is a visual metric on $S^2$ for $f$ with expansion factor $\Lambda > 1$.

        \smallskip

        \item \label{assumption:holder exponent}
        $\holderexp \in (0, 1]$.

        \smallskip

        \item \label{assumption:holder potential}
        $\potential \in C^{0,\holderexp}(S^2, d)$ is a real-valued H\"{o}lder continuous function with an exponent $\holderexp$.




    \end{enumerate}
\end{assumptions}

    
Observe that by Lemma~\ref{lem:invariant_Jordan_curve}, for each $f$ in \ref{assumption:expanding Thurston map}, there exists at least one Jordan curve $\mathcal{C}$ that satisfies \ref{assumption:Jordan curve}. 
Since for a fixed $f$, the number $n_{\mathcal{C}}$ is uniquely determined by $\mathcal{C}$ in \ref{assumption:Jordan curve}, in the remaining part of the paper, we will say that a quantity depends on $\mathcal{C}$ even if it also depends on $n_{\mathcal{C}}$.

Recall that the expansion factor $\Lambda$ of a visual metric $d$ on $S^2$ for $f$ is uniquely determined by $d$ and $f$. We will say that a quantity depends on $f$ and $d$ if it depends on $\Lambda$.

Note that even though the value of $L$ is not uniquely determined by the metric $d$, in the remainder of this paper, for each visual metric $d$ on $S^2$ for $f$, we will fix a choice of linear local connectivity constant $L$. We will say that a quantity depends on the visual metric $d$ without mentioning the dependence on $L$, even though if we had not fixed a choice of $L$, it would have depended on $L$ as well.

In the discussion below, depending on the conditions we will need, we will sometimes say ``Let $f$, $\mathcal{C}$, $d$, $\potential$ satisfy the Assumptions.", and sometimes say ``Let $f$ and $\mathcal{C}$ satisfy the Assumptions.'', etc.  



\section{Uniqueness of the equilibrium states for subsystems}
\label{sec:Uniqueness of the equilibrium states for subsystems}

This section is devoted to the uniqueness of the equilibrium states for subsystems, with the main result being Theorem~\ref{thm:subsystem:uniqueness of equilibrium state}. 
We first define normalized split Ruelle operators and obtain some basic properties in Subsection~\ref{sub:Normalized split Ruelle operators}. 
Then in Subsection~\ref{sub:Uniform convergence} we introduce the notion of abstract modulus of continuity and prove the uniform convergence for functions under iterations of the normalized split Ruelle operators. 
Finally, in Subsection~\ref{sub:Uniqueness} we establish Theorem~\ref{thm:subsystem:uniqueness of equilibrium state}.

\begin{theorem}    \label{thm:subsystem:uniqueness of equilibrium state}
    Let $f$, $\mathcal{C}$, $F$, $d$, $\potential$ satisfy the Assumptions in Section~\ref{sec:The Assumptions}. 
    We assume in addition that $f(\mathcal{C}) \subseteq \mathcal{C}$ and $F \in \subsystem$ is strongly primitive.
    Denote $\limitset \define \limitset(F, \mathcal{C})$ and $\limitmap \define F|_{\limitset}$.
    Then there exists a unique equilibrium state $\equstate$ for $\limitmap$ and $\potential|_{\limitset}$.
    Moreover, the map $\limitmap$ with respect to $\equstate$ is forward quasi-invariant (i.e., for each Borel set $A \subseteq \limitset$, if $\equstate(A) = 0$, then $\equstate(\limitmap(A)) = 0$) and non-singular (i.e., for each Borel set $A \subseteq \limitset$, if $\equstate(A) = 0$, then $\equstate(\limitmap^{-1}(A)) = 0$).
\end{theorem}


We adopt the following convention.
\begin{remark}\label{rem:invariant measure equivalent}
    Let $X$ be a non-empty Borel subset of $S^{2}$.
    Given a Borel probability measure $\mu \in \probmea{X}$, by abuse of notation, we can view $\mu$ as a Borel probability measure on $S^{2}$ by setting $\mu(A) \define \mu(A \cap X)$ for all Borel subsets $A \subseteq S^{2}$. 
    Conversely, for each measure $\nu \in \probsphere$ supported on $X$, we can view $\nu$ as a Borel probability measure on $X$. 
\end{remark}

To prove the uniqueness of the equilibrium state of a continuous map $g$ on a compact metric space $X$, one of the techniques is to prove the (G\^ateaux) differentiability of the topological pressure function $P(g, \cdot \, ) \colon C(X) \mapping \real$. 
We summarize the general ideas below, but refer the reader to \cite[Section~3.6]{przytycki2010conformal} for a detailed treatment.

For a continuous map $g \colon X \mapping X$ on a compact metric space $X$, the topological pressure function $P(g, \cdot \, ) \colon C(X) \mapping \real$ is Lipschitz continuous \cite[Theorem~3.6.1]{przytycki2010conformal} and convex \cite[Theorem~3.6.2]{przytycki2010conformal}. For an arbitrary convex continuous function $Q \colon V \mapping \real$ on a real topological vector space $V$, we call a continuous linear functional $L \colon V \mapping \real$ \emph{tangent to $Q$ at $x \in V$} if 
\begin{equation}     \label{eq:def:tangent}
    Q(x) + L(y) \leqslant Q(x + y), \qquad \text{for each } y \in V.
\end{equation}
We denote the set of all continuous linear functionals tangent to $Q$ at $x \in V$ by $\tangent{V}{x}{Q}$. It is know (see for example, \cite[Proposition~3.6.6]{przytycki2010conformal}) that if $\mu \in \mathcal{M}(X, g)$ is an equilibrium state for $g$ and $\varphi \in C(X)$, then the continuous linear functional $u \mapsto \int \! u \,\mathrm{d}\mu$ for $u \in C(X)$ is tangent to the topological pressure function $P(g, \cdot \, )$ at $\varphi$. Indeed, let $\varphi, \gamma \in C(X)$ and $\mu \in \mathcal{M}(X, g)$ be an equilibrium state for $g$ and $\varphi$. Then $P(g, \varphi + \gamma) \geqslant h_{\mu}(g) + \int \! \varphi + \gamma \,\mathrm{d}\mu$ by the Variational Principle~\eqref{eq:Variational Principle for pressure} in Subsection~\ref{sub:thermodynamic formalism}, and $P(g, \varphi) = h_{\mu}(g) + \int \! \varphi \,\mathrm{d}\mu$. It follows that $P(g, \varphi) + \int \! \gamma \,\mathrm{d}\mu \leqslant P(g, \varphi + \gamma)$. 

Thus to prove the uniqueness of the equilibrium state for a continuous map $g \colon X \mapping X$ and a continuous potential $\varphi$, it suffices to show that $\operatorname{card}\bigl(\tangent{C(X)}{\varphi}{P(g, \cdot \, )} \bigr) = 1$. 
Then we can apply the following fact from functional analysis (see \cite[Theorem~3.6.5]{przytycki2010conformal} for a proof):

\begin{theorem}[F.~Przytycki \& M.~Urba\'{n}ski \cite{przytycki2010conformal}]  \label{thm:tengent}
    Let $V$ be a separable Banach space and $Q \colon V \mapping \real$ be a convex continuous function. Then for each $x \in V$, the following statements are equivalent:
    \begin{enumerate}
        \smallskip

        \item $\operatorname{card}{\!\bigl(\tangent{V}{x}{Q}\bigr)} = 1$.

        \smallskip

        \item The function $t \mapsto Q(x + ty)$ is differentiable at $0$ for each $y \in V$.

        \smallskip

        \item There exists a subset $U \subseteq V$ that is dense in the weak topology on $V$ such that the function $t \mapsto Q(x + ty)$ is differentiable at $0$ for each $y \in U$.
    \end{enumerate}
\end{theorem}

Now the problem of the uniqueness of equilibrium states transforms to the problem of (G\^ateaux) differentiability of the topological pressure function. To investigate the latter, we need a closer study of the fine properties of split Ruelle operators. 

\subsection{Normalized split Ruelle operator}%
\label{sub:Normalized split Ruelle operators}

In this subsection, we define normalized split Ruelle operators and establish some basic properties which will be frequently used later.

\smallskip


Let $f \colon S^2 \mapping S^2$ be an expanding Thurston map with a Jordan curve $\mathcal{C}\subseteq S^2$ satisfying $\post{f} \subseteq \mathcal{C}$ and $f(\mathcal{C}) \subseteq \mathcal{C}$. 
Let $d$ be a visual metric for $f$ on $S^2$, and $\phi \in C^{0,\holderexp}(S^2,d)$ a real-valued \holder continuous function with an exponent $\holderexp \in (0,1]$. 
Let $\juxtapose{X^0_{\black}}{X^0_{\white}} \in \mathbf{X}^0(f, \mathcal{C})$ be the black $0$-tile and the white $0$-tile, respectively. 
Recall from Definition~\ref{def:split sphere} and Remark~\ref{rem:disjoint union} that $\splitsphere = X^0_{\black} \sqcup X^0_{\white}$ is the disjoint union of $X^0_{\black}$ and $X^0_{\white}$, and the product spaces $\splfunspace$, $\splboufunspace$, and $\splmeaspace$ are identified with $C(\splitsphere)$, $B(\splitsphere)$, and $\mathcal{M}(\splitsphere)$, respectively.

Let $F \in \subsystem$ be strongly irreducible and $\eigfun = \splfun \in C(\splitsphere)$ be the function given by Theorem~\ref{thm:existence of f invariant Gibbs measure}.
Note that $\eigfun(\widetilde{x}) > 0$ for each $\widetilde{x} \in \splitsphere = \disjointuniontile$ by \eqref{eq:two sides bounds for eigenfunction} in Theorem~\ref{thm:existence of f invariant Gibbs measure}.
Recall from \eqref{eq:def:normed potential} that $\normpotential = \phi - \pressure$.

\begin{definition}[Partial normalized split Ruelle operator]    \label{def:partial normalized split Ruelle operator}
    Let $f$, $\mathcal{C}$, $F$, $d$, $\potential$ satisfy the Assumptions in Section~\ref{sec:The Assumptions}.
    We assume in addition that $f(\mathcal{C}) \subseteq \mathcal{C}$ and $F \in \subsystem$ is strongly irreducible. 
    For each $n \in \n_0$ and each pair of $\juxtapose{\colour}{\colour'} \in \colours$, we define a map $\toperator{n}{\colour}{\colour'} \colon \boufunspace{\colour'} \mapping B(X^{0}_{\colour})$ by
    \begin{equation}    \label{eq:def:twisted partial Ruelle operator}
        \begin{split}
           & \toperator{n}{\colour}{\colour'}(v)(x) \\
            &\qquad\define \sum_{ y \in F^{-n}(x) } \ccndegF{\colour}{\colour'}{n}{y} v(y) \myexp[\big]{ S_{n}^{F}\phi(y) - n\pressure + \log u_{\colour'}(y) - \log u_{\colour}(x) }  \\
            &\qquad= \sum_{ \substack{ y = (F^{n}|_{X^n})^{-1}(x) \\ X^n \in \ccFTile{n}{\colour}{\colour'}} } v(y) \myexp[\big]{ S_{n}^{F}\phi(y) - n\pressure + \log u_{\colour'}(y) - \log u_{\colour}(x) }
        \end{split}
    \end{equation}
    for each $v \in \boufunspace{\colour'}$ and each point $x \in X^0_{\colour}$. 

\end{definition} 

\begin{rmk}
    By \eqref{eq:definition of partial Ruelle operators} in Definition~\ref{def:partial Ruelle operator}, we can write the right hand side of \eqref{eq:def:twisted partial Ruelle operator} as
    \begin{equation}\label{eq:equivalent definition of twisted partial Ruelle operator}
        \begin{split}        
            &\toperator{n}{\colour}{\colour'}(v)(x) \\
            &\qquad= \frac{1}{u_{\colour}(x)} \sum_{ X^n \in \ccFTile{n}{\colour}{\colour'} } 
            (u_{\colour'} \cdot v) \bigl((F^{n}|_{X^n})^{-1}(x) \bigr) \myexp[\big]{ S_{n}^{F}\phi\bigl( (F^{n}|_{X^n})^{-1}(x) \bigr) -n\pressure } \\
            &\qquad= \frac{1}{u_{\colour}(x)} \poperator[\normpotential]{n}(u_{\colour'} v)(x).
        \end{split}
    \end{equation}
    Thus, it follows immediately from Lemma~\ref{lem:well-defined partial Ruelle operator} that for all $\juxtapose{n}{k} \in \n_0$, $\juxtapose{\colour}{\colour'} \in \colours$, and $v \in \confunspace{\colour'}$, 
    \begin{equation}    \label{eq:well-defined continuity of twisted partial Ruelle operator}
        \, \toperator{n}{\colour}{\colour'} (v) \in C\bigl(X^{0}_{\colour}\bigr) \qquad\qquad \text{and }
    \end{equation}
    \begin{equation}    \label{eq:iteration of twisted partial Ruelle operator}
        \qquad \qquad  \toperator{n + k}{\colour}{\colour'} (v) = \sum_{\colour'' \in \colours} \toperator{n}{\colour}{\colour''} \bigl( \toperator{k}{\colour''}{\colour'}(v) \bigr).
    \end{equation}
\end{rmk}

\begin{definition}[Normalized split Ruelle operators]    \label{def:normalized split ruelle operator}
    Let $f$, $\mathcal{C}$, $F$, $d$, $\phi$ satisfy the Assumptions in Section~\ref{sec:The Assumptions}. 
    We assume in addition that $f(\mathcal{C}) \subseteq \mathcal{C}$ and $F \in \subsystem$ is strongly irreducible. 
    Let $\eigfun = \splfun \in C(\splitsphere)$ be the continuous function given by Theorem~\ref{thm:existence of f invariant Gibbs measure}. 
    The \emph{normalized split Ruelle operator} $\twistsplopt \colon \splfunspace \mapping \splfunspace$ for the subsystem $F$ and the potential $\phi$ is defined by
    \begin{equation}    \label{eq:def:normalized split Ruelle operator expand form}
        \twistsplopt\splfunv \define \bigl( 
            \toperator{1}{\black}{\black}(u_{\black}) + \toperator{1}{\black}{\white}(u_{\white}), 
            \toperator{1}{\white}{\black}(u_{\black}) + \toperator{1}{\white}{\white}(u_{\white})
        \bigr)
    \end{equation}
    for each $v_{\black} \in C\bigl(X^0_{\black}\bigr)$ and each $v_{\white} \in C\bigl(X^0_{\white}\bigr)$, or equivalently, $\twistsplopt \colon C(\splitsphere) \mapping C(\splitsphere)$ is defined by 
    \begin{equation}    \label{eq:def:normalized split Ruelle operator}
        \twistsplopt(\widetilde{v}) \define \frac{1}{\eigfun} \normsplopt(\eigfun \widetilde{v} )
    \end{equation} 
    for each $\widetilde{v} \in C(\splitsphere)$.    
\end{definition}

Note that by \eqref{eq:def:twisted partial Ruelle operator} and \eqref{eq:well-defined continuity of twisted partial Ruelle operator}, the normalized split Ruelle operator $\twistsplopt$ is well-defined, and the equivalence of the two definitions \eqref{eq:def:normalized split Ruelle operator expand form} and \eqref{eq:def:normalized split Ruelle operator} follows from \eqref{eq:equivalent definition of twisted partial Ruelle operator} and Definition~\ref{def:split ruelle operator}.
By \eqref{eq:def:twisted partial Ruelle operator}, $\twistsplopt^{0}$ is the identity map on $C(\splitsphere)$.
Moreover, one sees that $\twistsplopt \colon \splfunspace \mapping \splfunspace$ has a natural extension to the space $\splboufunspace$ given by \eqref{eq:def:normalized split Ruelle operator expand form} for each $v_{\black} \in B\bigl(X^0_{\black}\bigr)$ and each $v_{\white} \in B\bigl(X^0_{\white}\bigr)$. 

We show that the normalized split Ruelle operator $\twistsplopt$ is well-behaved under iterations.

\begin{lemma}    \label{lem:iteration of normalized split ruelle operator}
    Let $f$, $\mathcal{C}$, $F$, $d$, $\phi$ satisfy the Assumptions in Section~\ref{sec:The Assumptions}. 
    We assume in addition that $f(\mathcal{C}) \subseteq \mathcal{C}$ and $F \in \subsystem$ is strongly irreducible.
    Then for all $n \in \n_0$ and $\widetilde{v} = \splfunv \in C(\splitsphere)$, we have
    \begin{equation}    \label{eq:iteration of normalized split ruelle operator}
        \twistsplopt^{n}(\widetilde{v}) = \frac{1}{\eigfun} \normsplopt^{n}(\eigfun \widetilde{v} ) \qquad\quad \text{ and }
    \end{equation}
    \begin{equation}    \label{eq:iteration of normalized split ruelle operator expand form}
        \twistsplopt^{n} \splfunv = \bigl( \toperator{n}{\black}{\black}(v_{\black}) + \toperator{n}{\black}{\white}(v_{\white}), \
        \toperator{n}{\white}{\black}(v_{\black}) + \toperator{n}{\white}{\white}(v_{\white}) \bigr).
    \end{equation}
\end{lemma}
\begin{proof}
    It follows immediately from \eqref{eq:def:normalized split Ruelle operator} and Definition~\ref{def:split ruelle operator} that \eqref{eq:iteration of normalized split ruelle operator} holds for all $n \in \n$. 
    Since $F$ is surjective, i.e., $F(\domF) = S^{2}$, we know that $\splopt^{0}$ is the identity map on $C(\splitsphere)$ by Definition~\ref{def:split ruelle operator}. 
    Thus \eqref{eq:iteration of normalized split ruelle operator} also holds for $n = 0$.

    The case where $n = 0$ and the case where $n = 1$ both hold by definition. 
    Assume now \eqref{eq:iteration of normalized split ruelle operator expand form} holds for $n = k$ for some $k \in \n$. 
    Then by Definition~\ref{def:normalized split ruelle operator} and \eqref{eq:iteration of twisted partial Ruelle operator}, for each $\colour \in \colours$ we have
    \begin{align*}
    \pi_{\colour}\bigl( \twistsplopt^{k + 1}\splfunv \bigr) 
        &= \pi_{\colour}\bigl( 
            \twistsplopt\bigl( 
                \toperator{k}{\black}{\black}(v_{\black}) + \toperator{k}{\black}{\white}(v_{\white}), \
                \toperator{k}{\white}{\black}(v_{\black}) + \toperator{k}{\white}{\white}(v_{\white}) 
            \bigr) 
        \bigr) \\
        &= \sum_{\colour' \in \colours} \toperator{1}{\colour}{\colour'} \bigl( 
                \toperator{k}{\colour'}{\black}(v_{\black}) + \toperator{k}{\colour'}{\white}(v_{\white}) 
            \bigr)  \\
        &= \sum_{\colour'' \in \colours} \sum_{\colour' \in \colours} \toperator{1}{\colour}{\colour'} \bigl( \toperator{k}{\colour'}{\colour''}(v_{\colour''}) \bigr)  \\
        &= \sum_{\colour'' \in \colours} \toperator{k + 1}{\colour}{\colour''}(v_{\colour''}),
    \end{align*}
    for $v_{\black} \in C\bigl(X^0_{\black}\bigr)$ and $v_{\white} \in C\bigl(X^{0}_{\white}\bigr)$. 
    This completes the inductive step, establishing \eqref{eq:iteration of normalized split ruelle operator expand form}.
\end{proof}

\begin{rmk}
    Similarly, one can show that \eqref{eq:iteration of normalized split ruelle operator} and \eqref{eq:iteration of normalized split ruelle operator expand form} holds for $\widetilde{v} = \splfunv \in \splboufunspace$.
\end{rmk}

Recall $\spleigmea \in \probmea{\splitsphere}$ from Theorem~\ref{thm:subsystem:eigenmeasure existence and basic properties}. 
By Theorem~\ref{thm:subsystem:eigenmeasure existence and basic properties}~\ref{item:thm:subsystem:eigenmeasure existence and basic properties:Jacobian:Gibbs property}, $\normdualsplopt\spleigmea = \spleigmea$.
Then we can show that $\splmea = \eigfun \spleigmea \in \probmea{\splitsphere}$ defined in Theorem~\ref{thm:existence of f invariant Gibbs measure} satisfies
\begin{equation}  \label{eq:equilibrium state is eigenmeasure of dual normalized split operator}
    \dualtwistsplopt\splmea = \splmea,
\end{equation}
where $\dualtwistsplopt$ is the adjoint operator of $\twistsplopt$ on the space $\splmeaspace$.
Indeed, by \eqref{eq:def:normalized split Ruelle operator}, for every $\widetilde{v} \in C(\splitsphere)$,
\begin{align*}
    \functional[\big]{\dualtwistsplopt\splmea}{\widetilde{v}} 
    &= \functional[\big]{\eigfun\spleigmea}{\twistsplopt(\widetilde{v})} = \functional[\big]{\spleigmea}{ \normsplopt(\eigfun \widetilde{v})} \\
    &= \functional[\big]{ \normdualsplopt\spleigmea}{\eigfun \widetilde{v}} = \functional{\spleigmea}{\eigfun \widetilde{v}} = \functional{\splmea}{\widetilde{v}}.
\end{align*}

Recall that we equip the spaces $C(\splitsphere)$ and $\splfunspace$ with the uniform norm given by\[
    \normcontinuous{\widetilde{v}}{\splitsphere} = \norm{\splfunv} 
    = \max \bigl\{ \normcontinuous{v_{\black}}{X^0_{\black}}, \normcontinuous{v_{\white}}{X^{0}_{\white}} \bigr\}
\] for $\widetilde{v} = \splfunv \in \splfunspace$.

\begin{lemma}   \label{lem:norm of normalized split operator}
    Let $f$, $\mathcal{C}$, $F$, $d$, $\phi$ satisfy the Assumptions in Section~\ref{sec:The Assumptions}. 
    We assume in addition that $f(\mathcal{C}) \subseteq \mathcal{C}$ and $F \in \subsystem$ is strongly irreducible.
    Then the operator norm of $\twistsplopt$ is     $\normcontinuous[\big]{\twistsplopt}{\splitsphere} = 1$. 
    In addition, $\twistsplopt\bigl(\indicator{\splitsphere}\bigr) = \indicator{\splitsphere}$.
\end{lemma}
\begin{proof}
    By Definition~\ref{def:normalized split ruelle operator}, \eqref{eq:eigenfunction of normed split operator} in Theorem~\ref{thm:existence of f invariant Gibbs measure}, and \eqref{eq:def:twisted partial Ruelle operator} in Definition~\ref{def:partial normalized split Ruelle operator}, we have
    \[
        \twistsplopt\bigl(\indicator{\splitsphere}\bigr) = \frac{1}{\eigfun} \normsplopt(\eigfun) = \indicator{\splitsphere},
    \]
    i.e., for each $\widetilde{x} = (x, \colour) \in \splitsphere$,
    \[
        \begin{split}
            1 = \sum_{\colour' \in \colours} \sum_{ X^1 \in \ccFTile{1}{\colour}{\colour'} } \myexp[\big]{ \phi ( x_{X^{1}} ) - \pressure + \log u_{\colour'} ( x_{X^{1}} ) - \log u_{\colour}(x) },
        \end{split}
    \]
    where we write $x_{X^{1}} \define (F|_{X^{1}})^{-1}(x)$ for $X^1 \in \ccFTile{1}{\colour}{\colour'}$. 
    Then for each $\widetilde{x} = (x, \colour) \in \splitsphere$ and each $\widetilde{v} = \splfunv \in \splfunspace$, we have
    \begin{align*}
        \Bigl| \twistsplopt(\widetilde{v})(\widetilde{x}) \Bigl|
        &= \bigg| \sum_{\colour' \in \colours} \sum_{  X^1 \in \ccFTile{1}{\colour}{\colour'} } v_{\colour'}( x_{X^{1}} ) \myexp[\big]{ \phi( x_{X^{1}} ) - \pressure + \log u_{\colour'}( x_{X^{1}} ) - \log u_{\colour}(x) } \bigg|  \\
        &\leqslant \normcontinuous{\widetilde{v}}{\splitsphere}  \abs[\bigg]{ \sum_{\colour' \in \colours} \sum_{ X^1 \in \ccFTile{1}{\colour}{\colour'} } \myexp[\big]{ \phi( x_{X^{1}} ) - \pressure + \log u_{\colour'}( x_{X^{1}} ) - \log u_{\colour}(x) } }  \\
        &= \normcontinuous{\widetilde{v}}{\splitsphere}.
    \end{align*}
    Thus, $\normcontinuous[\big]{\twistsplopt}{\splitsphere} \leqslant 1$. 
    Since $\twistsplopt\bigl(\indicator{\splitsphere}\bigr) = \indicator{\splitsphere}$, we get $\normcontinuous[\big]{\twistsplopt}{\splitsphere} = 1$.
\end{proof}

\subsection{Uniform convergence}
\label{sub:Uniform convergence}

In this subsection, we prove the uniform convergence for functions under iterations of the normalized split Ruelle operators. 

\smallskip

Let $(X, d)$ be a metric space. A function $\modufun \colon [0, +\infty) \mapping [0, +\infty)$ is an \emph{abstract modulus of continuity} if it is continuous at $0$, non-decreasing, and $\modufun(0) = 0$.  
Given any constant $\moduconst \in [0, +\infty]$ and any abstract modulus of continuity $\modufun$, we define the subclass $\moduspace$ of $C(X)$ as
\begin{align*}
    \moduspace \define \bigl\{ u \in C(X) \describe \uniformnorm{u} \leqslant \moduconst \text{ and for } \juxtapose{x}{y} \in S^2, \, \abs{u(x) - u(y)} \leqslant \modufun(d(x, y)) \bigr\}.
\end{align*}

Assume now that $(X, d)$ is compact. Then by the \aalem Theorem, each $\moduspace$ is precompact in $C(X)$ equipped with the uniform norm. It is easy to see that each $\moduspace$ is compact. 
On the other hand, for $v \in C(X)$, we can define an abstract modulus of continuity by
\begin{equation}    \label{eq:abstract modulus of continuity for continuous function}
    \modufun(t) = \sup \{ \abs{v(x) - v(y)} \describe \juxtapose{x}{y} \in X, \, d(x, y) \leqslant t \}
\end{equation}
for $t \in [0, +\infty)$, so that $v \in \moduspace[\iota]$, where $\iota = \uniformnorm{v}$.

The following lemma is easy to check (see also \cite[Lemma~5.24]{li2017ergodic}).
\begin{lemma}    \label{lem:product and inverse of modulus of continuity}
    Let $(X, d)$ be a metric space. 
    For all constants $\moduconst > 0$, $\moduconst_1, \, \moduconst_2 \geqslant 0$ and abstract moduli of continuity $\modufun_1, \, \modufun_2$, we have
    \[
        \begin{split}
            \bigl\{ v_1 v_2 \describe v_1 \in \moduspace[\moduconst_{1}][\modufun_{1}], \, v_2 \in \moduspace[\moduconst_{2}][\modufun_{2}] \bigr\} &\subseteq \moduspace[\moduconst_{1}\moduconst_{2}][\moduconst_{1}\modufun_{2} + \moduconst_{2}\modufun_{1}], \\
            \bigl\{ 1 / v \describe v \in \moduspace[\moduconst_{1}][\modufun_{1}] , \, v(x) \geqslant \moduconst \text{ for each } x \in X \bigr\}  &\subseteq \moduspace[\moduconst^{-1}][\moduconst^{-2} \modufun_{1}] .
        \end{split}
    \]    
\end{lemma}

\begin{proposition}   \label{prop:modulus for function under split ruelle operator}
    Let $f$, $\mathcal{C}$, $F$, $d$, $\Lambda$ satisfy the Assumptions in Section~\ref{sec:The Assumptions}. 
    We assume in addition that $f(\mathcal{C}) \subseteq \mathcal{C}$ and $F \in \subsystem$ is strongly irreducible.
    Then for each $\holderexp \in (0,1]$, each $\moduconst \geqslant 0$, and each $K \geqslant 0$, there exist constants $\widehat{\moduconst} \geqslant 0$ and $\widehat{C} \geqslant 0$ with the following property:

    For each abstract modulus of continuity $\modufun$, there exists an abstract modulus of continuity $\widetilde{\modufun}$ such that for each $\potential \in \holderspacesphere$ with $\holdernorm{\phi}{S^2} \leqslant K$, we have
    \begin{align}
        \label{eq:modulus for function under normed split Ruelle operator}
        \bigl\{ \normsplopt^{n}(\widetilde{v}) \describe \widetilde{v}  \in \splmoduspace, \, n \in \n_0 \bigr\} &\subseteq \splmoduspace[\widehat{\moduconst}][\widehat{\modufun}],  \\
        \label{eq:modulus for function under normalized split Ruelle operator}
        \bigl\{ \twistsplopt^{n}(\widetilde{v}) \describe \widetilde{v}  \in \splmoduspace, \, n \in \n_0 \bigr\} &\subseteq \splmoduspace[\moduconst][\widetilde{\modufun}],
    \end{align}
    where $\widehat{\modufun}(t) \define \widehat{C}\bigl(t^{\holderexp} + \modufun(C_0 t)\bigr)$ is an abstract modulus of continuity, and $C_0 > 1$ is the constant depending only on $f$, $\mathcal{C}$, and $d$ from Lemma~\ref{lem:basic_distortion}.
\end{proposition}
\begin{proof}
    \def\moduspl#1#2{\widetilde{C}^{#1}_{#2}(\splitsphere, d)}
    We write $\moduspl{\moduconst}{\modufun} \define \splmoduspace[\moduconst][\modufun]$ for each $\moduconst > 0$ and each abstract modulus of continuity $\modufun$ in this proof. For each $\widetilde{v} \in \moduspl{\moduconst}{\modufun}$, write $\widetilde{v} = \splfunv$.

    Fix arbitrary $\holderexp \in (0,1]$, $\moduconst \geqslant 0$, and $K \geqslant 0$. 
    By Lemma~\ref{lem:iteration of split-partial ruelle operator} and \eqref{eq:second bound for normed split operator} in Lemma~\ref{lem:distortion lemma for normed split operator}, for all $n \in \n_0$, $\widetilde{v} = \splfunv \in \moduspl{\moduconst}{\modufun}$, and $\potential \in \holderspacesphere$ with $\holdernorm{\potential}{S^{2}} \leq K$, we have\[
        \normcontinuous[\big]{\normsplopt^{n}(\widetilde{v})}{\splitsphere} 
        \leqslant \normcontinuous{\widetilde{v}}{\splitsphere} \normcontinuous[\big]{\normsplopt^{n}(\indicator{\splitsphere})}{\splitsphere} 
        \leqslant \Csplratio \normcontinuous{\widetilde{v}}{\splitsphere},
    \]
    where $\Csplratio \geqslant 1$ is the constant defined in \eqref{eq:const:Csplratio} in Lemma~\ref{lem:distortion lemma for subsystem}~\ref{item:lem:distortion lemma for subsystem:uniform bound} and depends only on $F$, $\mathcal{C}$, $d$, $\phi$, and $\holderexp$.
    By \eqref{eq:const:Csplratio}, quantitatively,
    \[
        \Csplratio = \CsplratioExpression, 
    \]
    where $n_{F} \in \n$ is the constant depending only on $F$ and $\mathcal{C}$ in Definition~\ref{def:primitivity of subsystem} since $F$ is primitive, and $C_1 \geqslant 0$ is the constant defined in \eqref{eq:const:C_1} in Lemma~\ref{lem:distortion_lemma} depends only on $F$, $\mathcal{C}$, $d$, $\phi$, and $\holderexp$. 
    Quantitatively,\[
        C_1 = C_0 \holderseminorm{\potential}{S^2} / ({1 - \Lambda^{-\holderexp}}),
    \]
    where $C_0 > 1$ is the constant depending only on $f$, $\mathcal{C}$, and $d$ in Lemma~\ref{lem:basic_distortion}.
    Let $C'_{1} \define C_0 K /(1 - \Lambda^{-\holderexp})$ and \[
        \Csplratio' \define (\deg{f})^{n_{F}} \myexp[\big]{ 2n_{F} K + C'_{1} (\diam{d}{S^{2}} )^{\holderexp} }.
    \]
    Then we have $C_1 \leqslant C'_1$ and $\Csplratio \leqslant \Csplratio'$ for each $\potential \in \holderspacesphere$ with $\holdernorm{\potential}{S^{2}} \leq K$.
    Note that both $C'_{1}$ and $\Csplratio'$ only depend on $F$, $\mathcal{C}$, $d$, $K$, and $\holderexp$. 
    Thus we can choose $\widehat{\moduconst} \define \Csplratio' \moduconst$. 

    For each $\colour \in \colours$ and each pair of $\juxtapose{x}{y} \in X^0_{\colour}$, we have
    \begin{align*}
        &\abs[\Big]{\pi_{\colour}\parentheses[\big]{ \normsplopt^{n}\splfunv } (x) - \pi_{\colour} \parentheses[\big]{ \normsplopt^{n}\splfunv } (y)}  \\
        &\qquad= \abs[\bigg]{  \paroperator[\normpotential]{n}{\colour}{\black}(v_{\black})(x) + \paroperator[\normpotential]{n}{\colour}{\white}(v_{\white})(x) 
            - \paroperator[\normpotential]{n}{\colour}{\black}(v_{\black})(y) - \paroperator[\normpotential]{n}{\colour}{\white}(v_{\white})(y)  } \\
        &\qquad\leqslant \sum_{\colour' \in \colours} \abs[\bigg]{\paroperator[\normpotential]{n}{\colour}{\colour'}(v_{\colour})(x) - \paroperator[\normpotential]{n}{\colour}{\colour'}(v_{\colour})(y) }  \\
        &\qquad= \sum_{\colour' \in \colours} \abs[\bigg]{ \sum_{X^n \in \ccFTile{n}{\colour}{\colour'}} 
            \bigl( v_{\colour'} \myexp[\big]{  S_n^{F}\normpotential } \bigr) \bigl( (F^{n}|_{X^{n}})^{-1}(x) \bigr) - \bigl( v_{\colour'} \myexp[\big]{ S_n^{F}\normpotential } \bigr) \bigl( (F^{n}|_{X^{n}})^{-1}(y) \bigr) }    \\ 
        &\qquad\leqslant \sum_{\colour' \in \colours} \abs[\bigg]{ \sum_{X^n \in \ccFTile{n}{\colour}{\colour'}} 
            v_{\colour'}\bigl( (F^{n}|_{X^{n}})^{-1}(x) \bigr) \Bigl( e^{  S_n^{F}\normpotential( (F^{n}|_{X^{n}})^{-1}(x) ) } - e^{ S_n^{F}\normpotential( (F^{n}|_{X^{n}})^{-1}(y) ) } \Bigr) }  \\
        &\qquad \qquad + \sum_{\colour' \in \colours} \abs[\bigg]{ \sum_{X^n \in \ccFTile{n}{\colour}{\colour'}} 
            e^{ S_n^{F}\normpotential( (F^{n}|_{X^{n}})^{-1}(y) ) }   \Bigl( v_{\colour'}\bigl( (F^{n}|_{X^{n}})^{-1}(x) \bigr) - v_{\colour'}\bigl( (F^{n}|_{X^{n}})^{-1}(y) \bigr) \Bigr)   } .
    \end{align*}
    The second term above is \[
        \leqslant \Csplratio \modufun(C_0 \Lambda^{-n} d(x, y)) \leqslant \Csplratio \modufun(C_0 d(x, y)) \leqslant \Csplratio' \modufun(C_0 d(x, y)),
    \]
    due to \eqref{eq:second bound for normed split operator} in Lemma~\ref{lem:distortion lemma for normed split operator} and the fact that $d\bigl( (F^{n}|_{X^{n}})^{-1}(x), (F^{n}|_{X^{n}})^{-1}(y) \bigr) \leqslant C_0 \Lambda^{-n} d(x, y)$ by Lemma~\ref{lem:basic_distortion}, where the constant $C_0$ comes from.

    To estimate the first term, we use the following general inequality for $s, t \in \real$,\[
        | \myexp{s} -\myexp{t} | \leqslant ( \myexp{s} + \myexp{t} ) ( \myexp{|s - t|} - 1).
    \]
    Then it follows from Lemma~\ref{lem:distortion_lemma}, Lemma~\ref{lem:iteration of split-partial ruelle operator}, and \eqref{eq:second bound for normed split operator} in Lemma~\ref{lem:distortion lemma for normed split operator} that the first term is 
    \begin{align*}
        &\leqslant \sum_{\colour' \in \colours} \sum_{X^n \in \ccFTile{n}{\colour}{\colour'}} \uniformnorm{v_{\colour'}} 
            \Bigl( e^{S_n^{F}\normpotential( (F^{n}|_{X^{n}})^{-1}(x) )} + e^{S_n^{F}\normpotential( (F^{n}|_{X^{n}})^{-1}(y) )} \Bigr)  \\
        &\qquad\qquad\qquad\qquad\qquad\qquad\!  \cdot \Bigl( e^{| S_n^{F}\normpotential( (F^{n}|_{X^{n}})^{-1}(x) ) - S_n^{F}\normpotential( (F^{n}|_{X^{n}})^{-1}(y) ) |} - 1 \Bigr)  \\
        &\leqslant \sum_{X^n \in \ccFTile{n}{\colour}{}} \moduconst \Bigl( e^{S_n^{F}\normpotential( (F^{n}|_{X^{n}})^{-1}(x) )} + e^{S_n^{F}\normpotential( (F^{n}|_{X^{n}})^{-1}(y) )} \Bigr) \bigl( \myexp[\big]{ C_1 d(x, y)^{\holderexp} } - 1\bigr) \\
        &\leqslant 2 \moduconst \Csplratio \bigl( \myexp[\big]{ C_1 d(x, y)^{\holderexp} } - 1\bigr) 
        \leqslant 2 \moduconst \Csplratio' \bigl( \myexp[\big]{ C'_{1} d(x, y)^{\holderexp} } - 1\bigr) \\
        &\leqslant 2 \moduconst \widetilde{C}_{1} d(x, y)^{\holderexp},
    \end{align*}
    for some constant $\widetilde{C}_{1} > 0$ that only depends on $C'_1$, $\Csplratio'$, and $\diam{d}{S^2}$.
    Here the justification of the third inequality above is similar to that of \eqref{eq:third bound for normed split operator} in Lemma~\ref{lem:distortion lemma for normed split operator}.
    Recall that both $C'_{1}$ and $\Csplratio'$ only depend on $F$, $\mathcal{C}$, $d$, $K$, and $\holderexp$, so does $\widetilde{C}_{1}$.

    Hence for each $\colour \in \colours$ and each pair of $\juxtapose{x}{y} \in X^0_{\colour}$, we have
     \[
        \Big| \pi_{\colour}\bigl( \normsplopt^{n}\splfunv \bigl)(x) - \pi_{\colour}\bigl( \normsplopt^{n}\splfunv \bigl)(y) \Big| \leqslant 
        \Csplratio' \modufun(C_0 d(x, y)) + 2 \moduconst \widetilde{C}_{1} d(x, y)^{\holderexp}.
    \]
    By choosing $\widehat{C} \define \max \bigl\{ \Csplratio', 2 \moduconst \widetilde{C}_{1} \bigr\}$, which only depends on $F$, $\mathcal{C}$, $d$, $K$, and $\holderexp$, we complete the proof of \eqref{eq:modulus for function under normed split Ruelle operator}.

    \smallskip

    We now prove \eqref{eq:modulus for function under normalized split Ruelle operator}.

    We fix an arbitrary $\potential \in \holderspacesphere$ with $\holdernorm{\potential}{S^{2}} \leq K$.
    Recall that $\eigfun = \splfun \in \splfunspace$ is a continuous function on $\splitsphere$ given by Theorem~\ref{thm:existence of f invariant Gibbs measure}. 
    Thus, by \eqref{eq:two sides bounds for eigenfunction} in Theorem~\ref{thm:existence of f invariant Gibbs measure} we have\[
        \normcontinuous{\eigfun}{\splitsphere} \leqslant \moduconst_{1},
    \]
    where $\moduconst_{1} \define \Csplratio' = (\deg{f})^{n_{F}} \myexp[\big]{ 2n_{F} K + C_{1}' (\diam{d}{S^{2}} )^{\holderexp} }$.
    For each $\colour \in \colours$ and each pair of $\juxtapose{x}{y} \in X^0_{\colour}$, it follows from Theorem~\ref{thm:existence of f invariant Gibbs measure} and \eqref{eq:third bound for normed split operator} in Lemma~\ref{lem:distortion lemma for normed split operator} that 
    \begin{align*}
        | u_{\colour}(x) - u_{\colour}(y)| 
        &= \bigg| \lim_{n \to +\infty} \frac{1}{n} \sum_{j=0}^{n-1} \Bigl( \normsplopt^{j}\bigl(\indicator{\splitsphere}\bigr)(x, \colour) - \normsplopt^{j}\bigl(\indicator{\splitsphere}\bigr)(y, \colour) \Bigr)  \bigg| \\
        &\leqslant \limsup_{n \to +\infty} \frac{1}{n} \sum_{j=0}^{n-1} \Big| \normsplopt^{j}\bigl(\indicator{\splitsphere}\bigr)(x, \colour) - \normsplopt^{j}\bigl(\indicator{\splitsphere}\bigr)(y, \colour) \Big| \\
        &\leqslant \Csplratio \bigl( \myexp[\big]{ C_1 d(x, y)^{\holderexp} } - 1\bigr) \\
        &\leqslant \Csplratio' \bigl( \myexp[\big]{ C'_1 d(x, y)^{\holderexp} } - 1\bigr).
    \end{align*}
    Thus, $\eigfun = \splfun \in \moduspl{\moduconst_{1}}{\modufun_{1}}$, where $\modufun_{1}$ is an abstract modulus of continuity defined by
    \[
        \modufun_{1}(t) \define \Csplratio' \bigl( \myexp[\big]{ C'_1 t^{\holderexp} } - 1\bigr), \quad \text{ for } t \in [0, +\infty).
    \]

    Note that it follows from Lemma~\ref{lem:product and inverse of modulus of continuity} that
    \[
        \begin{split}
            \bigl\{ \widetilde{v} \eigfun \describe \widetilde{v} = \splfunv & \in \moduspl{\moduconst}{\modufun}, \,  \potential \in \holderspacesphere, \, \holderseminorm{\potential}{S^2} \leqslant K \bigr\} 
            \subseteq \moduspl{\moduconst\moduconst_{1}}{\moduconst\modufun_{1} + \moduconst_{1}\modufun}.
        \end{split}
    \]
    Then by \eqref{eq:iteration of normalized split ruelle operator} in Lemma~\ref{lem:iteration of normalized split ruelle operator}, \eqref{eq:modulus for function under normed split Ruelle operator}, and Lemma~\ref{lem:product and inverse of modulus of continuity}, we get that there exists a constant $\widetilde{\moduconst} \geqslant 0$ and an abstract modulus of continuity $\widetilde{\modufun}$ such that for each $\potential \in \holderspacesphere$ with $\holderseminorm{\potential}{S^2} \leqslant K$, \[
        \bigl\{ \twistsplopt^{n}(\widetilde{v}) \describe \widetilde{v} = \splfunv \in \moduspl{\moduconst}{\modufun}, \, n \in \n_0 \bigr\} \subseteq \moduspl{\widetilde{\moduconst}}{\widetilde{\modufun}}.
    \] 

    On the other hand, by Lemma~\ref{lem:norm of normalized split operator}, $\normcontinuous[\big]{\twistsplopt^{n}(\widetilde{v})}{\splitsphere} \leqslant \normcontinuous{\widetilde{v}}{\splitsphere} \leqslant \moduconst$ for each $\widetilde{v} = \splfunv \in \moduspl{\moduconst}{\modufun}$, each $n \in \n_0$, and each $\potential \in \holderspacesphere$. 
    Therefore, we have proved \eqref{eq:modulus for function under normalized split Ruelle operator}.
\end{proof}

\begin{lemma}    \label{lem:upper bound for uniform norm of normalized split ruelle operator}
    Let $f$, $\mathcal{C}$, $F$, $d$, $\Lambda$ satisfy the Assumptions in Section~\ref{sec:The Assumptions}. 
    We assume in addition that $f(\mathcal{C}) \subseteq \mathcal{C}$ and $F \in \subsystem$ is strongly primitive.
    Let $\modufun$ be an abstract modulus of continuity. 
    Then for each $\holderexp \in (0,1]$, each $K \in [0 +\infty)$, and each $\delta_{1} \in (0, +\infty)$, there exist constants $\delta_2 \in (0, +\infty)$ and $N \in \n$ with the following property:

    For each $\widetilde{v} = \splfunv \in \splmoduspace[+\infty]$, each $\potential \in \holderspacesphere$, and each choice of $\spleigmea \in \probmea{\splitsphere}$ from Theorem~\ref{thm:subsystem:eigenmeasure existence and basic properties}, if $\holdernorm{\potential}{S^2} \leqslant K$, $\normcontinuous{\widetilde{v}}{\splitsphere} \geqslant \delta_1$, and $\int \! \widetilde{v}\eigfun \,\mathrm{d}\spleigmea = 0$, then\[
        \normcontinuous[\big]{\twistsplopt^{N}(\widetilde{v})}{\splitsphere} \leqslant \normcontinuous{\widetilde{v}}{\splitsphere} - \delta_2.
    \]
\end{lemma}
\begin{rmk}
    Note that at this point, we have yet to prove that $\spleigmea$ from Theorem~\ref{thm:subsystem:eigenmeasure existence and basic properties} is unique. We will prove it in Proposition~\ref{prop:uniqueness of eigenmeasure subsystem}. 
    Recall that $\eigfun = \splfun \in C(\splitsphere)$ is defined in Theorem~\ref{thm:existence of f invariant Gibbs measure} that depends only on $F$, $\mathcal{C}$, and $\potential$. 
\end{rmk}
\begin{proof}
    Fix arbitrary constants $\holderexp \in (0, 1]$, $K \in [0, +\infty)$, and $\delta_1 \in (0, +\infty)$. 
    Fix $\epsilon > 0$ sufficiently small such that $\modufun(\epsilon) < \delta_{1}/2$. Then $\epsilon$ depends only on $\modufun$ and $\delta_{1}$.
    Fix an arbitrary choice of $\spleigmea \in \probmea{\splitsphere}$ from Theorem~\ref{thm:subsystem:eigenmeasure existence and basic properties}, an arbitrary $\potential \in \holderspacesphere$, and an arbitrary $\widetilde{v} = \splfunv \in \splmoduspace[+\infty]$ with $\holdernorm{\potential}{S^2} \leqslant K$, $\normcontinuous{\widetilde{v}}{\splitsphere} \geqslant \delta_1$, and $\int \! \widetilde{v}\eigfun \,\mathrm{d}\spleigmea = 0$.

    Let $\spllimitset$ be the subset of $\splitsphere$ defined by 
    \[
        \spllimitset \define \bigcap\limits_{n \in \n} \bigcup \splDomain{n},
    \]
    where $\splDomain{n} \define  \bigcup_{\colour \in \colours} \bigl\{ i_{\colour}(X^n) \describe X^n \in \Domain{n}, \, X^n \subseteq X^0_{\colour} \bigr\}$ and $i_{\colour}$ is defined by \eqref{eq:natural injection into splitsphere}. 

    We first show that $\spleigmea \bigl(\widetilde{\limitset} \bigr) = 1$.
    Indeed, since $\spleigmea$ is an eigenmeasure of $\dualsplopt$, it follows from Proposition~\ref{prop:dual split operator}
    and induction on $n$ that for each $n \in \n$,
    \[
        1 \geqslant \spleigmea\parentheses[\Big]{ \bigcup \splDomain{n} } = \spleigmea \parentheses[\Big]{ \bigcup \splDomain{0} } = \spleigmea \parentheses{ \splitsphere } = 1,
    \]
    where we use the fact that $\bigcup\splDomain{0} = \splitsphere$ since $F$ is irreducible so that $F(\domF) = S^2$.
    Note that by \cite[Proposition~5.5~(i)]{shi2024thermodynamic}, $\bigl\{ \bigcup\splDomain{n} \bigr\}_{n \in \n}$ is a decreasing sequence of sets.
    Thus, 
    \[
        \spleigmea\bigl(\spllimitset\bigr) = \lim_{n \to +\infty} \spleigmea \parentheses[\Big]{ \bigcup\splDomain{n} } = 1.
    \]

    Since $\spleigmea \in \probmea{\splitsphere}$ is supported on $\widetilde{\limitset}$ and $\int \! \widetilde{v}\eigfun \,\mathrm{d}\spleigmea = 0$, there exist points $\widetilde{y}_{-}, \, \widetilde{y}_{+} \in \widetilde{\limitset}$ such that $\widetilde{v}(\widetilde{y}_{-}) \leqslant 0$ and $\widetilde{v}(\widetilde{y}_{+}) \geqslant 0$. 
    By Definition~\ref{def:split sphere}, we have $\widetilde{y}_{-} = (y_{-}, \colour')$ for some $\colour' \in \colours$ and $y_{-} \in X^0_{\colour'}$, and $\widetilde{y}_{+} = (y_{+}, \colour'')$ for some $\colour'' \in \colours$ and $y_{+} \in X^0_{\colour''}$.

    We fix an arbitrary point $\widetilde{x} \in \splitsphere$. 
    Then there exist $\colour \in \colours$ and $x \in X^0_{\colour}$ satisfying $\widetilde{x} = (x, \colour)$. 

    Since $\widetilde{y}_{-} = (y_{-}, \colour') \in \widetilde{\limitset}$, it follows from the definition of $\widetilde{\limitset}$ that there exists a sequence of tiles $\{ X^{n} \}_{n \in \n}$ satisfying $X^n \in \Domain{n}$ and $y_{-} \in X^{n + 1} \subseteq X^n \subseteq X^0_{\colour'}$ for each $n \in \n$. 
    By \cite[Proposition~8.4~(ii)]{bonk2017expanding}, there exists an integer $n_{\epsilon} \in \n$ depending only on $F$, $\mathcal{C}$, $d$, $\modufun$, and $\delta_{1}$ such that $\diam{d}{Y^{n_{\epsilon}}} < \epsilon$ for each $n_{\epsilon}$-tile $Y^{n_{\epsilon}} \in \Tile{n_{\epsilon}}$.
    Thus we have $y_{-} \in X^{n_{\epsilon}} \subseteq B_{d}(y_{-}, \epsilon) \cap X^0_{\colour'}$. 
    By Proposition~\ref{prop:subsystem:preliminary properties}~\ref{item:subsystem:properties:homeo}, we have $X^0 \define F^{n_{\epsilon}}(X^{n_{\epsilon}}) \in \bigl\{ \blacktile, \whitetile \bigr\}$. 
    Since $F \in \subsystem$ is primitive, by Definition~\ref{def:primitivity of subsystem}, there exist $n_{F} \in \n$ and $Y^{n_{F}} \in \Domain{n_{F}}$ satisfying $Y^{n_{F}} \subseteq X^0$ and $F^{n_{F}}(Y^{n_{F}}) = X^0_{\colour}$. 
    Then it follows from \cite[Lemma~5.17~(i)]{bonk2017expanding} and Proposition~\ref{prop:subsystem:preliminary properties}~\ref{item:subsystem:properties:homeo} that $Y^{n_{\epsilon} + n_{F}} \define \parentheses{ F^{n_{\epsilon}}|_{X^{n_{\epsilon}}} }^{-1}(Y^{n_{F}}) \in \Domain{n_{\epsilon} + n_{F}}$. 
    Note that $Y^{n_{\epsilon} + n_{F}} \subseteq X^{n_{\epsilon}} \subseteq X^0_{\colour'}$ and $F^{n_{\epsilon} + n_{F}}(Y^{n_{\epsilon} + n_{F}}) = F^{n_{F}}(Y^{n_{F}}) = X^0_{\colour}$. 
    Thus $Y^{n_{\epsilon} + n_{F}} \in \ccFTile{n_{\epsilon} + n_{F}}{\colour}{\colour'}$.
    Set $y \define \parentheses{  F^{n_{\epsilon} + n_{F}}|_{Y^{n_{\epsilon} + n_{F}}} }^{-1}(x)$. 
    Then we have $y \in Y^{n_{\epsilon} + n_{F}} \subseteq X^{n_{F}} \subseteq B_{d}(y_{-}, \epsilon) \cap X^0_{\colour'}$.
    Thus
    \[
        v_{\colour'}(y) \leqslant v_{\colour'}(y_{-}) + \modufun(\epsilon) = \widetilde{v}(\widetilde{y}_{-}) + \modufun(\epsilon) \leqslant \modufun(\epsilon) < \delta_{1} / 2 \leqslant \normcontinuous{\widetilde{v}}{\splitsphere} - \delta_{1} / 2.
    \]

    Denote $N \define n_{\epsilon} + n_{F}$, which depends only on $F$, $\mathcal{C}$, $d$, $\modufun$, and $\delta_{1}$. 
    Write $ x_{X^{N}} \define \bigl( F^{N}|_{X^N} \bigr)^{-1}(x) $ for $\widetilde{\colour} \in \colours$ and $X^N \in \ccFTile{N}{\colour}{\widetilde{\colour}}$.
    By Definition~\ref{def:normalized split ruelle operator}, \eqref{eq:def:twisted partial Ruelle operator}, and Lemma~\ref{lem:norm of normalized split operator}, we have
    \begin{align*}
        \twistsplopt^{N}(\widetilde{v})(\widetilde{x}) 
        &= \toperator{N}{\colour}{\black}(v_{\black})(x) + \toperator{N}{\colour}{\white}(v_{\white})(x) \\
        &= v_{\colour'}(y) \myexp[\big]{ S_{N}^{F}\normpotential(y) + \log u_{\colour'}(y) - \log u_{\colour}(x) }    \\
        &\quad\ + \sum_{\widetilde{\colour} \in \colours}  \sum_{ X^N \in \ccFTile{N}{\colour}{\widetilde{\colour}} \setminus \{ Y^{N} \} } 
        v_{\widetilde{\colour}}( x_{X^{N}} ) \myexp[\big]{ S_{N}^{F}\normpotential( x_{X^{N}} ) + \log u_{\widetilde{\colour}}( x_{X^{N}} ) - \log u_{\colour}(x) }     \\
        &\leqslant \bigl( \normcontinuous{\widetilde{v}}{\splitsphere} - \delta_{1} / 2 \bigr) \myexp[\big]{ S_{N}^{F}\normpotential(y) + \log u_{\colour'}(y) - \log u_{\colour}(x) } \\
        &\quad\ + \normcontinuous{\widetilde{v}}{\splitsphere} \sum_{\widetilde{\colour} \in \colours}  \sum_{ X^N \in \ccFTile{N}{\colour}{\widetilde{\colour}} \setminus \{ Y^{N} \} }  \myexp[\big]{ S_{N}^{F}\normpotential( x_{X^{N}} ) + \log u_{\widetilde{\colour}}( x_{X^{N}} ) - \log u_{\colour}(x) }   \\
        &\leqslant \normcontinuous{\widetilde{v}}{\splitsphere}  \sum_{\widetilde{\colour} \in \colours} \, \sum_{ X^N \in \ccFTile{N}{\colour}{\widetilde{\colour}} }  \myexp[\big]{ S_{N}^{F}\normpotential( x_{X^{N}} ) + \log u_{\widetilde{\colour}}( x_{X^{N}} ) - \log u_{\colour}(x) } \\
        &\quad\ - 2^{-1} \delta_{1} \myexp[\big]{ S_{N}^{F}\normpotential(y) + \log u_{\colour'}(y) - \log u_{\colour}(x) } \\
        &= \normcontinuous{\widetilde{v}}{\splitsphere} - 2^{-1} \delta_{1} \myexp[\big]{ S_{N}^{F}\normpotential(y) + \log u_{\colour'}(y) - \log u_{\colour}(x) }.
    \end{align*}
    Similarly, there exists $z \define \bigl(F^{N}|_{Z^{N}}\bigr)^{-1}(x)$ for some $Z^N \in \ccFTile{N}{\colour}{\colour''}$ such that $z \in Z^{N} \subseteq B_{d}(y_{+}, \epsilon) \cap X^0_{\colour''}$ and 
    \[
        \twistsplopt^{N}(\widetilde{v})(\widetilde{x}) \geqslant -\normcontinuous{\widetilde{v}}{\splitsphere} + 2^{-1} \delta_{1} \myexp[\big]{ S_{N}^{F}\normpotential(z) + \log u_{\colour''}(z) - \log u_{\colour}(x) }.
    \]
    Recall that $\eigfun = \splfun \in \splfunspace$ is a continuous function on $\splitsphere$ given by Theorem~\ref{thm:existence of f invariant Gibbs measure}. 
    Then by \eqref{eq:two sides bounds for eigenfunction} in Theorem~\ref{thm:existence of f invariant Gibbs measure} we have\[
        \Csplratio^{-1} \leqslant \eigfun(\widetilde{w}) \leqslant \Csplratio, \qquad \text{for each } \widetilde{w} \in \splitsphere,
    \]
    where $\Csplratio \geqslant 1$ is the constant defined in \eqref{eq:const:Csplratio} in Lemma~\ref{lem:distortion lemma for subsystem}~\ref{item:lem:distortion lemma for subsystem:uniform bound} and depends only on $F$, $\mathcal{C}$, $d$, $\phi$, and $\holderexp$.
    Hence we get 
    \begin{equation}    \label{eq:temp:upper bound for norm of normalized split ruelle operator}
        \begin{split}
            \normcontinuous[\big]{\twistsplopt^{N}(\widetilde{v})}{\splitsphere} 
            &\leqslant \normcontinuous{\widetilde{v}}{\splitsphere} 
            - 2^{-1} \delta_{1} \Csplratio^{-2} \inf_{w \in S^2} \myexp[\big]{ S_{N}^{F}\normpotential(w) } \\
            &\leqslant \normcontinuous{\widetilde{v}}{\splitsphere} 
            - 2^{-1} \delta_{1} \Csplratio^{-2} \myexp{ - N \uniformnorm{\normpotential} }.
        \end{split}
    \end{equation}

    Now we bound $\uniformnorm{\normpotential} = \uniformnorm{\potential - \pressure}$. 
    By the definition of \holder norm in Section~\ref{sec:Notation} and the hypothesis, $\uniformnorm{\potential} \leqslant \holdernorm{\potential}{S^2} \leqslant K$. Recall the definition of topological pressure as given in \eqref{eq:def:topological pressure} and the Variational Principle~\eqref{eq:Variational Principle for pressure} in Subsection~\ref{sub:thermodynamic formalism}. 
    Then by \eqref{eq:subsystem Variational Principle} in Theorem~\ref{thm:subsystem characterization of pressure and existence of equilibrium state}, \eqref{eq:def:measure-theoretic pressure}, and the fact that $\uniformnorm{\potential} \leqslant K$, we get\[
        - K \leqslant \pressure \leqslant P(f, \potential) \leqslant h_{\operatorname{top}}(f) + K.
    \]    
    Then $|\pressure| \leqslant K + h_{\operatorname{top}}(f) = K + \log(\deg{f})$ (see \cite[Corollary~17.2]{bonk2017expanding}). 
    Hence \[
        \uniformnorm{\normpotential} \leqslant \uniformnorm{\potential} + |\pressure| \leqslant 2K +  \log(\deg{f}).
    \]

    By \eqref{eq:const:Csplratio} in Lemma~\ref{lem:distortion lemma for subsystem}~\ref{item:lem:distortion lemma for subsystem:uniform bound} and \eqref{eq:const:C_1} in Lemma~\ref{lem:distortion_lemma}, quantitatively, we have
    \[
        \Csplratio = (\deg{f})^{n_{F}} \myexp[\bigg]{ 2n_{F} \uniformnorm{\phi} + C_0 \frac{ \holderseminorm{\potential}{S^2} }{1 - \Lambda^{-\holderexp}} (\diam{d}{S^{2}} )^{\holderexp} },
    \]
    where $C_0 > 1$ is the constant depending only on $f$, $\mathcal{C}$, and $d$ from Lemma~\ref{lem:basic_distortion}. 
    Set
    \[
        \Csplratio' \define (\deg{f})^{n_{F}} \myexp[\bigg]{ 2n_{F} K + \frac{ C_0 K }{1 - \Lambda^{-\holderexp}} (\diam{d}{S^{2}} )^{\holderexp} }.
    \]
    Then we have $\Csplratio \leqslant \Csplratio'$ for each $\potential \in \holderspacesphere$ with $\holdernorm{\potential}{S^{2}} \leq K$, and the constant $\Csplratio'$ only depends on $F$, $\mathcal{C}$, $d$, $K$, and $\holderexp$. 

    Therefore, by \eqref{eq:temp:upper bound for norm of normalized split ruelle operator}, $\normcontinuous[\big]{\twistsplopt^{n}(\widetilde{v})}{\splitsphere} \leqslant \normcontinuous{\widetilde{v}}{\splitsphere} - \delta_2$, where\[
        \delta_{2} \define 2^{-1} \delta_{1} \bigl( \Csplratio' \bigl)^{-2} \myexp{ - 2NK - N \log(\deg{f}) },
    \]
    which depends only on $F$, $\mathcal{C}$, $d$, $\modufun$, $\holderexp$, $K$, and $\delta_{1}$. 
\end{proof}
\begin{rmk}
    In Lemma~\ref{lem:upper bound for uniform norm of normalized split ruelle operator}, one cannot reduce the assumption ``$F \in \subsystem$ is strongly primitive'' to ``$F \in \subsystem$ is strongly irreducible''.
    To see this, let $F$ be as in Example~\ref{exam:subsystems}~\ref{item:exam:subsystems:strongly irreducible but not primitive}, which is strongly irreducible but not strongly primitive.
    In this case, $\limitset = \{p, \, q\}$ for some points $p \in \inte[\big]{X^0_{\black}}$ and $q \in \inte[\big]{X^0_{\white}}$ that satisfy $F(p) = q$ and $F(q) = p$.
    Set $\potential \equiv 0$ on $S^{2}$, $v_{\black} \equiv 1$ on $X^{0}_{\black}$, and $v_{\white} \equiv -1$ on $X^{0}_{\white}$.
    Then $\eigfun = \indicator{\splitsphere}$ and $\spleigmea = ( \delta_{p} / 2, \delta_{q} / 2 )  \in \probmea{\splitsphere}$ is an eigenmeasure of $\dualsplopt$ such that $\int \! \widetilde{v}\eigfun \,\mathrm{d}\spleigmea = 0$.
    However, $\twistsplopt(\widetilde{v}) = - \widetilde{v}$, which implies that $\normcontinuous[\big]{\twistsplopt^{n}(\widetilde{v})}{\splitsphere} = \normcontinuous{\widetilde{v}}{\splitsphere}$ for each $n \in \n$, contradicting the conclusion in Lemma~\ref{lem:upper bound for uniform norm of normalized split ruelle operator}.
\end{rmk}

We now establish a generalization of \cite[Theorem~5.17]{li2023prime:split}.

\begin{theorem}    \label{thm:converge of uniform norm of normalized split ruelle operator}
    Let $f$, $\mathcal{C}$, $F$, $d$, $\Lambda$ satisfy the Assumptions in Section~\ref{sec:The Assumptions}. 
    We assume in addition that $f(\mathcal{C}) \subseteq \mathcal{C}$ and $F \in \subsystem$ is strongly primitive.
    Let $\moduconst \in (0, +\infty)$ be a constant and $\modufun \colon [0, +\infty) \mapping [0, +\infty)$ an abstract modulus of continuity. 
    Let $H$ be a bounded subset of $\holderspacesphere$ for some $\holderexp \in (0, 1]$.
    Then for each $\widetilde{v} \in \splmoduspace$, each $\potential \in H$, and each choice of $\spleigmea \in \probmea{\splitsphere}$ from Theorem~\ref{thm:subsystem:eigenmeasure existence and basic properties}, we have
    \begin{equation}    \label{eq:converge of uniform norm of normed split ruelle operator}
         \lim_{n \to +\infty} \normcontinuous[\bigg]{\normsplopt^{n}(\widetilde{v}) - \eigfun \int \! \widetilde{v} \,\mathrm{d}\spleigmea}{\splitsphere} = 0.
    \end{equation}
    If, in addition, $\int \! \widetilde{v} \eigfun \,\mathrm{d}\spleigmea = 0$, then
    \begin{equation}    \label{eq:converge of uniform norm of normalized split ruelle operator}
        \lim_{n \to +\infty} \normcontinuous[\big]{\twistsplopt^{n}(\widetilde{v})}{\splitsphere} = 0.
    \end{equation}
    Moreover, the convergence in both \eqref{eq:converge of uniform norm of normed split ruelle operator} and \eqref{eq:converge of uniform norm of normalized split ruelle operator} is uniform in $\widetilde{v} \in \splmoduspace$, $\potential \in H$, and the choice of $\spleigmea$.
\end{theorem}
\begin{proof}
    \def\moduspl{\widetilde{C}^{\moduconst}_{\modufun}(\splitsphere, d)}
    We write $\moduspl \define \splmoduspace$ in this proof.

    Fix a constant $K \in [0, +\infty)$ such that $\holdernorm{\potential}{S^2} \leqslant K$ for each $\potential \in H$. 
    Let $\mathcal{M}_{F, \potential}$ be the set of possible choices of $\spleigmea \in \probmea{\splitsphere}$ from Theorem~\ref{thm:subsystem:eigenmeasure existence and basic properties}, i.e., 
    \begin{equation}    \label{eq:temp:set of eigenmeasure}
        \mathcal{M}_{F, \potential} \define \bigl\{ \splmeav \in \probmea{\splitsphere} \describe \dualsplopt\splmeav = \eigenvalue \splmeav \text{ for some } c \in \real \bigr\}.
    \end{equation}

    We recall that $\splmea \in \probmea{\splitsphere}$ defined in Theorem~\ref{thm:existence of f invariant Gibbs measure} by $\splmea = \eigfun \spleigmea$ depends on the choice of $\spleigmea$.

    Define for each $n \in \n_0$,\[
        a_n \define \sup \Bigl\{ \normcontinuous[\big]{\twistsplopt^{n}(\widetilde{v})}{\splitsphere} \describe \potential \in H, \, \widetilde{v} \in \moduspl, \, \int \! \widetilde{v} \eigfun \,\mathrm{d}\spleigmea = 0, \, \spleigmea \in \mathcal{M}_{F, \potential} \Bigr\}.
    \]
    By Lemma~\ref{lem:norm of normalized split operator}, $\normcontinuous[\big]{\twistsplopt}{\splitsphere} = 1$, so $\normcontinuous[\big]{\twistsplopt^{n}(\widetilde{v})}{\splitsphere}$ is non-increasing in $n$ for fixed $\potential \in H$ and $\widetilde{v} \in \moduspl$. Note that $a_0 \leqslant \moduconst < +\infty$. Thus $\{a_{n}\}_{n \in \n_0}$ is a non-increasing sequence of non-negative real numbers.

    Suppose now that $\lim_{n \to +\infty} a_{n} = a > 0$. By Proposition~\ref{prop:modulus for function under split ruelle operator}, there exists an abstract modulus of continuity $\widetilde{\modufun}$ such that \[
        \set[\big]{ \twistsplopt^{n}(\widetilde{v}) \describe n \in \n_0, \, \potential \in H, \, \widetilde{v} \in \moduspl } \subseteq \splmoduspace[\moduconst][\widetilde{\modufun}].
    \]
    Note that for each $\potential \in H$, each $n \in \n_0$, and each $\widetilde{v} \in \moduspl$ with $\int \! \widetilde{v}\eigfun \,\mathrm{d}\spleigmea = 0$, it follows from \eqref{eq:equilibrium state is eigenmeasure of dual normalized split operator} that \[
        \int \! \twistsplopt^{n}(\widetilde{v}) \eigfun \,\mathrm{d}\spleigmea = \int \! \twistsplopt^{n}(\widetilde{v}) \,\mathrm{d}\splmea = \int \! \widetilde{v} \,\mathrm{d}\splmea = 0.
    \] 
    Then by applying Lemma~\ref{lem:upper bound for uniform norm of normalized split ruelle operator} with $\widetilde{\modufun}$, $\holderexp$, $K$, and $\delta_{1} = a/2$, we find constants $n_0 \in \n$ and $\delta_{2} > 0$ such that \[
        \normcontinuous[\big]{\twistsplopt^{n_{0}}\bigl( \twistsplopt^{n}(\widetilde{v}) \bigl)  }{\splitsphere} 
        \leqslant \normcontinuous[\big]{\twistsplopt^{n}(\widetilde{v})}{\splitsphere} - \delta_{2},
    \]  
    for each $n \in \n_0$, each $\potential \in H$, each $\spleigmea \in \mathcal{M}_{F, \potential}$, and each $\widetilde{v} \in \moduspl$ with $\int \! \widetilde{v} \eigfun \,\mathrm{d}\spleigmea = 0$ and $\normcontinuous[\big]{\twistsplopt^{n}(\widetilde{v})}{\splitsphere} \geqslant a/2$. Since $\lim_{n \to +\infty} a_{n} = a$, we can fix integer $m > 1$ sufficiently large such that $a_{m} \leqslant a + \delta_{2}/2$. 
    Then for each $\potential \in H$, each $\spleigmea \in \mathcal{M}_{F, \potential}$, and each $\widetilde{v} \in \moduspl$ with $\int \! \widetilde{v} \eigfun \,\mathrm{d}\spleigmea = 0$ and $\normcontinuous[\big]{\twistsplopt^{m}(\widetilde{v})}{\splitsphere} \geqslant a/2$. we have\[
        \normcontinuous[\big]{\twistsplopt^{n_{0} + m}(\widetilde{v})}{\splitsphere} 
        \leqslant \normcontinuous[\big]{\twistsplopt^{m}(\widetilde{v})}{\splitsphere} - \delta_{2} 
        \leqslant a_{m} - \delta_{2} \leqslant a - \delta_{2}/2.
    \]
    On the other hand, since $\normcontinuous[\big]{\twistsplopt^{n}(\widetilde{v})}{\splitsphere}$ is non-increasing in $n$, we have that for each $\potential \in H$, each $\spleigmea \in \mathcal{M}_{F, \potential}$, and each $\widetilde{v} \in \moduspl$ with $\int \! \widetilde{v} \eigfun \,\mathrm{d}\spleigmea = 0$ and $\normcontinuous[\big]{\twistsplopt^{m}(\widetilde{v})}{\spllimitset} < a/2$, the following holds: \[
        \normcontinuous[\big]{\twistsplopt^{n_{0} + m}(\widetilde{v})}{\splitsphere} 
        \leqslant \normcontinuous[\big]{\twistsplopt^{m}(\widetilde{v})}{\splitsphere} < a/2.
    \]
    Thus $a_{n_{0} + m} \leqslant \max \bigl\{ a - \delta_{2}/2, \, a/2 \bigr\} < a$, contradicting the fact that $\{a_{n}\}_{n \in \n_0}$ is a non-increasing sequence and the assumption that $\lim_{n \to +\infty} a_{n} = a$. This proves the uniform convergence in \eqref{eq:converge of uniform norm of normalized split ruelle operator}.

    Next, we prove the uniform convergence in \eqref{eq:converge of uniform norm of normed split ruelle operator}. By Lemma~\ref{lem:norm of normalized split operator} and \eqref{eq:iteration of normalized split ruelle operator} in Lemma~\ref{lem:iteration of normalized split ruelle operator}, for each $\widetilde{v} \in \moduspl$, each $\potential \in H$, and each $\spleigmea \in \mathcal{M}_{F, \potential}$, we have
    \begin{equation}    \label{eq:temp:upper bound for norm associated with normed split ruelle operator}
        \begin{split}
            &\normcontinuous[\bigg]{\normsplopt^{n}(\widetilde{v}) - \eigfun \int \! \widetilde{v} \,\mathrm{d}\spleigmea}{\splitsphere} \\
            &\qquad\leqslant \normcontinuous{\eigfun}{\splitsphere} \normcontinuous[\bigg]{ \frac{1}{\eigfun} \normsplopt^{n}(\widetilde{v}) - \int \! \widetilde{v} \,\mathrm{d}\spleigmea}{\splitsphere} \\
            &\qquad= \normcontinuous{\eigfun}{\splitsphere} \normcontinuous[\bigg]{ \twistsplopt^{n}\biggl( \frac{\widetilde{v}}{\eigfun} \biggr) - \int \! \frac{\widetilde{v}}{\eigfun} \,\mathrm{d}\splmea}{\splitsphere} \\
            &\qquad= \normcontinuous{\eigfun}{\splitsphere} \normcontinuous[\bigg]{ \twistsplopt^{n}\biggl( \frac{\widetilde{v}}{\eigfun} - \indicator{\splitsphere} \int \! \frac{\widetilde{v}}{\eigfun} \,\mathrm{d}\splmea \biggr) }{\splitsphere}.
        \end{split}
    \end{equation}
    By \eqref{eq:eigenfunction of normed split operator} in Theorem~\ref{thm:existence of f invariant Gibbs measure}, we have
    \begin{equation}    \label{eq:temp:bounds for norm of eigenfunction}
        \Csplratio^{-1} \leqslant \normcontinuous{\eigfun}{\splitsphere} \leqslant \Csplratio,
    \end{equation}
    where $\Csplratio \geqslant 1$ is the constant defined in \eqref{eq:const:Csplratio} in Lemma~\ref{lem:distortion lemma for subsystem}~\ref{item:lem:distortion lemma for subsystem:uniform bound} and depends only on $F$, $\mathcal{C}$, $d$, $\phi$, and $\holderexp$.
    By \eqref{eq:const:Csplratio} and \eqref{eq:const:C_1}, quantitatively,
    \[
        \Csplratio = (\deg{f})^{n_{F}} \myexp[\bigg]{ 2n_{F} \uniformnorm{\phi} + C_0 \frac{ \holderseminorm{\potential}{S^2} }{1 - \Lambda^{-\holderexp}} (\diam{d}{S^{2}} )^{\holderexp} },
    \]
    where $C_0 > 1$ is the constant depending only on $f$, $\mathcal{C}$, and $d$ from Lemma~\ref{lem:basic_distortion} and $n_{F} \in \n$ is the constant depending only on $F$ and $\mathcal{C}$ from Definition~\ref{def:primitivity of subsystem} since $F$ is primitive. 
    Set
    \[
        \Csplratio' \define (\deg{f})^{n_{F}} \myexp[\bigg]{ 2n_{F} K + \frac{ C_0 K }{1 - \Lambda^{-\holderexp}} (\diam{d}{S^{2}} )^{\holderexp} }.
    \]
    Then we have $\Csplratio \leqslant \Csplratio'$ for each $\potential \in \holderspacesphere$ with $\holdernorm{\potential}{S^{2}} \leq K$, and the constant $\Csplratio'$ only depends on $F$, $\mathcal{C}$, $d$, $K$, and $\holderexp$. 
    Denote 
    \[
        \widetilde{w} \define \frac{\widetilde{v}}{\eigfun} - \indicator{\splitsphere} \int \! \frac{\widetilde{v}}{\eigfun} \,\mathrm{d}\splmea \in C(\splitsphere).
    \] 
    Then $\normcontinuous{\widetilde{w}}{\splitsphere} \leqslant 2 \normcontinuous{ \widetilde{v} / \eigfun} {\splitsphere} \leqslant 2 \moduconst \Csplratio'$.
    Due to the first inequality in \eqref{eq:two sides bounds for eigenfunction} and the fact that $\eigfun \in \splholderspace$ by Theorem~\ref{thm:existence of f invariant Gibbs measure}, we can apply Lemma~\ref{lem:product and inverse of modulus of continuity} and conclude that there exists an abstract modulus of continuity $\widehat{\modufun}$ associated with $\widetilde{v} / \eigfun$ such that $\widehat{\modufun}$ is independent of the choices of $\widetilde{v} \in \moduspl$, $\potential \in H$, and $\spleigmea \in \mathcal{M}_{F, \potential}$. 
    Thus $\widetilde{w} \in \moduspace[\widehat{\moduconst}][\widehat{\modufun}]$, where $\widehat{\moduconst} \define 2 \moduconst \Csplratio'$. Note that $\int \! \widetilde{w} \eigfun \,\mathrm{d}\spleigmea = \int \! \widetilde{w} \,\mathrm{d}\splmea = 0$. Finally, we can apply the uniform convergence in \eqref{eq:converge of uniform norm of normalized split ruelle operator} with $\widetilde{v} = \widetilde{w}$  to conclude the uniform convergence in \eqref{eq:converge of uniform norm of normed split ruelle operator} by \eqref{eq:temp:upper bound for norm associated with normed split ruelle operator} and \eqref{eq:temp:bounds for norm of eigenfunction}.   
\end{proof}

The following proposition is an immediate consequence of Theorem~\ref{thm:converge of uniform norm of normalized split ruelle operator}.
\begin{proposition}    \label{prop:convergence to equilibrium state under dual split Ruelle operator}
    Let $f$, $\mathcal{C}$, $F$, $d$, $\potential$ satisfy the Assumptions in Section~\ref{sec:The Assumptions}. 
    We assume in addition that $f(\mathcal{C}) \subseteq \mathcal{C}$ and $F \in \subsystem$ is strongly primitive.
    Let $\splmea \in \probmea{\splitsphere}$ be a Borel probability measure defined in Theorem~\ref{thm:existence of f invariant Gibbs measure}. 
    Then for each Borel probability measure $\splmeav \in \probmea{\splitsphere}$, we have
    \[
        \bigl( \dualtwistsplopt \bigl)^{n}\splmeav \weakconverge \splmea \quad \text{ as } n \to +\infty.
    \]
\end{proposition}
\begin{proof}
    Recall that for each $\widetilde{v} \in C(\splitsphere)$, there exists some abstract modulus of continuity $\modufun$ such that $\widetilde{v} \in \moduspace[][][S^2]$, where $\moduconst \define \normcontinuous{\widetilde{v}}{\splitsphere}$. 
    Recall from Theorem~\ref{thm:existence of f invariant Gibbs measure} that $\splmea = \eigfun \spleigmea$. 
    Then by Lemma~\ref{lem:norm of normalized split operator} and \eqref{eq:converge of uniform norm of normalized split ruelle operator} in Theorem~\ref{thm:converge of uniform norm of normalized split ruelle operator},
    \[
        \begin{split}
            &\lim_{n \to +\infty} \functional[\big]{ \parentheses[\big]{ \dualtwistsplopt } ^{n}\splmeav }{\widetilde{v}} \\
            &\qquad= \lim_{n \to +\infty} \parentheses[\big]{ \functional[\big]{ \splmeav }{ \twistsplopt^{n} \parentheses[\big]{ \widetilde{v} - \functional{\splmea}{\widetilde{v}}\indicator{\splitsphere} }  } 
                + \functional[\big]{\splmeav}{ \twistsplopt^{n} \parentheses[\big]{ \functional{\splmea}{\widetilde{v}} \indicator{\splitsphere} }  } } \\
            &\qquad= 0 + \functional[\big]{\splmeav}{ \functional{\splmea}{\widetilde{v}} \indicator{\splitsphere} } \\
            &\qquad= \functional{\splmea}{\widetilde{v}},
        \end{split}
    \]
    for each $\widetilde{v} \in C(\splitsphere)$. 
    This completes the proof.
\end{proof}

\subsection{Uniqueness}%
\label{sub:Uniqueness}
In this subsection, we finish the proof of Theorem~\ref{thm:subsystem:uniqueness of equilibrium state}. 

\smallskip

Theorem~\ref{thm:converge of uniform norm of normalized split ruelle operator} implies in particular the uniqueness of $\spleigmea \in \probmea{\splitsphere}$ from Theorem~\ref{thm:subsystem:eigenmeasure existence and basic properties}.

\begin{proposition}    \label{prop:uniqueness of eigenmeasure subsystem}
    Let $f$, $\mathcal{C}$, $F$, $d$, $\potential$ satisfy the Assumptions in Section~\ref{sec:The Assumptions}. 
    We assume in addition that $f(\mathcal{C}) \subseteq \mathcal{C}$ and $F \in \subsystem$ is strongly primitive.
    Then the measure $\eigmea = \spleigmea \in \probmea{\splitsphere}$ from Theorem~\ref{thm:subsystem:eigenmeasure existence and basic properties} is unique, i.e., $\spleigmea$ is the unique Borel probability measure on $\splitsphere$ that satisfies $\dualsplopt\spleigmea = \eigenvalue \spleigmea$ for some constant $\eigenvalue \in \real$. 
    Moreover, the measure $\equstate = \splmea \define \eigfun \spleigmea$ from Theorem~\ref{thm:existence of f invariant Gibbs measure} is the unique Borel probability measure on $\splitsphere$ that satisfies $\dualtwistsplopt\splmea = \splmea$.
\end{proposition}
Note that by Theorem~\ref{thm:subsystem characterization of pressure and existence of equilibrium state}, $\equstate$ is an equilibrium state for $\limitmap$ and $\potential|_{\limitset}$.
\begin{proof}
    Let $\spleigmea, \, \spleigmeaanother \in \probmea{\splitsphere}$ be two measures, both of which arise from Theorem~\ref{thm:subsystem:eigenmeasure existence and basic properties}. Note that for each $\widetilde{v} = \splfunv \in C(\splitsphere)$, there exists some abstract modulus of continuity $\modufun$ such that $\widetilde{v} = \splfunv \in \splmoduspace$, where $\moduconst \define \normcontinuous{\widetilde{v}}{\splitsphere}$.
    Then by \eqref{eq:converge of uniform norm of normed split ruelle operator} in Theorem~\ref{thm:converge of uniform norm of normalized split ruelle operator} and \eqref{eq:two sides bounds for eigenfunction} in Theorem~\ref{thm:existence of f invariant Gibbs measure}, we see that $\int \! \widetilde{v} \,\mathrm{d}\spleigmea = \int \! \widetilde{v} \,\mathrm{d}\spleigmeaanother$ for each $\widetilde{v} \in C(\splitsphere)$. Thus $\spleigmea = \spleigmeaanother$.

    Recall from \eqref{eq:equilibrium state is eigenmeasure of dual normalized split operator} that $\dualtwistsplopt\splmea = \splmea$. 
    Suppose $\splmeav \in \probmea{\splitsphere}$ is another measure with $\dualtwistsplopt\splmeav = \splmeav$. 
    It suffices to show that $\splmeav = \splmea$.
    Note that by \eqref{eq:two sides bounds for eigenfunction} in Theorem~\ref{thm:existence of f invariant Gibbs measure}, there exists a constant $C > 0$ such that $\eigfun(\widetilde{x}) \geqslant C$ for each $\widetilde{x}$. 
    Then by \eqref{eq:def:normalized split Ruelle operator} in Definition~\ref{def:normalized split ruelle operator}, for each $\widetilde{v} \in C(\splitsphere)$, we have
    \[
        \begin{split}
            \functional[\Big]{\dualtwistsplopt\splmeav}{\widetilde{v}} &= \functional[\Big]{\splmeav}{\twistsplopt(\widetilde{v})} = \functional[\Big]{\splmeav}{ \frac{1}{\eigfun}  \normsplopt( \eigfun \widetilde{v})} \\
            &= \functional[\bigg]{ \frac{\splmeav}{\eigfun} }{\normsplopt(\eigfun \widetilde{v})} = \functional[\bigg]{\eigfun \normdualsplopt\biggl(\frac{\splmeav}{\eigfun}\biggr)}{\widetilde{v}}.
        \end{split}
    \]
    This implies $\eigfun \normdualsplopt\Bigl(\frac{\splmeav}{\eigfun}\Bigr) = \dualtwistsplopt\splmeav = \splmeav$, i.e., $\normdualsplopt\Bigl(\frac{\splmeav}{\eigfun}\Bigr) = \frac{\splmeav}{\eigfun}$. 
    Denote $\lambda \define \functional[\Big]{\frac{\splmeav}{\eigfun}}{\indicator{\splitsphere}} > 0$. 
    Then by \eqref{eq:def:normed potential} we have $\dualsplopt\Bigl(\frac{\splmeav}{\lambda \eigfun}\Bigr) = e^{\pressure} \frac{\splmeav}{\lambda \eigfun}$.
    Noting that $\frac{\splmeav}{\lambda \eigfun}$ is also a Borel probability measure on $\splitsphere$, by the uniqueness of $\spleigmea$ we have $\frac{\splmeav}{\lambda \eigfun} = \spleigmea$. 
    Hence $\splmeav = \lambda \eigfun \spleigmea = \lambda \splmea$. 
    Since $\juxtapose{\splmeav}{\splmea} \in \probmea{\splitsphere}$, we get $\lambda = 1$ and $\splmeav = \splmea$. 
    Thus $\splmea$ is the unique Borel probability measure on $\splitsphere$ that satisfies $\dualtwistsplopt\splmea = \splmea$.
\end{proof}

We follow the conventions discussed in Remarks~\ref{rem:disjoint union} and \ref{rem:probability measure in split setting}.
\begin{lemma}     \label{lem:converge of derivative for pressure of subsystem} \def\ncolour{\colour_{n}}
    Let $f$, $\mathcal{C}$, $F$, $d$ satisfy the Assumptions in Section~\ref{sec:The Assumptions}. 
    We assume in addition that $f(\mathcal{C}) \subseteq \mathcal{C}$ and $F \in \subsystem$ is strongly primitive.
    Let $\moduconst \geqslant 0$ be a constant and $\modufun$ an abstract modulus of continuity. 
    Let $H$ be a bounded subset of $\holderspacesphere$ for some $\holderexp \in (0, 1]$.
    Fix arbitrary sequence $\sequen{x_{n}}$ of points in $S^2$ and sequence $\sequen{\ncolour}$ of colors in $\colours$ that satisfies $x_{n} \in X^0_{\ncolour}$ for each $n \in \n$.
    Then for each $v \in \moduspace[\moduconst][\modufun][S^2]$ and each $\potential \in H$, we have
    \begin{equation}    \label{eq:converge of derivative for pressure of subsystem}
        \lim_{n \to +\infty}  \frac{ \frac{1}{n} \sum\limits_{ X^n \in \ccFTile{n}{\ncolour}{} } \bigl(S_{n}^{F}v( x_{X^{n}} )\bigr) \myexp[\big]{ S^{F}_{n} \phi ( x_{X^{n}} ) } }{ \sum\limits_{ X^n \in \ccFTile{n}{\ncolour}{} } \myexp[\big]{ S^{F}_{n} \phi ( x_{X^{n}} ) } }
        = \int_{S^2} \! v \,\mathrm{d}\equstate,
    \end{equation}
    where we denote $x_{X^{n}} \define (F^{n}|_{X^{n}})^{-1}(x_{n})$ for each $X^{n} \in \ccFTile{n}{\ncolour}{}$, and $\equstate \in \probsphere$ is defined in Theorem~\ref{thm:existence of f invariant Gibbs measure}. 
    Moreover, the convergence is uniform in $v \in \moduspace[\moduconst][\modufun][S^2]$ and $\potential \in H$.
\end{lemma}
\begin{proof}
    \def\ncolour{\colour_{n}}
    By Lemma~\ref{lem:iteration of split-partial ruelle operator} and Definition~\ref{def:partial Ruelle operator}, for each $n \in \n$, each $v \in \moduspace[\moduconst][\modufun][S^2]$, and each $\potential \in H$,
    \begin{align*}
        \frac{ \frac{1}{n} \sum\limits_{ X^n \in \ccFTile{n}{\ncolour}{} } \bigl(S_{n}^{F}v( x_{X^{n}} )\bigr) \myexp[\big]{ S^{F}_{n} \phi ( x_{X^{n}} ) } }{ \sum\limits_{ X^n \in \ccFTile{n}{\ncolour}{} } \myexp[\big]{ S^{F}_{n} \phi ( x_{X^{n}} ) } }
        &= \frac{ \frac{1}{n} \sum\limits_{j=0}^{n-1} \sum\limits_{ X^n \in \ccFTile{n}{\ncolour}{} } v(f^{j}( x_{X^{n}} )) \myexp[\big]{ S^{F}_{n} \phi ( x_{X^{n}} ) } }{ \bigl( \splopt^{n}\bigl(\indicator{\splitsphere}\bigr) \bigr)(x_{n}, \ncolour) } \\
        &= \frac{ \frac{1}{n} \sum\limits_{j=0}^{n-1} \Bigl( \splopt^{n}(\widetilde{v\circ f^{j}}) \Bigl)(x_{n}, \ncolour) }{ \bigl( \splopt^{n}\bigl(\indicator{\splitsphere}\bigr) \bigr)(x_{n}, \ncolour) },
    \end{align*}
    where, by abuse of notation, for each $u \in C(S^2)$ we denote by $\widetilde{u}$ the continuous function on $\splitsphere$ given by $\widetilde{u}(\widetilde{z}) \define u(z)$ for each $\widetilde{z} = (z, \colour) \in \splitsphere$. 
    Note that by Lemma~\ref{lem:iteration of split-partial ruelle operator} and Definition~\ref{def:partial Ruelle operator}, for each $j \in \n_0$, 
    \[
        \splopt^{j}(\widetilde{v \circ f^{j}}) = \widetilde{v} \, \splopt^{j}\bigl(\indicator{\splitsphere}\bigr).
    \]
    Hence, 
    \[
        \begin{split}
            \frac{ \frac{1}{n} \sum\limits_{j=0}^{n-1} \Bigl( \splopt^{n}(\widetilde{v\circ f^{j}}) \Bigl)(x_{n}, \ncolour) }{ \bigl( \splopt^{n}\bigl(\indicator{\splitsphere}\bigr) \bigr)(x_{n}, \ncolour) } 
            &= \frac{ \frac{1}{n} \sum\limits_{j=0}^{n-1} \bigl( \splopt^{n - j} \bigl( \widetilde{v}\, \splopt^{j}\bigl(\indicator{\splitsphere}\bigr) \bigr) \bigl)(x_{n}, \ncolour) }{ \bigl( \splopt^{n}\bigl(\indicator{\splitsphere}\bigr) \bigr)(x_{n}, \ncolour) }  \\
            &= \frac{ \frac{1}{n} \sum\limits_{j=0}^{n-1} \bigl( \normsplopt^{n - j} \bigl( \widetilde{v}\, \normsplopt^{j}\bigl(\indicator{\splitsphere}\bigr) \bigr) \bigl)(x_{n}, \ncolour) }{ \bigl( \normsplopt^{n}\bigl(\indicator{\splitsphere}\bigr) \bigr)(x_{n}, \ncolour) }.
        \end{split}
    \]
    By Proposition~\ref{prop:modulus for function under split ruelle operator}, $\bigl\{ \normsplopt^{n}\bigl( \indicator{\splitsphere} \bigl) \describe n \in \n_0 \bigr\} \subseteq \splmoduspace[\widehat{\moduconst}][\widehat{\modufun}]$, for some constant $\widehat{\moduconst} \geqslant 0$ and some abstract modulus of continuity $\widehat{\modufun}$, which are independent of the choice of $\potential \in H$. Thus by Lemma~\ref{lem:product and inverse of modulus of continuity},
    \begin{equation}    \label{eq:temp:modulus of product of function and split operator}
        \bigl\{ \widetilde{v} \, \normsplopt^{n}\bigl( \indicator{\splitsphere} \bigl) \describe n \in \n_0, \, v \in \moduspace[][][S^2] \bigr\} \subseteq \splmoduspace[\moduconst_1][\modufun_1],
    \end{equation}
    for some constant $\moduconst_1 \geqslant 0$ and some abstract modulus of continuity $\modufun_1$, which are independent of the choice of $\potential \in H$.

    By Theorem~\ref{thm:converge of uniform norm of normalized split ruelle operator} and Proposition~\ref{prop:uniqueness of eigenmeasure subsystem}, we have
    \begin{equation}    \label{eq:temp:converge indicator under operator}
        \normcontinuous[\big]{ \normsplopt^{k}\bigl( \indicator{\splitsphere} \bigl) - \eigfun }{\splitsphere} \converge 0,
    \end{equation}
    as $k \to +\infty$, uniformly in $\potential \in H$. Moreover, by \eqref{eq:temp:modulus of product of function and split operator}, the independence of $\moduconst_1$ and $\modufun_1$ on $\potential \in H$ in \eqref{eq:temp:modulus of product of function and split operator}, Theorem~\ref{thm:converge of uniform norm of normalized split ruelle operator}, and Proposition~\ref{prop:uniqueness of eigenmeasure subsystem}, we have 
    \begin{equation}    \label{eq:temp:converge product function under operator}
           \normcontinuous[\bigg]{ \normsplopt^{k}\Bigl( \widetilde{v} \, \normsplopt^{j} \bigl( \indicator{\splitsphere} \bigr) \Bigl) - \eigfun \int \! \widetilde{v} \, \normsplopt^{j} \bigl( \indicator{\splitsphere} \bigr) \,\mathrm{d}\spleigmea }{\splitsphere} \converge 0,
    \end{equation}   
    as $k \to +\infty$, uniformly in $j \in \n_0$, $\potential \in H$, and $v \in \moduspace[][][S^2]$.

    Fix a constant $K \geqslant 0$ such that for each $\potential \in H$, $\holdernorm{\potential}{S^2} \leqslant K$.
    By \eqref{eq:eigenfunction of normed split operator} in Theorem~\ref{thm:existence of f invariant Gibbs measure}, we have
    \begin{equation}    \label{eq:temp:bounds for norm of eigenfunction use in converge}
        \Csplratio^{-1} \leqslant \normcontinuous{\eigfun}{\splitsphere} \leqslant \Csplratio,
    \end{equation}
    where $\Csplratio \geqslant 1$ is the constant defined in \eqref{eq:const:Csplratio} in Lemma~\ref{lem:distortion lemma for subsystem}~\ref{item:lem:distortion lemma for subsystem:uniform bound} and depends only on $F$, $\mathcal{C}$, $d$, $\phi$, and $\holderexp$.
    By \eqref{eq:const:Csplratio} and \eqref{eq:const:C_1}, quantitatively,
    \[
        \Csplratio = (\deg{f})^{n_{F}} \myexp[\bigg]{ 2n_{F} \uniformnorm{\phi} + C_0 \frac{ \holderseminorm{\potential}{S^2} }{1 - \Lambda^{-\holderexp}} (\diam{d}{S^{2}} )^{\holderexp} },
    \]
    where $C_0 > 1$ is the constant depending only on $f$, $\mathcal{C}$, and $d$ from Lemma~\ref{lem:basic_distortion} and $n_{F} \in \n$ is the constant depending only on $F$ and $\mathcal{C}$ from Definition~\ref{def:primitivity of subsystem} since $F$ is primitive. 
    Define
    \[
        \Csplratio' \define (\deg{f})^{n_{F}} \myexp[\bigg]{ 2n_{F} K + \frac{ C_0 K }{1 - \Lambda^{-\holderexp}} (\diam{d}{S^{2}} )^{\holderexp} }.
    \]
    Then we have $\Csplratio \leqslant \Csplratio'$ for each $\potential \in \holderspacesphere$ with $\holdernorm{\potential}{S^{2}} \leq K$, and the constant $\Csplratio'$ only depends on $F$, $\mathcal{C}$, $d$, $K$, and $\holderexp$. 

    Thus by \eqref{eq:temp:modulus of product of function and split operator}, we get that for $j \in \n_0$, $v \in \moduspace[][][S^2]$, and $\potential \in H$, 
    \begin{equation}    \label{eq:temp:upper bound for norm of second term}
        \normcontinuous[\bigg]{ \eigfun \int \! \widetilde{v} \, \normsplopt^{j} \bigl( \indicator{\splitsphere} \bigr) \,\mathrm{d}\spleigmea }{\splitsphere}
        \leqslant \normcontinuous{\eigfun}{\splitsphere} \normcontinuous[\Big]{\widetilde{v} \, \normsplopt^{j} \bigl( \indicator{\splitsphere} \bigr)}{\splitsphere}
        \leqslant \moduconst_1 \Csplratio'.
    \end{equation}
    By \eqref{eq:modulus for function under normed split Ruelle operator} in Proposition~\ref{prop:modulus for function under split ruelle operator} and \eqref{eq:temp:modulus of product of function and split operator}, we get some constant $\moduconst_2 > 0$ such that for each $j, \, k \in \n_0$, each $v \in \moduspace[][][S^2]$, and each $\potential \in H$,
    \begin{equation}    \label{eq:temp:upper bound for norm of first term}
        \normcontinuous[\Big]{\normsplopt^{k} \Bigl( \widetilde{v} \, \normsplopt^{j} \bigl( \indicator{\splitsphere} \bigr) \Bigr) }{\splitsphere} < \moduconst_2.
    \end{equation}
    Hence we can conclude from \eqref{eq:temp:upper bound for norm of second term}, \eqref{eq:temp:upper bound for norm of first term}, and \eqref{eq:temp:converge product function under operator} that
    \[
        \lim_{n \to +\infty} \frac{1}{n} \normcontinuous[\bigg]{ \sum\limits_{j=0}^{n-1} \normsplopt^{n - j}\Bigl( \widetilde{v} \, \normsplopt^{j} \bigl( \indicator{\splitsphere} \bigr) \Bigl) - \sum\limits_{j=0}^{n-1} \eigfun \int \! \widetilde{v} \, \normsplopt^{j} \bigl( \indicator{\splitsphere} \bigr) \,\mathrm{d}\spleigmea  }{\splitsphere} = 0,
    \]
    uniformly in $v \in \moduspace[][][S^2]$ and $\potential \in H$. Thus by \eqref{eq:temp:converge indicator under operator} and \eqref{eq:temp:bounds for norm of eigenfunction use in converge}, we have\[
        \lim_{n \to +\infty}  \normcontinuous[\biggg]{ \frac{ \frac{1}{n} \sum\limits_{j=0}^{n-1} \normsplopt^{n - j}\Bigl( \widetilde{v} \, \normsplopt^{j} \bigl( \indicator{\splitsphere} \bigr) \Bigl)} { \normsplopt^{n}\bigl( \indicator{\splitsphere} \bigl) }  -  \frac{ \frac{1}{n} \sum\limits_{j=0}^{n-1} \eigfun \displaystyle\int \! \widetilde{v} \, \normsplopt^{j} \bigl( \indicator{\splitsphere} \bigr) \,\mathrm{d}\spleigmea }{\eigfun}  }{\splitsphere} = 0,
    \] 
    uniformly in $v \in \moduspace[][][S^2]$ and $\potential \in H$. Combining the above with \eqref{eq:temp:modulus of product of function and split operator}, \eqref{eq:temp:converge indicator under operator}, \eqref{eq:temp:bounds for norm of eigenfunction use in converge}, and the calculation at the beginning of the proof, we can conclude, therefore, that the left-hand side of \eqref{eq:converge of derivative for pressure of subsystem} is equal to 
    \[
        \begin{split}
            \lim_{n \to +\infty} \frac{1}{n} \sum\limits_{j = 0}^{n - 1} \int \! \widetilde{v} \, \normsplopt^{j} \bigl( \indicator{\splitsphere} \bigr) \,\mathrm{d}\spleigmea 
            &= \lim_{n \to +\infty} \frac{1}{n} \sum\limits_{j = 0}^{n - 1} \int \! \widetilde{v} \, \eigfun \,\mathrm{d}\spleigmea \\
            &= \int \! \widetilde{v} \,\mathrm{d}\splmea = \int \! v \,\mathrm{d}\equstate,
        \end{split}
    \]
    where $\equstate \in \probsphere$ is defined in Theorem~\ref{thm:existence of f invariant Gibbs measure}, and the convergence is uniform in $v \in \moduspace[][][S^2]$ and $\potential \in H$.
\end{proof}

\begin{theorem}    \label{thm:derivative of pressure subsystem}
    Let $f$, $\mathcal{C}$, $F$, $d$, $\potential$ satisfy the Assumptions in Section~\ref{sec:The Assumptions}. 
    We assume in addition that $f(\mathcal{C}) \subseteq \mathcal{C}$ and $F \in \subsystem$ is strongly primitive.
    Let $\potential, \, \gamma \in \holderspacesphere$ be real-valued \holder continuous function with an exponent $\holderexp \in (0, 1]$. 
    Then for each $t \in \real$, we have
    \[
        \frac{\mathrm{d}}{\mathrm{d}t} \pressure[\potential + t \gamma] = \int \! \gamma \,\mathrm{d}\equstate[\potential + t \gamma],
    \]
    where $\equstate[\potential + t \gamma] \in \probsphere$ is defined in Theorem~\ref{thm:existence of f invariant Gibbs measure}. 
\end{theorem}
\begin{proof}
    We will use the well-known fact from real analysis that if a sequence $\{g_{n}\}_{n \in \n}$ of real-valued differentiable functions defined on a finite interval in $\real$ converges pointwise to some function $g$ and the sequence of the corresponding derivatives $\bigl\{ \frac{\mathrm{d}g_{n}}{\mathrm{d}t} \bigr\}_{n \in \n}$ converges uniformly to some function $h$, then $g$ is differentiable and $\frac{dg}{\mathrm{d}t} = h$.

    By \eqref{eq:equalities for characterizations of pressure} in Theorem~\ref{thm:existence of f invariant Gibbs measure} and \eqref{eq:iteration of split-partial ruelle operator} in Lemma~\ref{lem:iteration of split-partial ruelle operator}, for each $\colour \in \colours$, each $x \in X^0_{\colour}$, and each $\psi \in \holderspacesphere$, we have
    \begin{equation}    \label{eq:temp:characterization of pressure for subsystem}
        \begin{split}
            \pressure[\psi] &= \lim_{n \to +\infty} \frac{1}{n} \log \mathopen{}\bigl( \splopt[\psi]^{n}\bigl(\indicator{\splitsphere}\bigr)(x, \colour) \bigr) 
            = \lim_{n \to +\infty} \frac{1}{n} \log \sum\limits_{ X^n \in \ccFTile{n}{\colour}{} } \myexp[\big]{ S^{F}_{n} \psi( x_{X^{n}} ) },
        \end{split}
    \end{equation}
    where $x_{X^{n}} \define (F^{n}|_{X^{n}})^{-1}(x)$ for each $X^{n} \in \ccFTile{n}{\colour}{}$.

    Fix $\colour \in \colours$, $x \in X^0_{\colour}$, and $\ell \in (0, +\infty)$. 
    For each $n \in \n$ and each $t \in \real$, define\[
        P_n(t) \define \frac{1}{n} \log \sum_{ X^n \in \ccFTile{n}{\colour}{} } \myexp[\big]{ S^{F}_{n}(\phi + t \gamma)(x_{X^{n}}) },
    \]
    where $x_{X^{n}} \define (F^{n}|_{X^{n}})^{-1}(x)$ for each $X^{n} \in \ccFTile{n}{\colour}{}$.
    Observe that there exists a bounded subset $H$ of $\holderspacesphere$ such that $\phi + t \gamma \in H$ for each $t \in (-\ell, \ell)$. Then by Lemma~\ref{lem:converge of derivative for pressure of subsystem},\[
        \frac{\mathrm{d} P_{n}}{\mathrm{d}t}(t) = \frac{ \frac{1}{n} \sum\limits_{ X^n \in \ccFTile{n}{\colour}{} } \bigl(S_n^{F}\gamma( x_{X^{n}} )\bigr) \myexp[\big]{ S^{F}_{n} (\phi + t \gamma)(x_{X^{n}}) } }{ \sum\limits_{ X^n \in \ccFTile{n}{\colour}{} } \myexp[\big]{ S^{F}_{n} (\phi + t \gamma)(x_{X^{n}}) } }
    \]
    converges to $\int \! \gamma \,\mathrm{d}\equstate[\potential + t \gamma]$ as $n \to +\infty$, uniformly in $t \in (-\ell, \ell)$.

    On the other hand, by \eqref{eq:temp:characterization of pressure for subsystem}, for each $t \in (-\ell, \ell)$, we have\[
        \lim_{n \to +\infty}  P_{n}(t) = \pressure[\phi + t \gamma].
    \]
    Hence $\pressure[\phi + t \gamma]$ is differentiable with respect to $t$ on $(-\ell, \ell)$, and \[
        \frac{\mathrm{d}}{\mathrm{d}t} \pressure[\phi + t \gamma] = \lim_{n \to +\infty} \frac{\mathrm{d} P_{n}}{\mathrm{d}t}(t) = \int \! \gamma \,\mathrm{d}\equstate[\potential + t \gamma].
    \]
    Since $\ell \in (0, +\infty)$ is arbitrary, the proof is complete.
\end{proof}

We record the following well-known fact for the convenience of the reader.
\begin{lemma}    \label{lem:holder space dense in continuous space}
    Let $(X, d)$ be a compact metric space. 
    Then for each $\holderexp \in (0, 1]$, $\holderspace$ is a dense subset of $C(X)$ with respect to the uniform norm. 
\end{lemma}
\begin{proof}
    The lemma follows from the fact that the set of Lipschitz functions is dense in $C(X)$ with respect to the uniform norm (see for example, \cite[Theorem~6.8]{heinonen2001lectures}).
\end{proof}

Now we prove the uniqueness of the equilibrium states for subsystems. 

\begin{proof}[Proof of Theorem~\ref{thm:subsystem:uniqueness of equilibrium state}]
    The existence is from Theorem~\ref{thm:subsystem characterization of pressure and existence of equilibrium state}. 
    
    We now prove the uniqueness.
    
    Denote $\limitset \define \limitset(F, \mathcal{C})$ and $\limitmap \define F|_{\limitset}$. 
    Recall that for each $\varphi \in C(S^2)$, $\fpressure[\varphi|_{\limitset}]$ is the topological pressure of $\limitmap \colon \limitset \mapping \limitset$ with respect to the potential $\varphi|_{\limitset}$. 

    Since $\potential \in \holderspacesphere$ for some $\holderexp \in (0, 1]$, it follows from \eqref{eq:subsystem Variational Principle} in Theorem~\ref{thm:subsystem characterization of pressure and existence of equilibrium state} and Theorem~\ref{thm:derivative of pressure subsystem} that the function \[
        t \mapsto \fpressure[(\potential + t \gamma)|_{\limitset}]
    \] 
    is differentiable at $0$ for each $\gamma \in \holderspacesphere$.
    Write 
    \[
        W \define \bigl\{ \psi|_{\limitset} \in \holderspace[][][\limitset] \describe \psi \in \holderspacesphere \bigr\}.
    \] 
    By Lemma~\ref{lem:holder space dense in continuous space}, $W$ is a dense subset of $C(\limitset)$ with respect to the uniform norm.
    In particular, $W$ is a dense subset of $C(\limitset)$ in the weak topology.
    We note that the topological pressure function $\fpressure[\cdot] \colon C(\limitset) \mapping \real$ is convex and continuous (see for example, \cite[Theorem~3.6.1 and Theorem~3.6.2]{przytycki2010conformal}). 
    Thus by Theorem~\ref{thm:tengent} with $V = C(\limitset)$, $x = \phi$, $U = W$, and $Q = \fpressure[\cdot]$, we get $\card[\big]{ \tangent{C(\limitset)}{\phi}{\fpressure[\cdot]} } = 1$.

    On the other hand, if $\mu \in \mathcal{M}(\limitset, \limitmap)$ is an equilibrium state for $\limitmap$ and $\potential|_{\limitset}$, then by \eqref{eq:def:measure-theoretic pressure} and \eqref{eq:Variational Principle for pressure},\[
        h_{\mu}(\limitmap) + \int \! \potential \,\mathrm{d}\mu = \fpressure[\potential|_{\limitset}],
    \]  
    and for each $\gamma \in C(\limitset)$,
    \[
        h_{\mu}(\limitmap) + \int \! (\potential + \gamma) \,\mathrm{d}\mu \leqslant \fpressure[(\potential + \gamma)|_{\limitset}].
    \]
    Thus $\int \! \gamma \,\mathrm{d}\mu \leqslant \fpressure[(\potential + \gamma)|_{\limitset}] - \fpressure[\potential|_{\limitset}]$. 
    Then by \eqref{eq:def:tangent}, the continuous functional $\gamma \mapsto \int \! \gamma \,\mathrm{d}\mu$ on $C(\limitset)$ is in $\tangent{C(\limitset)}{\potential}{\fpressure[\cdot]}$. Since $\equstate$ defined in Theorem~\ref{thm:existence of f invariant Gibbs measure} is an equilibrium state for $\limitmap$ and $\potential|_{\limitset}$, and $\card[\big]{ \tangent{C(\limitset)}{\phi}{\fpressure[\cdot]} } = 1$, we get that each equilibrium state $\mu$ for $\limitmap$ and $\potential|_{\limitset}$ must satisfy $\int \! \gamma \,\mathrm{d}\mu = \int \! \gamma \,\mathrm{d}\equstate$ for $\gamma \in C(\limitset)$, i.e., $\mu = \equstate$.

    Finally, it follows from Theorem~\ref{thm:subsystem characterization of pressure and existence of equilibrium state} that the map $\limitmap$ is forward quasi-invariant and non-singular with respect to $\equstate$.
\end{proof}

\begin{rmk}
    Let $f$, $\mathcal{C}$, $F$, $d$, $\potential$ satisfy the Assumptions in Section~\ref{sec:The Assumptions}. 
    We assume in addition that $f(\mathcal{C}) \subseteq \mathcal{C}$ and $F \in \subsystem$ is strongly primitive.
    Since the entropy map $\mu \mapsto h_\mu(\limitmap)$ for $\limitmap \colon \limitset \mapping \limitset$ is \defn{affine} (see for example, \cite[Theorem~8.1]{walters1982introduction}), i.e., if $\juxtapose{\mu}{\nu} \in \invmea[\limitset][\limitmap]$ and $p \in [0, 1]$, then $h_{p\mu + (1 - p)\nu}(\limitmap) = p h_{\mu}(\limitmap) + (1 - p) h_{\nu}(\limitmap)$, so is the pressure map $\mu \mapsto P_{\mu}(\limitmap, \potential|_{\limitset})$ for $\limitmap$ and $\potential|_{\limitset}$. 
    Thus, the uniqueness of the equilibrium state $\equstate$ and the Variational Principle~\eqref{eq:Variational Principle for pressure} imply that $\equstate$ is an extreme point of the convex set $\invmea[\limitset][\limitmap]$. 
    It follows from the fact (see for example, \cite[Theorem~2.2.8]{przytycki2010conformal}) that the extreme points of $\invmea[\limitset][\limitmap]$ are exactly the ergodic measures in $\invmea[\limitset][\limitmap]$ that $\equstate$ is ergodic. 
    However, we are going to prove a much stronger ergodic property of $\equstate$ in Section~\ref{sec:Ergodic Properties}.
\end{rmk}

Theorem~\ref{thm:subsystem:uniqueness of equilibrium state} implies in particular that there exists a unique equilibrium state $\mu_{\potential}$ for each expanding Thurston map $f \colon S^2 \mapping S^2$ together with a real-valued \holder continuous potential $\potential$.
\begin{corollary}    \label{coro:existence and uniqueness of equilibrium state for expanding Thurston map}
    Let $f$, $d$, $\potential$ satisfy the Assumptions in Section~\ref{sec:The Assumptions}.
    Then there exists a unique equilibrium state $\mu_{\potential}$ for the map $f$ and the potential $\potential$.
\end{corollary}
\begin{proof}
    By Lemma~\ref{lem:invariant_Jordan_curve} we can find a sufficiently high iterate $F \define f^n$ of $f$ for some $n \in \n$ that has an $F$-invariant Jordan curve $\mathcal{C} \subseteq S^2$ with $\post{F} = \post{f} \subseteq \mathcal{C}$. 
    Then $F$ is also an expanding Thurston map by Remark~\ref{rem:Expansion_is_independent}.
    In particular, by \cite[Lemma~5.10]{li2018equilibrium} and Definition~\ref{def:primitivity of subsystem}, $F$ is a strongly primitive subsystem of $F$ with respect to $\mathcal{C}$ and $\limitset(F, \mathcal{C}) = S^2$.

    Denote $\Phi \define S_{n}^{f} \potential$. 
    By Theorem~\ref{thm:subsystem:uniqueness of equilibrium state} there exists a unique equilibrium state $\equstate[\Phi] \in \mathcal{M}(S^2, F)$ for the map $F$ and the potential $\Phi$. 
    Note that $\equstate[\Phi]$ is $F$-invariant. 
    Set 
    \[
        \mu \define \frac{1}{n} \sum_{i = 0}^{n - 1} f^{i}_{*} \equstate[\Phi] \in \mathcal{M}(S^2, f).
    \]
    Then we get $n h_{\mu}(f) = h_{\equstate[\Phi]}(F)$ (see for example, \cite[Lemma~5.2]{shi2024entropy}) and
    \[
        \int \! \potential \,\mathrm{d}\mu = \frac{1}{n} \int \sum_{i = 0}^{n - 1} \potential \,\mathrm{d}f_{*}^{i} \equstate[\Phi]
        = \frac{1}{n} \int \sum_{i = 0}^{n - 1} \potential \circ f^{i} \,\mathrm{d}\equstate[\Phi]
        = \frac{1}{n} \int \! \Phi \,\mathrm{d}\equstate[\Phi].
    \]
    Noting that $P(F, \Phi) = n P(f, \potential)$ (recall \eqref{eq:def:topological pressure}), we obtain
    \[
         h_{\mu}(f) + \int \! \phi \,\mathrm{d}\mu
        = \frac{1}{n} \biggl( h_{\equstate[\Phi]}(F) + \int \! \Phi \,\mathrm{d}\equstate[\Phi] \biggr)
        = \frac{1}{n} P(F, \Phi) = P(f, \potential).
    \]
    Thus $\mu$ is an equilibrium state for the map $f$ and the potential $\potential$.

    Now we prove the uniqueness. Suppose $\nu$ is an equilibrium state for the map $f$ and the potential $\potential$. By similar arguments as above one sees that $\nu$ is an equilibrium state for the map $F$ and the potential $\Phi$. Then it follows from the uniqueness part of Theorem~\ref{thm:subsystem:uniqueness of equilibrium state} that $\nu = \equstate[\Phi]$.
\end{proof} 

\section{Ergodic Properties}
\label{sec:Ergodic Properties}

In this section, we show in Theorem~\ref{thm:subsystem exact with respect to equilibrium state} that if $f$, $\mathcal{C}$, $F$, $d$, and $\potential$ satisfied the Assumptions in Section~\ref{sec:The Assumptions}, $f(\mathcal{C}) \subseteq \mathcal{C}$, and $F \in \subsystem$ is strongly primitive, then the measure-preserving transformation $F|_{\limitset}$ of the probability space $(\limitset, \equstate)$ is exact (see Definition~\ref{def:exact}), and as an immediate consequence, mixing (see Corollary~\ref{coro:subsystem mixing}). 
Another consequence of Theorem~\ref{thm:subsystem exact with respect to equilibrium state} is that $\equstate$ is non-atomic (see Corollary~\ref{coro:subsystem equilibrium state non-atomic}).

\smallskip

For each Borel measure $\mu$ on a compact metric space $(X, d)$, we denote by $\overline{\mu}$ the \emph{completion} of $\mu$, i.e., $\overline{\mu}$ is the unique measure defined on the smallest $\sigma$-algebra $\overline{\mathcal{B}}$ containing all Borel sets and all subsets of $\mu$-null sets, satisfying $\overline{\mu}(E) = \mu(E)$ for each Borel set $E \subseteq X$.

\begin{definition}    \label{def:exact}
    Let $T \colon X \mapping X$ be a measure-preserving transformation of a probability space $(X, \mu)$. 
    Then $T$ is called \emph{exact} if for every measurable set $E$ with $\mu(E) > 0$ and measurable images $T(E), \, T^2(E), \, \dots$, the following holds:
    \begin{equation*}
        \lim\limits_{n \to +\infty} \mu (T^n(E)) = 1.
    \end{equation*}
\end{definition}

\begin{remark}\label{rem:subsystem preserve Borel sets}
    Note that in Definition~\ref{def:exact}, we do not require $\mu$ to be a Borel measure. 
    In the case when $F \in \subsystem$ is a subsystem of some expanding Thurston map $f$ with respect to some Jordan curve $\mathcal{C} \subseteq S^2$ containing $\post{f}$ and $\mu$ is a Borel measure on $\limitset \define \limitset(F, \mathcal{C})$, the set $(F|_{\limitset})^{n}(E)$ is a Borel set for each $n \in \n$ and each Borel subset $E \subseteq \limitset$. 
    Indeed, a Borel set $E \subseteq \limitset$ can be covered by $n$-tiles in the cell decompositions of $S^2$ induced by $f$ and $\mathcal{C}$. 
    For each $n$-tile $X \in \Tile{n}$, the restriction $f^n|_{X}$ of $f^n$ to $X$ is a homeomorphism from the closed set $X$ onto $f^n(X)$ by \cite[Proposition~5.16~(i)]{bonk2017expanding}. 
    Thus the set $f^n(E)$ is Borel. 
    Recall from Subsection~\ref{sub:Subsystems of expanding Thurston maps} that $F|_{\limitset} = f|_{\limitset}$ and $F(\limitset) \subseteq \limitset$.
    It is then clear that the set $(F|_{\limitset})^{n}(E)$ is also Borel.
\end{remark}

We now prove that the measure-preserving transformation $\limitmap \define F|_{\limitset}$ of the probability space $(\limitset, \equstate)$ is exact. 
We follow the conventions discussed in Remarks~\ref{rem:disjoint union} and \ref{rem:probability measure in split setting}. 

\begin{theorem}   \label{thm:subsystem exact with respect to equilibrium state}
    Let $f$, $\mathcal{C}$, $F$, $d$, $\potential$ satisfy the Assumptions in Section~\ref{sec:The Assumptions}. 
    We assume in addition that $f(\mathcal{C}) \subseteq \mathcal{C}$ and $F \in \subsystem$ is strongly primitive.
    Denote $\limitmap \define F|_{\limitmap}$.
    Let $\equstate$ be the unique equilibrium state for $\limitmap$ and $\potential|_{\limitset}$, and $\completionequstate$ its completion.
    Then the measure-preserving transformation $\limitmap$ of the probability space $(\limitset, \equstate)$ (\resp $(\limitset, \completionequstate)$) is exact.
\end{theorem}
\begin{proof}
    \def\set{\limitset}  \def\map{\limitmap}
    Recall from Theorems~\ref{thm:subsystem characterization of pressure and existence of equilibrium state} and~\ref{thm:existence of f invariant Gibbs measure} that $\equstate = \splmea = \eigfun \spleigmea$, where $\eigmea = \spleigmea$ is an eigenmeasure of $\dualsplopt$ from Theorem~\ref{thm:subsystem:eigenmeasure existence and basic properties} and $\eigfun$ is an eigenfunction of $\splopt$ from Theorem~\ref{thm:existence of f invariant Gibbs measure}.
    Then by \eqref{eq:two sides bounds for eigenfunction} in Theorem~\ref{thm:existence of f invariant Gibbs measure}, it suffices to prove that 
    \begin{equation*}
        \lim_{n \to +\infty} \eigmea(\set \setminus \map^{n}(A))  = 0
    \end{equation*}
    for each Borel set $A \subseteq \set$ with $\eigmea(A) > 0$.
    We follow the conventions discussed in Remark~\ref{rem:invariant measure equivalent} so that $\eigmea \in \probmea{\set} \subseteq \probsphere$.

    \def\tileapproximation{\mathbf{T}^{n}} \def\approximation{T^{n}}
    \def\bnftile{X^{n_{F}}_{\black}}  \def\wnftile{X^{n_{F}}_{\white}}
    Let $A \subseteq \set$ be an arbitrary Borel subset of $\set$ with $\eigmea(A) > 0$.
    Fix an arbitrary $\varepsilon > 0$.
    By the regularity of $\eigmea$ there exists a compact set $K \subseteq A$ and an open set $U \subseteq S^2$ with $K \subseteq A \subseteq U$ and $\eigmea(U \setminus K) < \varepsilon$.
    Since the diameters of tiles approach $0$ uniformly as their levels becomes larger, there exists $N \in \n$ such that for each integer $n \geqslant N$, every $n$-tile that meets $K$ is contained in the open neighborhood $U$ of $K$.
    For each $n \in \n$, we define
    \begin{equation*}
        \tileapproximation \define \{ X^n \in \Domain{n} \describe X^n \cap K \ne \emptyset \} \quad \text{ and } \quad \approximation \define \bigcup \tileapproximation.
    \end{equation*}
    Then for each integer $n \geqslant N$, we have $K \subseteq \approximation \subseteq U$ and $\eigmea(\approximation \setminus A) \leqslant \eigmea(U \setminus K) < \varepsilon$. 
    Thus, it follows from Theorem~\ref{thm:subsystem:eigenmeasure existence and basic properties}~\ref{item:thm:subsystem:eigenmeasure existence and basic properties:strongly irreducible edge zero measure} that $\sum_{X^{n} \in \tileapproximation} \eigmea(X^{n} \setminus K) < \varepsilon$.
    This implies
    \begin{equation}    \label{eq:temp:thm:subsystem exact with respect to equilibrium state:measure tile}
        \frac{ \sum_{X^{n} \in \tileapproximation} \eigmea(X^{n} \setminus K) }{ \sum_{X^{n} \in \tileapproximation} \eigmea(X^{n}) } < \frac{\varepsilon}{\eigmea(K)}.
    \end{equation}
    Hence for each integer $n \geqslant N$, there exists an $n$-tile $Y^{n} \in \tileapproximation$ such that \[
        \frac{ \eigmea(Y^{n} \setminus K) }{\eigmea(Y^{n})} < \frac{\varepsilon}{\eigmea(K)}.
    \]
    By Proposition~\ref{prop:subsystem:preliminary properties}~\ref{item:subsystem:properties:homeo}, the map $F^{n}$ is injective on $Y^{n}$.
    Then it follows from Theorem~\ref{thm:subsystem:eigenmeasure existence and basic properties}~\ref{item:thm:subsystem:eigenmeasure existence and basic properties:Jacobian}, Lemma~\ref{lem:distortion_lemma}, and \eqref{eq:temp:thm:subsystem exact with respect to equilibrium state:measure tile} that 
    \[
        \begin{split}
            \frac{ \eigmea( F^{n}(Y^{n}) \setminus F^{n}(K) ) }{ \eigmea(F^{n}(Y^{n})) } 
            &\leqslant \frac{ \eigmea( F^{n}( Y^{n} \setminus K ) ) }{ \eigmea(F^{n}(Y^{n})) } 
            = \frac{ \int_{ Y^{n} \setminus K } \myexp{-S_{n}\potential} \,\mathrm{d}\eigmea }{ \int_{ Y^{n} } \myexp{-S_{n}\potential} \,\mathrm{d}\eigmea }  \\
            &\leqslant C \frac{ \eigmea(Y^{n} \setminus K) }{\eigmea(Y^{n})} \leqslant \frac{C \varepsilon}{\eigmea(K)},
        \end{split}
    \]
    where $C \define \myexp[\big]{ \Cdistortion }$ and $C_{1} \geqslant 0$ is the constant defined in \eqref{eq:const:C_1} in Lemma~\ref{lem:distortion_lemma} which depends only on $f$, $\mathcal{C}$, $d$, $\phi$, and $\holderexp$.
    Let $n_{F} \in \n$ be the constant from Definition~\ref{def:primitivity of subsystem}, which depends only on $f$ and $\mathcal{C}$.
    Note that it follows from Proposition~\ref{prop:subsystem:preliminary properties}~\ref{item:subsystem:properties:homeo} that $F^{n}(Y^{n}) = X^{0}_{\colour}$ for some $\colour \in \colours$.
    Since $F$ is strongly primitive, by Lemma~\ref{lem:strongly primitive:tile in interior tile for high enough level}, there exist $\bnftile \in \ccFTile{n_{F}}{\black}{\colour}$ and $\wnftile \in \ccFTile{n_{F}}{\white}{\colour}$ such that $\bnftile \cup \wnftile \subseteq X^{0}_{\colour} = F^{n}(Y^{n})$.
    We claim that
    \begin{equation} \label{eq:temp:thm:subsystem exact with respect to equilibrium state:bound set by tile} 
        \set = F^{n + n_{F}}(Y^{n} \cap \set).
    \end{equation}
    
    Indeed, since $F(\set) \subseteq \set$ (recall Proposition~\ref{prop:subsystem:preliminary properties}~\ref{item:subsystem:properties:limitset forward invariant}), it suffices to show that $\set \subseteq F^{n + n_{F}}(Y^{n} \cap \set)$.
    For each $x \in \limitset$, by \eqref{eq:def:limitset}, there exists a sequence of tiles $\sequen[\big]{X^k}[k]$ such that $\{x\} = \bigcap_{k \in \n} X^{k}$ and $X^k \in \Domain{k}$ for each $k \in \n$.
    By Proposition~\ref{prop:cell decomposition: invariant Jordan curve}, we may assume without loss of generality that $X^k \subseteq X^0_{\black}$ for each $k \in \n$.
    Since $F$ is strongly primitive, by Lemma~\ref{lem:strongly primitive:tile in interior tile for high enough level}, there exists $X^{n + n_{F}}_{\black} \in \bFTile{n + n_{F}}$ such that $X^{n + n_{F}}_{\black} \subseteq Y^n$ and $F^{n + n_{F}}\bigl(X^{n + n_{F}}_{\black}\bigl) = X^0_{\black}$.
    Then it follows from Proposition~\ref{prop:subsystem:preliminary properties}~\ref{item:subsystem:properties:homeo} and Lemma~\ref{lem:cell mapping properties of Thurston map}~\ref{item:lem:cell mapping properties of Thurston map:i} that 
    $Y^{k + n + n_{F}} \define \bigl( F^{n + n_{F}}|_{X^{n + n_{F}}_{\black}} \bigr)^{-1} \bigl( X^{k} \bigr) \in \Domain{k + n + n_{F}}$ for each $k \in \n$. 
    Set $y \define \bigl( F^{n + n_{F}}|_{X^{n + n_{F}}_{\black}} \bigr)^{-1}(x)$.
    Note that $y \in Y^{k + n + n_{F}} \subseteq Y^{n}$ for each $k \in \n$.
    Thus by \eqref{eq:def:limitset} and Proposition~\ref{prop:subsystem:preliminary properties}~\ref{item:subsystem:properties:properties invariant Jordan curve:decreasing relation of domains}, we conclude that $y \in Y^n \cap \set$ and \eqref{eq:temp:thm:subsystem exact with respect to equilibrium state:bound set by tile} holds.

    By \eqref{eq:temp:thm:subsystem exact with respect to equilibrium state:bound set by tile}, we get
    \[
        \begin{split}
            \set \setminus \map^{n + n_{F}}(K) 
            &= F^{n + n_{F}}(Y^{n} \cap \set) \setminus F^{n + n_{F}}(K) \subseteq F^{n_{F}} ( F^{n}(Y^{n} \cap \set) \setminus F^{n}(K) )  \\
            &\subseteq F^{n_{F}} ( ( F^{n}(Y^{n}) \cap \set ) \setminus F^{n}(K) ) = F^{n_{F}} ( (F^{n}(Y^{n}) \setminus F^{n}(K)) \cap \set ).
        \end{split}
    \] 
    Hence, by Theorem~\ref{thm:subsystem:eigenmeasure existence and basic properties}~\ref{item:thm:subsystem:eigenmeasure existence and basic properties:Jacobian} and \ref{item:thm:subsystem:eigenmeasure existence and basic properties:Jacobian:Gibbs property}, for each integer $n \geqslant N$,
    \[
        \begin{split}
            \eigmea( \set \setminus \map^{n + n_{F}}(K) )
            &\leqslant \eigmea( F^{n_{F}} ( (F^{n}(Y^{n}) \setminus F^{n}(K)) \cap \set ) ) \\
            &\leqslant \int_{F^{n}(Y^{n}) \setminus F^{n}(K)} \myexp{ n_{F} P(F, \potential) - S_{n_{F}}\potential } \,\mathrm{d} \eigmea  \\
            &\leqslant \myexp{ n_{F}P(F, \potential) + n_{F}\uniformnorm{\potential} } C \varepsilon / \eigmea(K).
        \end{split}
    \]
    Since $\varepsilon > 0$ was arbitrary, we get $\lim_{n \to +\infty} \eigmea \parentheses[\big]{ \set \setminus \map^{n + n_{F}}(K) }  = 0$.
    This implies
    \[
        \lim_{n \to +\infty} \eigmea(\map^{n}(A)) \geqslant \lim_{n \to +\infty} \eigmea(\map^{n}(K)) = 1.
    \]
    Hence the measure-preserving transformation $\map$ of the probability space $(\set, \equstate)$ is exact.

    Next, we observe that since $\map$ is $\equstate$-measurable, and is a non-singular measure-preserving transformation of the probability space $(\set, \equstate)$ by Theorem~\ref{thm:subsystem:uniqueness of equilibrium state}, it follows that $\map$ is also $\completionequstate$-measurable, and is a measure-preserving transformation of the probability space $(\set, \completionequstate)$.

    To prove that the measure-preserving transformation $\set$ of the probability space $(\set, \completionequstate)$ is exact, we consider a $\completionequstate$-measurable set $B \subseteq \set$ with $\completionequstate(B) > 0$.
    Since $\completionequstate \in \probmea{\set}$ is the completion of the Borel probability measure $\equstate \in \probmea{\set}$, we can choose Borel subsets $A$ and $C$ of $\set$ such that $A \subseteq B \subseteq C \subseteq \set$ and $\completionequstate(B) = \completionequstate(A) = \completionequstate(C) = \equstate(A) = \equstate(C)$.
    For each $n \in \n$, we have $\map^{n}(A) \subseteq \map^{n}(B) \subseteq \map^{n}(C)$, and both $\map^{n}(A)$ and $\map^{n}(C)$ are Borel sets by Remark~\ref{rem:subsystem preserve Borel sets}.
    Since $\map$ is forward quasi-invariant with respect to $\equstate$ by Theorem~\ref{thm:subsystem:uniqueness of equilibrium state}, it follows that $\equstate(\map^{n}(A)) = \equstate(\map^{n}(C))$.
    Thus
    \[
        \equstate(\map^{n}(A)) = \completionequstate(\map^{n}(A)) = \completionequstate(\map^{n}(B)) = \completionequstate(\map^{n}(C)) = \equstate(\map^{n}(C)).
    \]
    Therefore, $\lim_{n \to +\infty} \completionequstate(\map^{n}(B)) = \lim_{n \to +\infty} \equstate(\map^{n}(A)) = 1$.
\end{proof}

The following corollary strengthens \cite[Theorem~6.16~(ii)]{shi2024thermodynamic} for strongly primitive subsystems.

\begin{corollary}    \label{coro:subsystem equilibrium state non-atomic}
    Let $f$, $\mathcal{C}$, $F$, $d$, $\potential$ satisfy the Assumptions in Section~\ref{sec:The Assumptions}. 
    We assume in addition that $f(\mathcal{C}) \subseteq \mathcal{C}$ and $F \in \subsystem$ is strongly primitive. 
    Let $\equstate$ be the unique equilibrium state for $\limitmap$ and $\potential|_{\limitset}$, and $\eigmea$ be as in Proposition~\ref{prop:uniqueness of eigenmeasure subsystem}.
    Then both $\equstate$ and $\eigmea$ as well as their corresponding completions are non-atomic.
\end{corollary} 
Recall that a measure $\mu$ on a topological space $X$ is called \defn{non-atomic} if $\mu(\{x\}) = 0$ for each $x \in X$.
\begin{proof}
    \def\set{\limitset}  \def\map{\limitmap}
    Recall from Theorems~\ref{thm:subsystem characterization of pressure and existence of equilibrium state} and \ref{thm:existence of f invariant Gibbs measure} that $\equstate = \splmea = \eigfun \spleigmea$, where $\eigmea = \spleigmea$ is an eigenmeasure of $\dualsplopt$ and $\eigfun$ is an eigenfunction of $\splopt$.
    Then by \eqref{eq:two sides bounds for eigenfunction} in Theorem~\ref{thm:existence of f invariant Gibbs measure}, it suffices to prove that $\equstate$ is non-atomic.

    Suppose that there exists a point $x \in \set$ with $\equstate(\{x\}) > 0$, then for each $y \in \set$, we have
    \[
        \equstate(\{y\}) \leqslant \max \{ \equstate(\{x\}), 1 - \equstate(\{x\}) \}.
    \]
    Since the transformation $\map$ of $(\set, \equstate)$ is exact by Theorem~\ref{thm:subsystem exact with respect to equilibrium state}, it follows that $\equstate(\{x\}) = 1$ and $\map(x) = x$.
    By Lemma~\ref{lem:strongly primitive:tile in interior tile for high enough level}, there exist $n \in \n$ and $X^{n} \in \Domain{n}$ such that $x \notin X^{n}$.
    This implies $\equstate(X^{n}) = 0$, which contradicts with the fact that $\equstate$ is a Gibbs measure for $F$, $\mathcal{C}$, and $\potential$ (see Theorem~\ref{thm:existence of f invariant Gibbs measure} and Definition~\ref{def:subsystem gibbs measure}).

    The fact that the completions are non-atomic now follows immediately.
\end{proof}

\begin{remark}\label{rem:Lebesgue space exactness implies mixing}
    Let $f$, $\mathcal{C}$, $F$, $d$, $\potential$ satisfy the Assumptions in Section~\ref{sec:The Assumptions}. 
    We assume in addition that $f(\mathcal{C}) \subseteq \mathcal{C}$ and $F \in \subsystem$ is strongly primitive.
    Let $\equstate$ be the unique equilibrium state for $\limitmap$ and $\potential|_{\limitset}$ from Theorem~\ref{thm:subsystem:uniqueness of equilibrium state}, and $\completionequstate$ its completion. 
    Then by Theorem~2.7 in \cite{rohlin1949fundamental}, the complete separable metric space $(\limitset, d)$ equipped the complete non-atomic measure $\completionequstate$ is a Lebesgue space in the sense of V.~Rokhlin. 
    We omit V.~Rokhlin's definition of a \emph{Lebesgue space} here and refer the reader to \cite[Section~2]{rohlin1949fundamental}, since the only results we will use about Lebesgue spaces are V.~Rokhlin's definition of exactness of a measure-preserving transformation on a Lebesgue space and its implication to the mixing properties. 
    More precisely, in \cite{rohlin1964exact}, V.~Rokhlin gave a definition of exactness for a measure-preserving transformation on a Lebesgue space equipped with a complete non-atomic measure, and showed \cite[Section~2.2]{rohlin1964exact} that in such a context, it is equivalent to our definition of exactness in Definition~\ref{def:exact}. 
    Moreover, he proved \cite[Section~2.6]{rohlin1964exact} that if a measure-preserving transformation on a Lebesgue space equipped with a complete non-atomic measure is exact, then it is \emph{mixing} (he actually proved that it is \emph{mixing of all degrees}, which we will not discuss here).
\end{remark}

It is well-known and easy to see that if $g$ is mixing (recall \eqref{eq:def:mixing}), then it is ergodic (see for example, \cite[Theorem~1.17]{walters1982introduction}).

\begin{corollary}    \label{coro:subsystem mixing}
    Let $f$, $\mathcal{C}$, $F$, $d$, $\potential$ satisfy the Assumptions in Section~\ref{sec:The Assumptions}. 
    We assume in addition that $f(\mathcal{C}) \subseteq \mathcal{C}$ and $F \in \subsystem$ is strongly primitive.
    Let $\equstate$ be the unique equilibrium state for $\limitmap$ and $\potential|_{\limitset}$, and $\completionequstate$ its completion. 
    Then the measure-preserving transformation $\limitmap$ of the probability space $(\limitset, \equstate)$ (\resp $(\limitset, \completionequstate)$) is mixing and ergodic.
\end{corollary}
\begin{proof}
    By Remark~\ref{rem:Lebesgue space exactness implies mixing}, the measure-preserving transformation $\limitmap$ of $(\limitset, \completionequstate)$ is mixing and thus ergodic. 
    Since any $\equstate$-measurable sets $\juxtapose{A}{B} \subseteq S^2$ are also $\completionequstate$-measurable, the measure-preserving transformation $\limitmap$ of $(\limitset, \equstate)$ is also mixing and ergodic.
\end{proof}

\begin{definition}    \label{def:topological transitive and mixing}
  Let $T \colon X \mapping X$ be a continuous map on a topological space $X$.
  We say $T \colon X \mapping X$ is \emph{topologically transitive} if for any non-empty open subsets $\juxtapose{U}{V}$ of $X$, there exists $n \in \n_0$ such that $T^{n}(U) \cap V \ne \emptyset$.
  We say $T \colon X \mapping X$ is \emph{topologically mixing} if for any non-empty open subsets $\juxtapose{U}{V}$ of $X$, there exists $N \in \n_0$ such that $T^{n}(U) \cap V \ne \emptyset$ for any integer $n \geqslant N$.
\end{definition}

The following proposition is not used in this paper but should be of independent interest.

\begin{proposition}    \label{prop:irreducible and primitive subsystem is topological transitive and mixing}
  Let $f$, $\mathcal{C}$, $F$ satisfy the Assumptions in Section~\ref{sec:The Assumptions}.
  We assume in addition that $f(\mathcal{C}) \subseteq \mathcal{C}$.
  Then the following statements hold:
  \begin{enumerate}[label=\rm{(\roman*)}]
      \smallskip
      \item    \label{item:prop:irreducible and primitive subsystem is topological transitive and mixing:irreducibility implies transitivity} 
          If $F$ is irreducible, then $F|_{\limitset} \colon \limitset \mapping \limitset$ is topological transitive and has a dense forward orbit.
      \smallskip
        
      \item    \label{item:prop:irreducible and primitive subsystem is topological transitive and mixing:primitivity implies mixing} 
          If $F$ is primitive, then $F|_{\limitset} \colon \limitset \mapping \limitset$ is topological mixing.
  \end{enumerate}
\end{proposition}
\begin{proof}
    \def\set{\limitset}
    \ref{item:prop:irreducible and primitive subsystem is topological transitive and mixing:irreducibility implies transitivity}
    Assume that $F$ is irreducible.
    Note that $F(\limitset) = \limitset$ by Proposition~\ref{prop:subsystem:preliminary properties}~\ref{item:subsystem:properties:sursubsystem properties:property of domain and limitset}.

    We first show that $F|_{\limitset} \colon \limitset \mapping \limitset$ is topological transitive.
    Consider arbitrary non-empty open subsets $U \cap \limitset$ and $V \cap \limitset$ of $\limitset$, where $U$ and $V$ are open subsets of $S^2$.
    Since $U \cap \limitset \ne \emptyset$ and $U$ is open, by \eqref{eq:def:limitset}, Definition~\ref{def:expanding_Thurston_maps}, and Remark~\ref{rem:Expansion_is_independent}, there exist $N \in \n$ and $X^N \in \Domain{N}$ such that $X^N \subseteq U$.

    Fix arbitrary $y \in V \cap \set$.
    Then by \eqref{eq:def:limitset}, there exists a sequence $\sequen[\big]{Y^k}[k]$ of tiles such that $\{ y \} = \bigcap_{k \in \n} Y^{k}$ and $Y^k \in \Domain{k}$ for each $k \in \n$.
    By Proposition~\ref{prop:cell decomposition: invariant Jordan curve}, we may assume without loss of generality that $Y^k \subseteq X^0_{\black}$ for each $k \in \n$.    
    Since $F$ is irreducible, by Lemma~\ref{lem:strongly irreducible:tile in interior tile for high enough level}, there exist $n_0 \in \n$ and $X^{N + n_{0}}_{\black} \in \bFTile{N + n_{0}}$ such that $X^{N + n_{0}}_{\black} \subseteq X^{N}$.
    Then it follows from Lemma~\ref{lem:cell mapping properties of Thurston map}~\ref{item:lem:cell mapping properties of Thurston map:i} and Proposition~\ref{prop:subsystem:preliminary properties}~\ref{item:subsystem:properties:homeo} that $X^{k + N + n_{0}} \define \bigl( F^{N + n_{0}}|_{X^{N + n_{0}}_{\black}} \bigr)^{-1} \bigl( Y^{k} \bigr) \in \Domain{k + N + n_{0}}$ for each $k \in \n$. 
    Set $x \define \bigl( F^{N + n_{0}}|_{X^{N + n_{0}}_{\black}} \bigr)^{-1}(y)$.
    Note that for each $k \in \n$ we have $x \in X^{k + N + n_{0}} \subseteq X^{N}$ since $y \in Y^k$.
    Thus by \eqref{eq:def:limitset} and Proposition~\ref{prop:subsystem:preliminary properties}~\ref{item:subsystem:properties:properties invariant Jordan curve:decreasing relation of domains}, we conclude that $x \in X^N \cap \set$.
    This implies $y = F^{N + n_{0}}(x) \in F^{N + n_{0}}\bigl(X^N \cap \set\bigr) \subseteq F^{N + n_{0}}(U \cap \set)$.
    Since $y \in V \cap \set$, it follows immediately from Definition~\ref{def:topological transitive and mixing} that $F|_{\set} \colon \set \mapping \set$ is topologically transitive.

    We now prove that there exists $x \in \set$ such that $\sequen{F^{n}(x)}$ is dense in $\set$.
    By \eqref{eq:def:limitset}, Definition~\ref{def:expanding_Thurston_maps}, and Remark~\ref{rem:Expansion_is_independent}, it suffices to show that there exists $x \in \set$ such that for each $Y \in \bigcup_{n \in \n} \Domain{n}$, there exists $k \in \n$ satisfying $F^{k}(x) \in Y$.

    \def\mytile#1{X^{n_{#1}}_{\colour_{#1}}} \def\tileset#1{\ccFTile{n_{#1}}{\colour_{#1}}{}}
    Since for each $n \in \n$ the set $\Domain{n}$ is finite, we can write the countable set $\bigcup_{n \in \n} \Domain{n}$ as $\sequen{Y_{n}}$.
    By Proposition~\ref{prop:cell decomposition: invariant Jordan curve}, for each $n \in \n$ there exists $\colour_{n} \in \colours$ such that $Y_{n} \subseteq X^{0}_{\colour_{n}}$.
    Since $F$ is irreducible, by Lemma~\ref{lem:strongly irreducible:tile in interior tile for high enough level}, there exist $n_1 \in \n$ and $\mytile{1} \in \tileset{1}$ such that $\mytile{1} \subseteq Z_0 \define Y_1$.
    Then it follows from Lemma~\ref{lem:cell mapping properties of Thurston map}~\ref{item:lem:cell mapping properties of Thurston map:i} and Proposition~\ref{prop:subsystem:preliminary properties}~\ref{item:subsystem:properties:homeo} that $Z_{1} \define \bigl( F^{n_{1}}|_{\mytile{1}} \bigr)^{-1} \bigl( Y_{1} \bigr) \in \Domain{k_{1}}$ for some $k_{1} \in \n$.
    Note that $k_1 > n_1$ and $Z_1 \subseteq \mytile{1}$.
    Similarly, by Lemma~\ref{lem:strongly irreducible:tile in interior tile for high enough level}, there exist $n_2 \in \n$ and $\mytile{2} \in \tileset{2}$ such that $\mytile{2} \subseteq Z_1$ and $n_{2} > k_1$.
    In particular, we have $F^{n_{1}}(\mytile{2}) \subseteq F^{n_{1}}(Z_1) = Y_{1}$.
    It follows from Lemma~\ref{lem:cell mapping properties of Thurston map}~\ref{item:lem:cell mapping properties of Thurston map:i} and Proposition~\ref{prop:subsystem:preliminary properties}~\ref{item:subsystem:properties:homeo} that $Z_{2} \define \bigl( F^{n_{2}}|_{\mytile{2}} \bigr)^{-1} \bigl( Y_{2} \bigr) \in \Domain{k_{2}}$ for some $k_{2} \in \n$.
    Note that $k_2 > n_2$ and $Z_2 \subseteq \mytile{2}$.
    Then we can inductively construct a strictly increasing sequence $\sequen{n_{j}}[j]$ of integers and a sequence $\sequen[\big]{\mytile{j}}[j]$ of tiles such that $\mytile{j} \in \tileset{j}$, $\mytile{j + 1} \subseteq \mytile{j}$, and $F^{n_{j}}\bigl(\mytile{j + 1}\bigr) \subseteq Y_{j}$ for each $j \in \n$.
    Since $\lim_{j \to +\infty} n_{j}  = +\infty$, it follows from Definition~\ref{def:expanding_Thurston_maps} and Remark~\ref{rem:Expansion_is_independent} that the set $\bigcap_{j \in \n} \mytile{j}$ is a singleton set.
    We write $\{ x \} = \bigcap_{j \in \n} \mytile{j}$.
    Then by \eqref{eq:def:limitset} and Proposition~\ref{prop:subsystem:preliminary properties}~\ref{item:subsystem:properties:properties invariant Jordan curve:decreasing relation of domains}, we have $x \in \limitset$.
    This finishes the proof since $F^{n_{j}}(x) \in F^{n_{j}} \bigl(\mytile{j + 1}\bigr) \subseteq Y_j$ for each $j \in \n$.

    \smallskip

    \ref{item:prop:irreducible and primitive subsystem is topological transitive and mixing:primitivity implies mixing} 
    Let $F$ be primitive and $n_{F}$ be the constant from Definition~\ref{def:primitivity of subsystem}, which depends only on $F$ and $\mathcal{C}$.
    Note that $F(\limitset) = \limitset$ by Proposition~\ref{prop:subsystem:preliminary properties}~\ref{item:subsystem:properties:sursubsystem properties:property of domain and limitset}.
    Consider arbitrary non-empty open subsets $U \cap \limitset$ and $V \cap \limitset$ of $\limitset$, where $U$ and $V$ are open subsets of $S^2$.
    Since $U \cap \limitset \ne \emptyset$ and $U$ is open, by \eqref{eq:def:limitset}, Definition~\ref{def:expanding_Thurston_maps}, and Remark~\ref{rem:Expansion_is_independent}, there exist $N \in \n$ and $X^N \in \Domain{N}$ such that $X^N \subseteq U$.
    We claim that $\limitset \subseteq F^{N + n_{F}}\bigl(X^N \cap \limitset\bigr)$.

    Indeed, for each $y \in \limitset$, by \eqref{eq:def:limitset}, there exists a sequence $\sequen[\big]{Y^k}[k]$ of tiles such that $\{ y \} = \bigcap_{k \in \n} Y^{k}$ and $Y^k \in \Domain{k}$ for each $k \in \n$.
    By Proposition~\ref{prop:cell decomposition: invariant Jordan curve}, we may assume without loss of generality that $Y^k \subseteq X^0_{\black}$ for each $k \in \n$.
    Since $F$ is primitive, by Lemma~\ref{lem:strongly primitive:tile in interior tile for high enough level}, there exists $X^{N + n_{F}}_{\black} \in \bFTile{N + n_{F}}$ such that $X^{N + n_{F}}_{\black} \subseteq X^{N}$.
    Then it follows from Lemma~\ref{lem:cell mapping properties of Thurston map}~\ref{item:lem:cell mapping properties of Thurston map:i} and Proposition~\ref{prop:subsystem:preliminary properties}~\ref{item:subsystem:properties:homeo} that $X^{k + N + n_{F}} \define \bigl( F^{N + n_{F}}|_{X^{N + n_{F}}_{\black}} \bigr)^{-1} \bigl( Y^{k} \bigr) \in \Domain{k + N + n_{F}}$ for each $k \in \n$. 
    Set $x \define \bigl( F^{N + n_{F}}|_{X^{N + n_{F}}_{\black}} \bigr)^{-1}(y)$.
    Note that for each $k \in \n$ we have $x \in X^{k + N + n_{F}} \subseteq X^{N}$ since $y \in Y^k$.
    Thus by \eqref{eq:def:limitset} and Proposition~\ref{prop:subsystem:preliminary properties}~\ref{item:subsystem:properties:properties invariant Jordan curve:decreasing relation of domains}, we conclude that $x \in X^N \cap \set$.
    This implies $\limitset \subseteq F^{N + n_{F}}(X^N \cap \limitset)$ since $y \in \limitset$ is arbitrary and $y = F^{N + n_{F}}(x)$.

    Hence, for each integer $n \geqslant N + n_{F}$, we have $(F|_{\limitset})^{n}(U \cap \limitset) = F^{n}(U \cap \limitset) \supseteq F^{n}\bigl(X^N \cap \limitset \bigr) \supseteq F^{n - N - n_{F}}(\limitset) = \limitset \supseteq V \cap \limitset \ne \emptyset$.
    This implies that $F|_{\limitset} \colon \limitset \mapping \limitset$ is topologically mixing by Definition~\ref{def:topological transitive and mixing}.
\end{proof} 

\section{Equidistribution}
\label{sec:Equidistribution}

In this section, we establish equidistribution results for preimages for subsystems of expanding Thurston maps.

\begin{theorem}    \label{thm:split version subsystem equidistribution for preimages} \def\ncolour{\colour_{n}}
    Let $f$, $\mathcal{C}$, $F$, $d$, and $\potential$ satisfied the Assumptions in Section~\ref{sec:The Assumptions}.
    We assume in addition that $f(\mathcal{C}) \subseteq \mathcal{C}$ and $F \in \subsystem$ is strongly primitive.
    Let $\equstate$ be the unique equilibrium state for $F|_{\limitset}$ and $\potential|_{\limitset}$, and let $\eigmea$ be as in Proposition~\ref{prop:uniqueness of eigenmeasure subsystem}.
    Fix arbitrary sequence $\sequen{x_{n}}$ of points in $S^2$ and sequence $\sequen{\ncolour}$ of colors in $\colours$ that satisfies $x_{n} \in X^0_{\ncolour}$ for each $n \in \n$.
    For each $n \in \n$, we define the Borel probability measures
    \begin{align*}
        \nu_{n} &\define \frac{1}{Z_{n}(\potential)} \sum_{ y \in F^{-n}(x_{n}) } \ccndegF{\ncolour}{}{n}{y} \myexp[\big]{ S_{n}^{F} \potential(y) } \delta_{y}, \\
        \widehat{\nu}_{n} &\define \frac{1}{Z_{n}(\potential)} \sum_{ y \in F^{-n}(x_{n}) } \ccndegF{\ncolour}{}{n}{y} \myexp[\big]{ S_{n}^{F} \potential(y) } \frac{1}{n} \sum_{i = 0}^{n - 1} \delta_{F^{i}(y)},
    \end{align*}
    where $Z_{n}(\potential) \define \sum_{ y \in F^{-n}(x_{n}) } \ccndegF{\ncolour}{}{n}{y} \myexp[\big]{S_{n}^{F}\potential(y)}$.
    Then we have
    \begin{align}
        \label{eq:prop:split version subsystem equidistribution for preimages:eigenmeasure}
        \nu_{n} &\weakconverge \eigmea \quad \text{ as } n \to +\infty, \\
        \label{eq:prop:split version subsystem equidistribution for preimages:equilibrium state}
        \widehat{\nu}_{n} &\weakconverge \equstate \quad \text{ as } n \to +\infty.
    \end{align}
\end{theorem}

We follow the conventions discussed in Remarks~\ref{rem:disjoint union} and \ref{rem:probability measure in split setting} in this subsection.
Recall that $\splitsphere$ is the split sphere defined in Definition~\ref{def:split sphere}.
\begin{proof}
    \def\ncolour{\colour_{n}}
    Fix an arbitrary $u \in C(S^{2})$.
    We denote by $\widetilde{u}$ the continuous function on $\splitsphere$ given by $\widetilde{u}(\widetilde{z}) \define u(z)$ for each $\widetilde{z} = (z, \colour) \in \splitsphere$. 

    We prove \eqref{eq:prop:split version subsystem equidistribution for preimages:eigenmeasure} by showing that $\lim_{n \to +\infty} \functional{\nu_{n}}{u} = \functional{\eigmea}{u}$.
    Indeed, by Lemma~\ref{lem:iteration of split-partial ruelle operator}, Definition~\ref{def:partial Ruelle operator}, and \eqref{eq:def:normed potential}, we have 
    \[
        \functional{\nu_{n}}{u} = \frac{ \splopt^{n}(\widetilde{u})(x_{n}, \ncolour) }{ \splopt^{n}\bigl(\indicator{\splitsphere}\bigr)(x_{n}, \ncolour) } 
        = \frac{ \normsplopt^{n}(\widetilde{u})(x_{n}, \ncolour) }{ \normsplopt^{n}\bigl(\indicator{\splitsphere}\bigr)(x_{n}, \ncolour) }.
    \]
    Note that by \eqref{eq:converge of uniform norm of normed split ruelle operator} in Theorem~\ref{thm:converge of uniform norm of normalized split ruelle operator},
    \[
        \normcontinuous[\big]{\normsplopt^{n}(\indicator{\splitsphere}) - \eigfun}{\splitsphere} \converge 0    \quad \text{ and } \quad
        \normcontinuous[\bigg]{\normsplopt^{n}(\widetilde{u}) - \eigfun \int \! \widetilde{u} \,\mathrm{d}\spleigmea }{\splitsphere} \converge 0 
    \]
    as $n \to +\infty$, where $\eigfun \in C(\splitsphere)$ is an eigenfunction of $\splopt$ from Theorem~\ref{thm:existence of f invariant Gibbs measure}, and $\spleigmea \in \probmea{\splitsphere}$ is an eigenmeasure of $\dualsplopt$ from Theorem~\ref{thm:subsystem:eigenmeasure existence and basic properties}.
    Then it follows from \eqref{eq:two sides bounds for eigenfunction} in Theorem~\ref{thm:existence of f invariant Gibbs measure} that
    \[
        \lim_{n \to +\infty} \frac{ \normsplopt^{n}(\widetilde{u})(x_{n}, \ncolour) }{ \normsplopt^{n}\bigl(\indicator{\splitsphere}\bigr)(x_{n}, \ncolour) } 
        = \int \! \widetilde{u} \,\mathrm{d}\spleigmea = \int \! u \,\mathrm{d} \eigmea.
    \]
    Hence, \eqref{eq:prop:split version subsystem equidistribution for preimages:eigenmeasure} holds.

    Finally, \eqref{eq:prop:split version subsystem equidistribution for preimages:equilibrium state} follows directly from Lemma~\ref{lem:converge of derivative for pressure of subsystem}.
\end{proof}

\section{Large deviation principles for subsystems}
\label{sec:Large deviation principles for subsystems}

In this section, we establish level-$2$ large deviation principles for iterated preimages for subsystems of expanding Thurston maps without periodic critical point.



\subsection{Level-2 large deviation principles}%
\label{sub:Level-2 large deviation principles}
We first review some basic concepts and results from large deviation theory.
We refer the reader to \cite{dembo2009large,ellis2012entropy,rassoul2015course} for more details. 

Let $\{ \xi_{n} \}_{n \in \n}$ be a sequence of Borel probability measures on a topological space $\mathcal{X}$.
We say that $\{ \xi_{n} \}_{n \in \n}$ satisfies a \emph{large deviation principle} if there exists a lower semi-continuous function $I \colon \mathcal{X} \mapping [0, +\infty]$ such that 
\begin{equation}    \label{eq:level-2 LDP:lower bound}
    \liminf_{n \to +\infty} \frac{1}{n} \log \xi_{n}(\mathcal{G}) \geqslant - \inf_{\mathcal{G}} I  \quad \text{for all open } \mathcal{G} \subseteq \mathcal{X},
\end{equation}
and
\begin{equation}    \label{eq:level-2 LDP:upper bound}
    \limsup_{n \to +\infty} \frac{1}{n} \log \xi_{n}(\mathcal{K}) \leqslant - \inf_{\mathcal{K}} I  \quad \text{for all closed } \mathcal{K} \subseteq \mathcal{X},
\end{equation}
where $\log 0 = -\infty$ and $\inf \emptyset = +\infty$ by convention.
Such a function $I$ is then unique when $\mathcal{X}$ is metrizable; it is called the \emph{rate function}.

\begin{definition}    \label{def:entropy map}
	Let $T \colon X \mapping X$ be a continuous map on a compact metric space $X$.
	The \emph{entropy map} of $T$ is the map $\mu \mapsto h_{\mu}(T)$ which is defined on $\mathcal{M}(X, T)$ and has values in $[0, +\infty]$.
	Here $\mathcal{M}(X, T)$ is equipped with the weak$^*$ topology.
	We say that the entropy map of $T$ is \emph{upper semi-continuous} if $\limsup_{n \to +\infty} h_{\mu_{n}}(T) \leqslant h_{\mu}(T)$ holds for every sequence $\sequen{\mu_{n}}$ of Borel probability measures on $X$ which converges to $\mu \in \invmea[X][T]$ in the weak$^*$ topology.
\end{definition}

We recall the following theorem due to Y.~Kifer \cite[Theorem~4.3]{kifer1990large}, reformulated by H.~Comman and J.~Rivera-Letelier \cite[Theorem~C]{comman2011large}.

\begin{theorem}[Y.~Kifer \cite{kifer1990large}; H.~Comman \& J.~Rivera-Letelier \cite{comman2011large}]        \label{thm:large deviation principles Kifer Comman Juan}
    Let $X$ be a compact metrizable topological space, and let $g \colon X \mapping X$ be a continuous map. 
    Fix $\varphi \in C(X)$, and let $H$ be a dense vector subspace of $C(X)$ with respect to the uniform norm. 
    Let $\ratefun[\phi] \colon \probmea{X} \mapping [0, +\infty]$ be the function defined by
    \begin{equation*}
        \ratefun[\varphi](\mu) \define
        \begin{cases}
            P(g, \varphi) - h_{\mu}(g) - \int\! \varphi \,\mathrm{d}\mu  & \mbox{if } \mu \in \mathcal{M}(X, g); \\
            +\infty & \mbox{if } \mu \in \probmea{X} \setminus \mathcal{M}(X, g).
        \end{cases}
    \end{equation*}
    We assume the following conditions are satisfied:
    \begin{enumerate}
    \smallskip
    \item     \label{item:thm:large deviation principles Kifer Comman Juan:upper semi-continuity of entropy map}
        The measure-theoretic entropy $h_\mu(g)$ of $g$, as a function of $\mu$ defined on $\invmea[X][g]$ (equipped with the weak$^*$ topology), is finite and upper semi-continuous.

    \smallskip
    \item     \label{item:thm:large deviation principles Kifer Comman Juan:uniqueness of equilibrium state for dense potential}
        For each $\psi \in H$, there exists a unique equilibrium state for the map $g$ and the potential $\varphi + \psi$.
    \end{enumerate}

    Then every sequence $\sequen{\xi_n}$ of Borel probability measures on $\probmea{X}$ that satisfies the property that for each $\psi \in H$,
    \begin{equation}   \label{eq:thm:large deviation principles Kifer Comman Juan:requirement for distribution}
        \lim\limits_{n \to +\infty} \frac{1}{n} \log \int_{\probmea{X}} \! \myexp[\bigg]{n \int \! \psi \,\mathrm{d}\mu} \,\mathrm{d}\xi_n(\mu) = P(g, \varphi + \psi) - P(g, \varphi)
    \end{equation}
    satisfies a large deviation principle with rate function $\ratefun[\varphi]$, and it converges in the weak$^*$ topology to the Dirac measure supported on the unique equilibrium state for the map $g$ and the potential $\varphi$. 
    Furthermore, for each convex open subset $\mathcal{G}$ of $\probmea{X}$ containing some invariant measure, we have
    \begin{equation*}
        \lim\limits_{n \to +\infty}  \frac{1}{n} \log \xi_n(\mathcal{G}) 
        = \lim\limits_{n \to +\infty}  \frac{1}{n} \log \xi_n(\overline{\mathcal{G}}) 
        = -\inf\limits_{\mathcal{G}} \ratefun[\varphi] 
        = -\inf\limits_{\overline{\mathcal{G}}} \ratefun[\varphi].
    \end{equation*}
\end{theorem}

Recall that $P(g, \varphi)$ is the topological pressure of the map $g$ with respect to the potential $\varphi$.



\subsection{Characterizations of pressures}\label{sub:Characterizations of pressures} 

Let $f$, $\mathcal{C}$, $F$, $d$, $\potential$ satisfy the Assumptions in Section~\ref{sec:The Assumptions}.
Suppose that $f(\mathcal{C}) \subseteq \mathcal{C}$ and $F \in \subsystem$ is strongly irreducible.
In this subsection, we characterize the pressure function $\pressure$ in terms of Birkhoff averages (Proposition~\ref{prop:subsystem Birkhoff averages pressure characterization}) and iterated preimages (Proposition~\ref{prop:subsystem preimage pressure:recall}).

\begin{proposition}    \label{prop:subsystem Birkhoff averages pressure characterization}
    Let $f$, $\mathcal{C}$, $F$, $d$, $\phi$, $\holderexp$ satisfy the Assumptions in Section~\ref{sec:The Assumptions}. 
    We assume in addition that $f(\mathcal{C}) \subseteq \mathcal{C}$ and $F \in \subsystem$ is strongly irreducible. 
    Denote $\limitset \define \limitset(F, \mathcal{C})$.
    Then for each $\psi \in \holderspacesphere$, we have
    \begin{equation}    \label{eq:prop:subsystem Birkhoff averages pressure characterization}
        \pressure[\potential + \psi] - \pressure = \lim_{n \to +\infty} \frac{1}{n} \log \int \! \myexp[\big]{S_{n}^{F}\psi} \,\mathrm{d} \equstate,
    \end{equation}
    where $\pressure$ and $\pressure[\potential + \psi]$ are defined in \eqref{eq:pressure of subsystem}.     
\end{proposition}
\begin{proof}
    Recall from Theorems~\ref{thm:subsystem characterization of pressure and existence of equilibrium state} and \ref{thm:existence of f invariant Gibbs measure} that $\equstate = \splmea = \eigfun \spleigmea$, where $\eigmea = \spleigmea$ is an eigenmeasure of $\dualsplopt$ from Theorem~\ref{thm:subsystem:eigenmeasure existence and basic properties} and $\eigfun$ is an eigenfunction of $\splopt$ from Theorem~\ref{thm:existence of f invariant Gibbs measure}.
    Recall from Definition~\ref{def:split sphere} and Remark~\ref{rem:disjoint union} that $\splitsphere = X^0_{\black} \sqcup X^0_{\white}$ is the disjoint union of $X^0_{\black}$ and $X^0_{\white}$.
    Note that by Theorem~\ref{thm:subsystem:eigenmeasure existence and basic properties}~\ref{item:thm:subsystem:eigenmeasure existence and basic properties:Jacobian:Gibbs property}, we have $\dualsplopt \spleigmea = e^{\pressure} \spleigmea$.
    Since $\inf_{\splitsphere} \eigfun > 0$ and $\sup_{\splitsphere} \eigfun < +\infty$, it is enough to prove the limit with $\equstate$ replaced by $\eigmea$.

    For each $n \in \n$ we have\[
        \begin{split}
            \int_{S^{2}} \! \myexp[\big]{S_{n}^{F}\psi} \,\mathrm{d} \eigmea 
            &= \int_{\splitsphere} \parentheses[\big]{ e^{S_{n}^{F}\psi}|_{\blacktile}, e^{S_{n}^{F}\psi}|_{\whitetile} }  \,\mathrm{d} \spleigmea \\
            &= \int_{\splitsphere} \parentheses[\big]{ e^{S_{n}^{F}\psi}|_{\blacktile}, e^{S_{n}^{F}\psi}|_{\whitetile} }  \,\mathrm{d} \parentheses[\big]{ e^{-n \pressure} (\dualsplopt)^{n} \spleigmea }  \\
            &= e^{-n \pressure} \int_{\splitsphere} \splopt^{n} \parentheses[\big]{ e^{S_{n}^{F}\psi}|_{\blacktile}, e^{S_{n}^{F}\psi}|_{\whitetile} } \,\mathrm{d} \spleigmea.
        \end{split}
    \]
    Using $\splopt^{n} \parentheses[\big]{ e^{S_{n}^{F}\psi}|_{\blacktile}, e^{S_{n}^{F}\psi}|_{\whitetile} } = \splopt[\potential + \psi]^{n} \indicator{\splitsphere}$, the assertion of the proposition is then a direct consequence of \eqref{eq:equalities for characterizations of pressure} in Theorem~\ref{thm:existence of f invariant Gibbs measure}.
\end{proof}

The following proposition characterizes topological pressures in terms of iterated preimages.

\begin{proposition}    \label{prop:subsystem preimage pressure:recall}
    Let $f$, $\mathcal{C}$, $F$, $d$, $\phi$, $\holderexp$ satisfy the Assumptions in Section~\ref{sec:The Assumptions}. 
    We assume in addition that $f(\mathcal{C}) \subseteq \mathcal{C}$ and $F \in \subsystem$ is strongly irreducible. 
    Denote $\limitset \define \limitset(F, \mathcal{C})$.
    Then for each $y_0 \in \limitset \setminus \mathcal{C}$, we have
    \begin{equation}    \label{eq:prop:subsystem preimage pressure:recall}
        \pressure = P(F|_{\limitset}, \potential|_{\limitset}) = \lim_{n \to +\infty} \frac{1}{n} \log \sum_{x \in (F|_{\limitset})^{-n}(y_{0})} \myexp[\big]{ S_n^{F}\phi(x) },
    \end{equation}
    where $\pressure$ is defined in \eqref{eq:pressure of subsystem} and $P(F|_{\limitset}, \potential|_{\limitset})$ is defined in \eqref{eq:def:topological pressure}. 
\end{proposition}

Proposition~\ref{prop:subsystem preimage pressure:recall} follows immediately from \cite[Proposition~6.21 and Theorem~6.30]{shi2024thermodynamic}. 
Note that $\limitset \setminus \mathcal{C} \ne \emptyset$ by \cite[Proposition~5.20~(ii)]{shi2024thermodynamic}.

\subsection{Proof of large deviation principles}%
\label{sub:Proof of large deviation principles}

In this subsection, we establish Theorem~\ref{thm:level-2 larger deviation principles for subsystem} by applying Theorem~\ref{thm:large deviation principles Kifer Comman Juan}.

By the following two lemmas, we can show that conditions~\ref{item:thm:large deviation principles Kifer Comman Juan:upper semi-continuity of entropy map} and \ref{item:thm:large deviation principles Kifer Comman Juan:uniqueness of equilibrium state for dense potential} in Theorem~\ref{thm:large deviation principles Kifer Comman Juan} are satisfied in our context.

\begin{lemma} \label{lem:upper semi-continuity for submap}
    Let $T \colon X \mapping X$ be a continuous map on a compact metric space $X$.
    Suppose that the entropy map of $T$ is upper semi-continuous.
    Let $Y$ be a compact subset of $X$ with $T(Y) \subseteq Y$.
    Then the entropy map of $T|_{Y}$ is upper semi-continuous.
\end{lemma}
\begin{proof}
    Since $\invmea[Y][T|_{Y}] \subseteq \invmea[X][T]$ and $h_{\mu}(T|_{Y}) = h_{\mu}(T)$ for each $\mu \in \invmea[Y][T|_{Y}]$, the statement follows.
\end{proof}

\begin{lemma} \label{lem:restriction of holder potential to subset}
    Let $(X, d)$ be a metric space and $Y$ be a subset of $X$.
    Then for each $\holderexp \in (0, 1]$, we have\[
        \holderspace[][][Y] = \bigl\{ \psi|_{Y} \describe \psi \in \holderspace \bigr\}.
    \]
\end{lemma}
\begin{proof}
    For each $\psi \in \holderspace$, it follows immediately from the definition of \holder continuity that $\psi|_{Y} \in \holderspace[][][Y]$.
    The converse direction also holds since every function in $\holderspace[][][Y]$ can extend to a function in $\holderspace$ (see for example, \cite[Theorem~6.2 and p.~44]{heinonen2001lectures}).
\end{proof}

Now we are ready to prove the level-2 large deviation principles.

\begin{proof}[Proof of Theorem~\ref{thm:level-2 larger deviation principles for subsystem}]
    First note that by Remark~\ref{rem:chordal metric visual metric qs equiv}, if $f \colon X \mapping X$ is a postcritically-finite rational map with no periodic critical points on the Riemann sphere $X = \ccx$, then the classes of \holder continuous functions on $\ccx$ equipped with the chordal metric and on $S^2 = \ccx$ equipped with any visual metric for $f$ are the same.
    Thus we only need to prove for the case where $f \colon S^2 \mapping S^2$ is an expanding Thurston map with no periodic critical points on a topological 2-sphere $S^2$ equipped with a visual metric $d$ for $f$.

    Let $\potential \in \holderspacesphere$ for some $\holderexp \in (0, 1]$.

    We apply Theorem~\ref{thm:large deviation principles Kifer Comman Juan} with $X = \limitset$, $g = F|_{\limitset}$, $\varphi = \phi|_{\limitset}$, and $H = \holderspace[][][\limitset]$.
    Note that $P(F, \phi) = P(F|_{\limitset}, \phi|_{\limitset})$ by Proposition~\ref{prop:subsystem preimage pressure:recall}, and $\holderspace[][][\limitset]$ is dense in $C(\limitset)$ with respect to the uniform norm by Lemma~\ref{lem:holder space dense in continuous space}.
    By \cite[Lemma~6.4]{shi2024thermodynamic} and \eqref{eq:subsystem Variational Principle} in Theorem~\ref{thm:subsystem characterization of pressure and existence of equilibrium state}, the measure-theoretic entropy $h_{\mu}(F|_{\limitset})$ is finite for each $\mu \in \invmea[\limitset][F|_{\limitset}]$.
    Since $f$ has no periodic critical points, it follows from \cite[Theorem~1.1]{shi2024entropy} that the entropy of $f$ is upper semi-continuous.
    Then by Lemma~\ref{lem:upper semi-continuity for submap}, the entropy map of $F|_{\limitset} = f|_{\limitset}$ is upper semi-continuous.
    Thus, condition~\ref{item:thm:large deviation principles Kifer Comman Juan:upper semi-continuity of entropy map} in Theorem~\ref{thm:large deviation principles Kifer Comman Juan} is satisfied.
    Condition~\ref{item:thm:large deviation principles Kifer Comman Juan:uniqueness of equilibrium state for dense potential} in Theorem~\ref{thm:large deviation principles Kifer Comman Juan} follows from Theorem~\ref{thm:subsystem:uniqueness of equilibrium state} and Lemma~\ref{lem:restriction of holder potential to subset}.

    It now suffices to verify \eqref{eq:thm:large deviation principles Kifer Comman Juan:requirement for distribution} for each of the sequences $\sequen{\birkhoffmeasure}$ and $\sequen{\Omega_{n}(x_{n})}$ of Borel probability measures on $\probmea{\limitset}$.

    Fix an arbitrary $\psi \in \holderspace[][][\limitset]$. 
    By Lemma~\ref{lem:restriction of holder potential to subset}, there exists $\widetilde{\psi} \in \holderspacesphere$ such that $\widetilde{\psi}|_{\limitset} = \psi$.
    
    For the sequence $\sequen{\birkhoffmeasure}$, by \eqref{eq:prop:subsystem Birkhoff averages pressure characterization} in Proposition~\ref{prop:subsystem Birkhoff averages pressure characterization} and \eqref{eq:subsystem Variational Principle} in Theorem~\ref{thm:subsystem characterization of pressure and existence of equilibrium state}, we have
    \[
        \begin{split}
            \lim_{n \to +\infty} \frac{1}{n} \log \int_{\probmea{\limitset}} \! \myexp[\bigg]{n \int \! \psi \,\mathrm{d}\mu} \,\mathrm{d} \birkhoffmeasure(\mu) 
            &= \lim_{n \to +\infty} \frac{1}{n} \log \int_{\limitset} \! \myexp{ S_n^{F}\psi } \,\mathrm{d} \equstate \\
            &= P(F|_{\limitset}, \phi|_{\limitset} + \psi) - P(F|_{\limitset}, \phi|_{\limitset}).
        \end{split}
    \]
    Similarly, for the sequence $\sequen{\Omega_{n}(x_{n})}$, by \eqref{eq:prop:subsystem preimage pressure:recall} in Proposition~\ref{prop:subsystem preimage pressure:recall}, we have
    \[
        \begin{split}
            &\lim_{n \to +\infty} \frac{1}{n} \log \int_{\probmea{\limitset}} \! \myexp[\bigg]{n \int \! \psi \,\mathrm{d}\mu} \,\mathrm{d} \Omega_{n}(x_{n})(\mu) \\
            &\qquad = \lim_{n \to +\infty} \frac{1}{n} \log \sum_{y \in (F|_{\limitset})^{-n}(x_{n})} \frac{ \myexp{ S_{n}^{F}\potential(y) } }{\sum_{y' \in (F|_{\limitset})^{-n}(x_{n})} \myexp{ S_{n}^{F}\potential(y') } } e^{ \sum_{i = 1}^{n - 1} \psi\parentheses{F^{i}(y)} } \\
            &\qquad = \lim_{n \to +\infty} \frac{1}{n} \parentheses[\bigg]{ \log \sum_{y \in (F|_{\limitset})^{-n}(x_{n})} e^{ S_{n}^{F}(\phi + \widetilde{\psi})(y) } - \log \sum_{y' \in (F|_{\limitset})^{-n}(x_{n})} e^{ S_{n}^{F}\phi(y') } }   \\
            &\qquad = P(F|_{\limitset}, \phi|_{\limitset} + \psi) - P(F|_{\limitset}, \phi|_{\limitset}).
        \end{split}
    \]
    Therefore, all the assertions of Theorem~\ref{thm:level-2 larger deviation principles for subsystem} follow from Theorem~\ref{thm:large deviation principles Kifer Comman Juan}.
\end{proof}

We finally prove Corollary~\ref{coro:measure-theoretic pressure infimum on local basis:subsystem}, which gives a characterization of measure-theoretic pressure.

\begin{proof}[Proof of Corollary~\ref{coro:measure-theoretic pressure infimum on local basis:subsystem}]
    \def\localbasis{G_{\mu}}
    Fix $\mu \in \invmea[\limitset][F|_{\limitset}]$ and a convex local basis $\localbasis$ at $\mu$.
    We show that \eqref{eq:coro:measure-theoretic pressure infimum on local basis:subsystem} in Corollary~\ref{coro:measure-theoretic pressure infimum on local basis:subsystem} holds.
    By \eqref{eq:def:rate function:subsystem} and the upper semi-continuity of $h_{\mu}(F|_{\limitset})$ (\cite[Theorem~1.1]{shi2024entropy} and Lemma~\ref{lem:upper semi-continuity for submap}), we get
    \[
        - \ratefun(\mu) = \inf_{\mathcal{G} \in \localbasis} \sup_{\mathcal{G}} (- \ratefun) 
        = \inf_{\mathcal{G} \in \localbasis} \parentheses[\big]{- \inf_{\mathcal{G}} \ratefun}.
    \]
    Then it follows from \eqref{eq:def:rate function:subsystem} and \eqref{eq:equalities for rate function:subsystem} in Theorem~\ref{thm:level-2 larger deviation principles for subsystem} that 
    \[
        \begin{split}
            -\pressure + h_{\mu}(F|_{\limitset}) + \int \! \potential \,\mathrm{d}\mu 
            &= - \ratefun(\mu) = \inf_{\mathcal{G} \in \localbasis} \parentheses[\big]{- \inf_{\mathcal{G}} \ratefun} \\
            &= \inf_{\mathcal{G} \in \localbasis} \set[\bigg]{ \lim_{n \to +\infty} \frac{1}{n} \log \equstate (\set{x \in \limitset \describe \deltameasure[n]{x} \in \mathcal{G}}) } \\
            &= \inf_{\mathcal{G} \in \localbasis} \set[\bigg]{ \lim_{n \to +\infty} \frac{1}{n} \log \sum_{ y \in (F|_{\limitset})^{-n}(x_{n}), \deltameasure[n]{y} \in \mathcal{G} } \frac{ \myexp[\big]{S_{n}^{F}\potential(y)}}{Z_{n}(\potential)} },
        \end{split}
    \]
    where we write $Z_{n}(\potential) \define \sum_{ y \in (F|_{\limitset})^{-n}(x_{n}) } \myexp[\big]{ S_{n}^{F}\potential(y) }$.
    Note that by Propositions~\ref{prop:subsystem preimage pressure:recall} we have $\pressure = \lim_{n \to +\infty} \frac{1}{n} \log Z_{n}(\potential)$.
    Thus \eqref{eq:coro:measure-theoretic pressure infimum on local basis:subsystem} holds.
\end{proof}

\printbibliography

\end{document}